\documentclass[journa]{IEEEtran}
\usepackage{mathtools, cuted}
\usepackage{lipsum}
\usepackage{graphicx}
\usepackage{booktabs}                                                       
\usepackage{amssymb}
\usepackage{amsmath}
\usepackage{amsthm}
\usepackage{diagbox}
\usepackage{multirow}
\usepackage{lineno}
\usepackage{bm}
\usepackage{cite}
\usepackage{float}
\usepackage{subfig}
\usepackage[ruled,linesnumbered]{algorithm2e}
\usepackage{color}
\usepackage{ragged2e}
\usepackage{lscape}
\newtheorem{theorem}{Theorem}
\newtheorem{corollary}{Corollary}
\newtheorem{lemma}{Lemma}
\newtheorem{proposition}{Proposition}
\usepackage[justification=centering]{caption}
\newtheoremstyle{noparens}%
{}{}%
{\itshape}{}%
{\bfseries}{.}%
{ }%
{\thmname{#1}\itshape\thmnumber{ #2}\mdseries\thmnote{ #3}}

\theoremstyle{noparens}
\newtheorem{mylemma}[lemma]{Lemma}
\newtheorem{mytheorem}[theorem]{Theorem}

\newtheorem{myproposition}[proposition]{Proposition}
\begin{document}
\captionsetup[figure]{name={Fig.},labelsep=period,font=small,justification=justified} 
\captionsetup[table]{name={TABLE},labelsep=newline,font=small} 
\title{Direct Estimation of Eigenvalues of Large Dimensional Precision Matrix}

\author{Jie~Zhou, ~Junhao~Xie,~ and ~Jiaqi Chen 
\thanks{Jie~Zhou and ~Junhao~Xie are with Department of Electronic Engineering,
Harbin Institute of Technology, Harbin 150001, China (e-mail:22b905037@stu.hit.edu.cn; xj@hit.edu.cn). 
Jiaqi Chen is with School of Mathematics, Harbin Institute of Technology, Harbin 150001, China (e-mail:chenjq1016@hit.edu.cn). }
}

\maketitle

\begin{abstract}
 In this paper, we consider directly estimating the eigenvalues of precision matrix, without inverting the corresponding estimator for the eigenvalues of covariance matrix.  We focus on a general asymptotic regime, i.e., the large dimensional regime, where both the dimension $N$ and the sample size $K$ tend to infinity whereas their quotient $N/K$ converges to a positive constant.  By utilizing tools from random matrix theory, we construct an improved estimator for eigenvalues of precision matrix. We prove the consistency of the new estimator under large dimensional regime. In order to obtain the asymptotic bias term of the proposed estimator, we provide a theoretical result that characterizes the convergence rate of the expected Stieltjes transform (with its derivative) of the spectra of the sample covariance matrix. Using this result, we prove that the asymptotic bias term of the proposed estimator is of order $O(1/K^2)$.  Additionally, we establish a central limiting theorem (CLT) to describe the fluctuations of the new estimator. Finally, some numerical examples are presented to validate the excellent performance of the new estimator and to verify the accuracy of the CLT. 
\end{abstract}

\begin{IEEEkeywords}
	  Precision matrix, random matrix theory, Stieltjes transform, G-estimation, sample covariance matrix
\end{IEEEkeywords}

\IEEEpeerreviewmaketitle
\section{Introduction}
The precision matrix (i.e., the inverse of the covariance matrix) plays a crucial role in various fields, including statistics, finance, wireless communications, and radar signal processing \cite{PAUL2014Random, Ledoit2017Nonlinear,  Couillet2011Random,  Maio2015Modern}. For example, in finance, the Markowitz's optimal portfolio weight is defined in terms of precision matrix \cite{Markowitz1952Portfolio}, and in radar detection, the precision matrix is used to construct the test statistic of matched filter detector \cite{Robey1992A}.

In practical applications, however, due to the lack of a priori knowledge of the true precision matrix, one usually estimates it from limited observations. Traditionally, this is achieved by inverting an estimator of the covariance matrix. For instance, the sample covariance matrix (SCM), as the simplest but efficient estimator of the covariance matrix, is routinely inverted to obtain an estimate of precision matrix. In some signal processing applications \cite{Kelly1986An}, the inverse of SCM is often referred to as sample matrix inversion (SMI). The SMI is a biased but consistent estimator for the true precision matrix in the situation where the sample size $K$ is much larger than the dimension $N$. This scenario  has been extensively explored in the field of multivariate statistics over the past decades \cite{Muirhead1982Aspects}.

Unfortunately, in practice  we often encounter the scenario where the sample size $K$ is comparable in magnitude to the dimension $N$. Such a case is well known as the large dimensional (or high dimensional) regime in random matrix theory (RMT) \cite{Bai2009Spectral, Couillet2014Robust, Couillet2015The} and is of interest for current signal processing applications \cite{Mestre2008Modified,Couillet2016Second,Mestre2020On,Schenck2022Probability}. In large dimensional regime, it is often assumed that both the dimension $N$ and sample size $K$ tend to infinity while their quotient $N/K$ converges to a positive constant.   Under such a regime, the asymptotic behavior of the eigenvalues of SCM has attracted particular attention from statisticians \cite{Bai2009Spectral}. One of the best-known results is the Marcenko-Pastur law \cite{Marcenko1967Distribution} which describes the limiting spectral distribution (LSD) of the eigenvalues of SCM. Silverstein and Choi \cite{Silverstein1995Analysis} studied the analytic behavior of the LSD, offering several general properties pertaining to it. A noteworthy contribution by Bai and Silverstein \cite{Bai1998No} asserts that there will be no eigenvalues located outside the support of the LSD with probability one for all large $N$. They further advanced this understanding by providing an exact separation of eigenvalues \cite{Bai1999Exact}. 

Applying these fundamental results, Mestre \cite{Mestre2008On} demonstrated that a traditional estimator for eigenvalues of the covariance matrix becomes inconsistent in large dimensional regime. Indeed, many traditional estimators that are designed for large sample size case (i.e., $K$ is very large compared to $N$) always tend to perform poorly in large dimensional regime. To address this issue, statisticians have made a lot of efforts to develop improved estimators that maintain consistency in large dimensional regime.  Mestre, in his seminal work \cite{Mestre2008Improved}, pioneered a complex integration approach to construct a consistent estimator for eigenvalues of a large dimensional covariance matrix.  Subsequently, Yao \textit{et al.} \cite{Yao2013Fluctuations} provided the asymptotic fluctuations of Mestre's improved estimator. 
Mestre's methodology offers a standard process for constructing consistent estimators in large dimensional regime. Specifically, his approach comprises three  pivotal steps: (i) relating the LSD of the eigenvalues of SCM with the distribution of eigenvalues of true covariance matrix by means of their respective Stieljes transforms, then (ii) expressing the quantity to be estimated as a complex integral involving Stieltjes transform and finally (iii) evaluating this complex integral through complex analysis techniques.  Recently, Mestre's method has been fruitfully applied in various signal processing applications, such as direction of arrival estimation and correlation tests (see for instance, \cite{Mestre2008Modified, Vallet2015Performance, Mestre2017Correlation, Mestre2020On, Schenck2021Probability}).

In addition to Mestre's integration approach, numerous other methodologies have been developed to enhance the covariance matrix estimation in large dimensional regime, including linear shrinkage approach \cite{Ledoit2004A}, nonlinear shrinkage approach \cite{Ledoit2012Nonlinear, Ledoit2015Spectrum, Ledoit2020Analytical,Ledoit2021Shrinkage,BUN2017Cleaning} and the moment method \cite{Yao2012Eigenvalue}.  Ledoit and Wolf \cite{Ledoit2004A} pioneered the linear shrinkage approach, constructing a well-conditioned estimator for large covariance matrices by linearly shrinking the eigenvalues of SCM. They derived explicit expressions for the optimal shrinking weights by minimizing a quadratic loss function, and further provided  consistent estimators for these optimal weights. Beyond the linear shrinkage method, Ledoit and Wolf have made substantial contributions to the field of nonlinear shrinkage estimation \cite{Ledoit2012Nonlinear, Ledoit2015Spectrum, Ledoit2020Analytical}. In their earlier works \cite{Ledoit2012Nonlinear} and \cite{Ledoit2015Spectrum}, they introduced two indirect nonlinear shrinkage estimators, both of which heavily rely on a multivariate function known as the the Quantized Eigenvalues Sampling Transform (QuEST) function. However, the QuEST function needs numerical methods to solve, thereby leading to complexities in its implementation. To overcome this limitation, Ledoit and Wolf established an analytical nonlinear shrinkage estimation in  \cite{Ledoit2020Analytical}, where the optimal nonlinear shrinkage function is expressed as an explicit formula involving the Hilbert transform of the LSD of SCM. Further, in \cite{Ledoit2021Shrinkage}, they proposed a range of nonlinear shrinkage methodologies tailored to diverse loss functions. Regrettably, all these estimators still depend on the QuEST function to obtain the estimates of the population eigenvalues.  Parallel to these works, J. Bun \textit{et al.} \cite{BUN2017Cleaning} developed a nonlinear shrinkage approach using Bayes theory. Aside from these shrinking methods, the moment method was exploited by Yao \textit{et al.} \cite{Yao2012Eigenvalue} to estimate the eigenvalues of large covariance matrix. While sharing similarities with Mestre's methodology, the moment method differs primarily in step (ii). Specifically, in moment method, the quantities to be estimated are the moments of the spectra of the population covariance matrix, which can be formulated as complex integrals in terms of the LSD of SCM. Once these moment estimates are obtained, a system of equations is constructed to estimate the eigenvalues along with their respective multiplicities. It is worth mentioning that the moment method relaxes the spectral separation assumption, which is a prerequisite for Mestre's approach.

In contrast to the large SCM that has garnered extensive exploration in RMT, the asymptotic behavior of large SMI has received comparatively less attention. However, a more generalized matrix model, known as the Fisher matrix or F-matrix, has been widely studied in multivariate statistical analysis. The F-matrix is defined by $\mathbf{F} \triangleq \mathbf{B}_1\mathbf{B}_2^{-1}$ with $\mathbf{B}_1$ and $\mathbf{B}_2$ being two SCM. When taking $\mathbf{B}_1$ as an identity matrix, the F-matrix naturally degenerates into the SMI. The eigenvalues of F-matrix are usually used to construct a metric to measure the distance between two covariance matrices \cite{Couillet2019Random, Pereira2024Asymptotics}. The limiting behavior of the eigenvalues of a large dimensional F-matrix has been analyzed by some statisticians, see e.g., Yin, Bai and Krishnaiah \cite{Yin1983Limiting} and Silverstein \cite{Silverstein1985The}. Recently, Zheng, Bai and Yao \cite{Zheng2017CLT} provided the limiting distribution of eigenvalues of a large dimensional general F-matrix in terms of Stieltjes transform. We emphasize that the results in Zheng, Bai and Yao \cite{Zheng2017CLT} allow us to find the limiting behavior of eigenvalues of a large dimensional SMI, which will be stated in Theorem \ref{Theorem Weak Convergence of u_N} in Section II-B. 

Indeed, over the past decades, the direct estimation of precision matrix has always been one of the most popular research topics in multivariate statistics. For instance, Zhang \textit{et al.} \cite{Zhang2013Improved} derived an estimator for the precision matrix by utilizing weighted sampling and linear shrinkage. Bodnar \textit{et al.} \cite{Bodnar2016Direct} constructed the optimal shrinkage estimator for precision matrix which is proved to possess almost surely the smallest Frobenius loss under large dimensional regime. Furthermore, Tiomoko \textit{et al.} \cite{Tiomoko2020Random} proposed a series of improved precision matrix estimation under a wide family of metrics that characterize the covariance distances.

In this paper, we shift the focus from estimating the precision matrix itself to estimating its eigenvalues. This is because, in many real applications, such as principal component analysis \cite{Anderson1963Asymptotic} and source detection \cite{Bianchi2011Performance}, eigenvalues hold greater significance.
The limiting behavior of eigenvalues of SMI enables us to  directly estimate the eigenvalues of precision matrix, without inverting the estimators for eigenvalues of covariance matrix. We emphasize that the main contributions of this paper lie in two aspects. Firstly, we construct an improved estimator for eigenvalues of precision matrix. We prove the consistency of the proposed estimator under large dimensional regime. Further, we provide a CLT to characterize the asymptotic fluctuation of the proposed estimator. As we will see in Section V, although the consistency of the proposed estimator is established on a spectral separation assumption, it still outperforms competitors in the case where the separability condition is not met. Furthermore, the proposed estimator offers the additional advantage of high computational efficiency, because it is simply a combination of sample estimates.  Secondly, as a preliminary step, we establish a theoretical result that describes the convergence rate of the expected Stieltjes transform (with its derivative) of the spectra of SCM (see Theorem \ref{Theorem Convergence Rate of Expected Stieltjes Transform} in Section III-A). This result will play an important role in precisely controlling the bias term of the proposed estimator. 

The remainder of this paper is organized as follows. Section II introduces the notations in this paper and presents some standard assumptions and useful results that will be used in the following sections. Section III presents the main results of this paper. Section IV provides the proofs of some results presented in Section III.  Some numerical examples are provided in Section V to verify the superiority of the proposed estimator and to validate the theoretical results. Section VI concludes the paper. Additionally, due to the space constraints, the proofs of some theoretical results are deferred to Supplementary Material.

\section{Notations, Assumptions and Useful Results}
\subsection{Notations and Assumptions}
\subsubsection{Notations}Throughout the paper and Supplementary Material, we use lowercase, boldface lowercase and boldface uppercase to stand for scalars, vectors and matrices, respectively. For matrices, a subscript will be added to emphasize dimension, though it will be occasionally dropped for the sake of clarity of presentation. $\mathbf{I}_N$ denotes an $N\times N$ identity matrix.  For a matrix $\mathbf{X}$, its $(i,j)$th entry is denoted by $\mathbf{X}_{ij}$ and its spectral norm is represented by $\Vert\mathbf{X}\Vert$. The superscripts $(\cdot)^T$ and $(\cdot)^H$ stands for transpose and conjugate transpose, respectively.  $\mathbb{R}$ and $\mathbb{C}$ respectively denote the real and complex fields with dimension specified by superscripts. ${\rm tr}(\cdot)$, $E(\cdot)$ and ${\rm Var}(\cdot)$ denote trace, expectation and variance,  respectively. For a complex number $z$, its real and imaginary parts are respectively denoted by ${\rm Re}(z)$ and ${\rm Im}(z)$; $\bar{z}$ denotes the conjugate of $z$; $\mathsf{i}$ represents the imaginary unit. The convergence in distribution will be denoted by $\overset{d}{\longrightarrow}$, and almost sure convergence by  $\overset{a.s.}{\longrightarrow}$. For a continuous differentiable function $f$, its derivative is
denoted by $f'$. For any two functions $f$ and $g$, $f=O(g)$ will denote $|f|< k|g|$ for a positive constant $k$. $P_m$  denotes a generic deterministic polynomial of degree $m$, whose coefficients are nonnegative and independent of the sample size $K$.

\subsubsection{Assumptions}
We consider $K$ independent observations $\mathbf{y}_1,\mathbf{y}_2,...,\mathbf{y}_K$ with $\mathbf{y}_i\in\mathbb{C}^N$ ($i=1,2,...,K$) having zero mean and covariance matrix $\mathbf{R}_N\in\mathbb{C}^{N\times N}$.  

The inverse of $\mathbf{R}_N$, denoted by $\mathbf{M}_N$, is called precision matrix. Regarding $\mathbf{R}_N$ and $\mathbf{M}_N$, we make the following assumption.
\begin{itemize}
	\item (A1) $\mathbf{R}_N$ is an Hermitian matrix with $L$ ($L$ being fixed and known) distinct eigenvalues $\lambda_1<\lambda_2<\cdots<\lambda_L$ with respective multiplicities $N_1,N_2,...,N_L$ such that $\sum_{i=1}^LN_i=N$. In addition, there exists $\lambda_{\rm min}$ and $\lambda_{\rm max}$, such that $0<\lambda_{\rm min}<\inf\limits_{N}(\lambda_1)\leqslant\sup\limits_{N}(\lambda_L)<\lambda_{\rm max}<\infty$ for all large $N$.
\end{itemize}

\textit{Remark 1:} Assumption (A1) inherently implies that $\mathbf{M}_N$ also has $L$ distinct eigenvalues $\gamma_1>\gamma_2>\cdots>\gamma_L$. Notice that, for convenience, here we arrange the eigenvalues of $\mathbf{M}_N$ in ascending order. Given this notation, it is apparent that $\gamma_i=1/\lambda_i$, and $\lambda_i$ and $\gamma_i$ share the same multiplicity $N_i$ for all $i=1,2,...,L$. 

Regarding the dimension $N$, the sample size $K$, and the multiplicities $N_i$ ($1\leqslant i \leqslant L$), we make the following assumption:
\begin{itemize}
	\item (A2) As $N,K\to \infty$, $c_K\triangleq N/K\to c\in(0,1)$,  $N_i/K\to c_i \in(0,1)$ for all $1\leqslant i \leqslant L$.
\end{itemize}

All the observations $\mathbf{y}_1,\mathbf{y}_2,...,\mathbf{y}_K$ compose an observation matrix, denoted by $\mathbf{Y} = \frac{1}{\sqrt{K}} [\mathbf{y}_1,\mathbf{y}_2,...,\mathbf{y}_K]$, and then the SCM is defined by
\begin{equation}\label{key} 
	 \hat{\mathbf{R}}_N = \mathbf{Y}\mathbf{Y}^H.
\end{equation}

We further assume that
\begin{itemize}
	\item (A3) $\mathbf{Y} = \mathbf{R}_N^{1/2}\mathbf{X}$ with $\mathbf{X}\in \mathbb{C}^{N\times K}$ being composed of  independent and identically distributed (i.i.d.) complex Gaussian random variables with zero mean and unit variance.
\end{itemize}

 Under the assumptions (A1)-(A3), $\hat{\mathbf{R}}_N$ has an inverse, denoted by $\hat{\mathbf{M}}_N = \hat{\mathbf{R}}_N^{-1}$. We call $\hat{\mathbf{M}}_N$ the sample matrix inversion (SMI). We denote the eigenvalues of $\hat{\mathbf{R}}_N$ by $\hat{\sigma}_1<\hat{\sigma}_2<\cdots<\hat{\sigma}_N$ and the eigenvalues of $\hat{\mathbf{M}}_N$ by $\hat{\rho}_1>\hat{\rho}_2>\cdots>\hat{\rho}_N$. Given this notation, it is obvious that $\hat{\rho}_i = 1/\hat{\sigma}_i$ for all $i=1,2,...,N$.  The goal of this paper is to retrieve $\gamma_1,...,\gamma_L$ by using $\hat{\rho}_1,...,\hat{\rho}_N$.
 
 We emphasize that the Gaussianity assumption in (A3) is fairly standard in many adaptive signal processing literature \cite{ Rubio2012A, Schenck2021Probability, Orlando2022A,Dumont2010On}. In addition to these three standard assumptions widely used in RMT, we will further assume a spectral separation condition, which will be stated in the following subsection.

\subsection{Useful Results in RMT}
\subsubsection{The Limiting Distribution of Eigenvalues of $\hat{\mathbf{R}}_N$ and $\hat{\mathbf{M}}_N$} Denote the empirical distributions of eigenvalues of $\hat{\mathbf{R}}_N$ and $\hat{\mathbf{M}}_N$ respectively by $F_{\hat{\mathbf{R}}}(x)$ and $F_{\hat{\mathbf{M}}}(x)$. In order to understand the limiting behavior of eigenvalues of $\hat{\mathbf{R}}_N$ (or $\hat{\mathbf{M}}_N$), it is of vital importance to explore the limiting behavior of $F_{\hat{\mathbf{R}}}(x)$ (or $F_{\hat{\mathbf{M}}}(x)$) in large dimensional regime. 

The Stieltjes transform is a strong tool used in RMT to describe the limiting spectra of a large random matrix. Now let us briefly recall the definition of Stieltjes transform. Let $F$ denote a positive measure defined on $\mathbb{R}$. The Stieltjes transform of $F$, denoted by $b(z)$, is defined by
\begin{equation}\label{ST}
	b(z) = \int_{\mathbb{R}}\frac{dF(\lambda)}{\lambda-z}, \quad\forall z\in\mathbb{C}\setminus{\rm supp}(F)
\end{equation}
where ${\rm supp}(F)$ denotes the support of $F$. The measure $F$ can be uniquely retrieved from $b(z)$ by the following inversion formula:
\begin{equation}\label{Inversion formula}
	F(\lambda) = \frac{1}{\pi}\lim_{y\to0^+}\int_{-\infty}^\lambda{\rm Im}\left(b(x+\mathsf{i}y)\right)dx.
\end{equation}

We can see from \eqref{Inversion formula} that a probability measure is closely related to its Stieltjes transform. This close relationship enables us to study the asymptotic behavior of the measure by exploring the limit behavior of the corresponding Stieltjes transform.

Using \eqref{ST}, we can write the Stieltjes transform of $F_{\hat{\mathbf{R}}}(x)$ by
\begin{equation}\label{key}
	b_{\hat{\mathbf{R}}}(z) = \frac{1}{N}\sum_{i=1}^N\frac{1}{\hat{\sigma}_i-z} 
\end{equation}
for all $z\in\mathbb{C}\setminus\mathbb{R}$. 

Likewise, the Stieltjes transform of $F_{\hat{\mathbf{M}}}(x)$ is given by
\begin{equation}\label{key}
	b_{\hat{\mathbf{M}}}(z) = \frac{1}{N}\sum_{i=1}^N\frac{1}{\hat{\rho}_i-z} 
\end{equation}
for all $z\in\mathbb{C}\setminus\mathbb{R}$.

From the relation $\hat{\sigma}_i = 1/\hat{\rho}_i$ for all $i=1,2,..,N$, it is easy to check that the following relation hold true
\begin{equation}\label{relation between bR and bM}
	b_{\hat{\mathbf{R}}}\left(\frac{1}{z}\right) = -z-z^2b_{\hat{\mathbf{M}}}(z).
\end{equation}

Now we give the  following two theorems  that respectively describe the limiting behavior of $F_{\hat{\mathbf{R}}}$ and $F_{\hat{\mathbf{M}}}$ in terms of their Stieltjes transforms.
\begin{mytheorem}\label{Theorem MP equation}
	Let assumptions (A1)-(A3) hold true. Then $F_{\hat{\mathbf{R}}}$ almost surely converges to a deterministic measure ${F}_K$ with support $\mathcal{S}\subset\mathbb{R}$. Denote the Stieltjes transform of ${F}_K(x)$ by ${b}_K(z)$. Then it holds true that
	\begin{equation}\label{MP equation}
		\left|b_{\hat{\mathbf{R}}}(z)-{b}_K(z)\right|\overset{a.s.}{\longrightarrow}0
	\end{equation}
	for all $z\in \mathbb{C}\setminus \mathbb{R}$, where $b_K(z) = b$ is the unique solution to the equation
	\begin{equation}\label{9}
		{b} = \frac{1}{N}\sum_{i=1}^L{N}_i\frac{1}{\lambda_i(1-c-cz{b})-z}.
	\end{equation}
\end{mytheorem}
\begin{proof}
	The result \eqref{MP equation}  is very famous and has been established under different assumptions in RMT literature (see for instance \cite{Silverstein1995Strong, GIRKO1996Strong}). This theorem has been used in many literature for different purposes (see for instance \cite{Mestre2008Improved, Mestre2008On}).
\end{proof}

\begin{mytheorem}\label{Theorem Weak Convergence of u_N}
	Let assumptions (A1)-(A3) hold true. There exists a deterministic measure $\tilde{F}_K$ with support $\tilde{\mathcal{S}} \subset \mathbb{R}$ such that $\left|{F}_{\hat{\mathbf{M}}}(x) -\tilde{F}_K(x)\right|\overset{a.s.}{\longrightarrow } 0$ weakly. Denote the Stieltjes transform of $\tilde{F}_K$ by $\tilde{b}_K(z)$. Then it holds that
	\begin{equation}\label{key}
		|b_{\hat{\mathbf{M}}}(z) - \tilde{b}_{K}(z)|\overset{a.s.}{\longrightarrow } 0
	\end{equation}
	for all $z\in \mathbb{C}\setminus \mathbb{R}$, where $\tilde{b}_K(z) = \tilde b$ is the unique solution to the  equation
	\begin{equation}\label{11}
		{\tilde b} = \frac{1}{N}\sum_{i=1}^LN_i\frac{1+zc{\tilde b}}{\gamma_i - z(1+zc{\tilde b})}.
	\end{equation}	
\end{mytheorem}

\begin{proof}
We note that this theorem has not been explicitly stated as an independent entity in the existing literature. Actually, this theorem is a direct corollary of \cite[Theorem 2.1]{Zheng2017CLT}, as the precision matrix $\mathbf{M}_N$ can be viewed as a particular instance of an F-matrix $\mathbf{F} = \mathbf{B}\mathbf{M}_N$ with  $\mathbf{B}$ being an identity matrix. In addition, it is worth mentioning that this theorem can be directly derived from Theorem \ref{Theorem MP equation} by utilizing the relationship between the Stieltjes transforms of the spectral densities of $\hat{\mathbf{R}}_N$ and $\hat{\mathbf{M}}_N$ (as shown in \eqref{relation between bR and bM}). 
\end{proof}

Next, we will present some relations that will be used in what follows. We first introduce the dual representation of $\hat{\mathbf{R}} $, defined by $\underline{\hat{\mathbf{R}}}= \mathbf{Y}^H\mathbf{Y}$, which plays an important role in RMT. We denote the Stieltjes transform of the empirical spectral distribution of $\underline{\hat{\mathbf{R}}}$ by $b_{\underline{\hat{\mathbf{R}}}}(z)$, which has the following relation with $b_{\hat{\mathbf{R}}}(z)$:
\begin{equation}\label{relation between bRunder and bR}
	b_{\underline{\hat{\mathbf{R}}}}(z) = c_Kb_{{\hat{\mathbf{R}}}}(z) -\left(1-c_K\right)\frac{1}{z}.
\end{equation}

From Theorem \ref{Theorem MP equation}, we easily check that $b_{\underline{\hat{\mathbf{R}}}}(z)$ almost surely converges to a Stieltjes transform, denoted by $\underline{b}_K(z)$, which has a relation with ${b}_K(z)$:
\begin{equation}\label{key}
	\underline{b}_K(z) = cb_K(z) -\left(1-c\right)\frac{1}{z}.
\end{equation}

From \eqref{relation between bR and bM} and \eqref{relation between bRunder and bR}, it is easy to obtain the following relation:
\begin{equation}\label{relation 1}
	b_{\underline{\hat{\mathbf{R}}}}\left(\frac{1}{z}\right) = -z\left(1+c_Kzb_{\hat{\mathbf{M}}}(z)\right).
\end{equation}
\begin{figure*}[ht]
	\centering
	\subfloat[]{
		\centering
		\includegraphics[trim=0.5cm 0 0 2cm, scale=0.42]{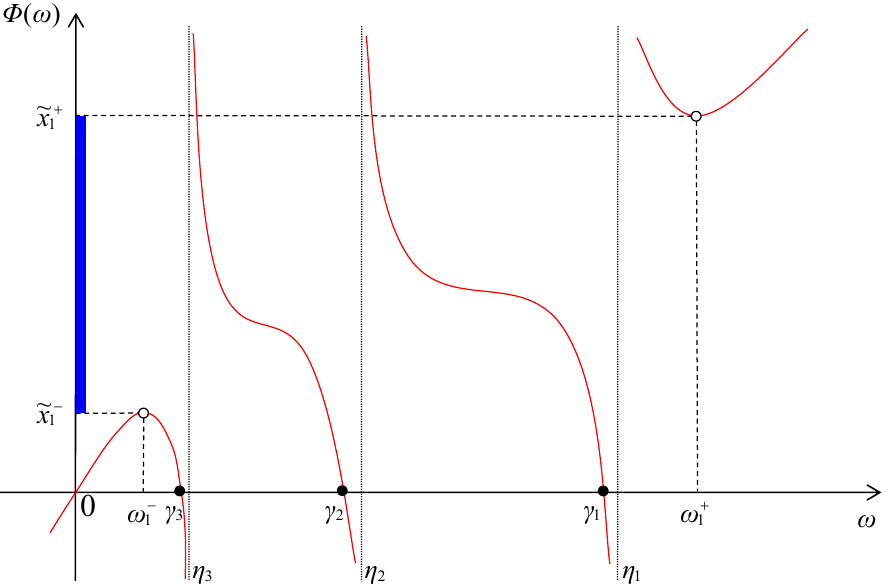}
	}
	\subfloat[]{
		\centering
		\includegraphics[trim=0.5cm 0 0 2cm, scale=0.42]{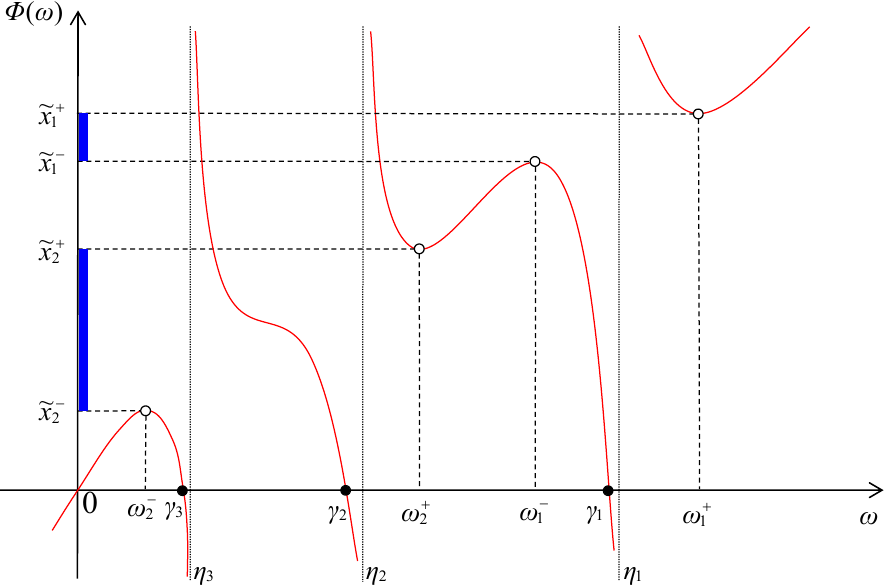}
	}
	\subfloat[]{
		\centering
		\includegraphics[trim=0.5cm 0 0 2cm, scale=0.42]{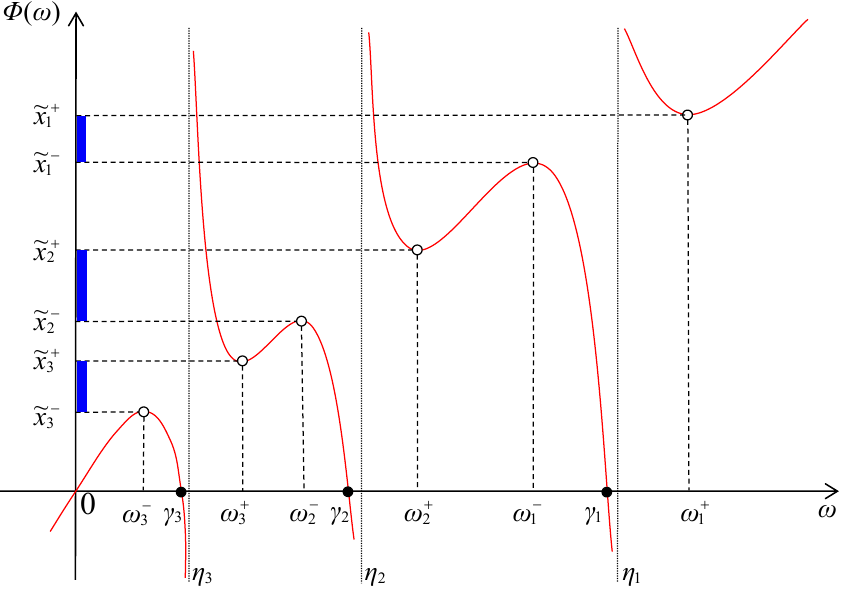}
	}\\
	\caption{When $L=3$, typical representations of $\Phi(\omega)$ as a function of $\omega$ for (a) $Q=1$, (b) $Q=2$ and (c) $Q=3$.}
	\label{typical representations of Phi}
\end{figure*}

From Theorem \ref{Theorem MP equation} and Theorem \ref{Theorem Weak Convergence of u_N}, we know that $b_{\underline{\hat{\mathbf{R}}}}\left(\frac{1}{z}\right)$ and $b_{\hat{\mathbf{M}}}(z) $ almost surely converge to $\underline{{b}}_K\left(\frac{1}{z}\right)$ and $\tilde{b}_K(z)$, respectively. Using these two convergences in \eqref{relation 1}, we obtain the following relation
\begin{equation}\label{relation 2}
	\underline{{b}}_K\left(\frac{1}{z}\right) = -z\left(1+cz\tilde{b}_K(z)\right).
\end{equation}

We emphasize that the relations \eqref{relation 1} and \eqref{relation 2} are very useful in the following derivations.

\subsubsection{The Supports of $F_K$ and $\tilde{F}_K$}
In order to obtain the limiting spectral measures $F_K$ and $\tilde{F}_K$, we need to solve \eqref{9} and \eqref{11} and then use the Stieltjes  inversion formula. Unfortunately, both of them  are explicitly solvable only in the case $\mathbf{R}_N=\sigma^2\mathbf{I}_N$. In other cases, we are unable to get the explicit expression for $F_K$ or $\tilde{F}_K$.

Fortunately, for the majority of problems that we are concerned with, it suffices to leverage certain properties of $F_K$ and $\tilde{F}_K$ rather than seeking their explicit representations. The supports of $F_K$ and $\tilde{F}_K$ (i.e., $\mathcal{S}$ and $\tilde{\mathcal{S}}$ respectively) are useful attributes for comprehending the behavior of $F_K$ and $\tilde{F}_K$. 

In the seminal paper \cite{Mestre2008On}, Mestre provided an exactly characterization of the support of $F_K$. He claimed that the support of $F_K$ can be divided into  $Q$ disjoint compacts, i.e., $\mathcal{S} = \bigcup_{k=1}^Q\mathcal{S}_k$ ( $1\leqslant Q\leqslant L$),  when a spectral separation condition holds true. In this paper, we consider the case $Q=L$. In this case, the spectral separation condition could be formulated as the following assumption:
\begin{itemize}
	\item (A4) ${\inf_{1\leqslant i\leqslant L-1}}\Psi(\bar{f}_i)>0$ where the function $\Psi(f)$ is define by 
	\begin{equation}\label{Psi(f)}
		\Psi(f) = 1-\frac{c}{N}\sum_{i=1}^LN_i\left(\frac{\lambda_i}{\lambda_i-f}\right)^2
	\end{equation}
	and $\bar{f}_i$ are the real-valued solutions to the equation $\Psi'(f) = 0$ at $f$ under convention $0<\bar{f}_1<\bar{f}_2<\cdots<\bar{f}_{L-1}$.
\end{itemize}

If assumption (A4) is satisfied, it holds true that $\mathcal{S} = \bigcup_{k=1}^L\mathcal{S}_k$ with $\mathcal{S}_k = (x_k^-,x_k^+)$ being associated with $\lambda_k$. For more details about $\mathcal{S}$, we refer the readers to Mestre's seminal papers \cite{Mestre2008On} and \cite{Mestre2008Improved}.

Next we will provide an explicit characterization of the support of $\tilde{F}_K$. We consider a map $z\mapsto \omega=z(1+cz{\tilde{b}_K(z)})$. Then using this map in \eqref{11}, one may find that \eqref{11} is equivalent to the equation
\begin{equation}\label{key}
	z = \Phi(\omega)
\end{equation}
where 
\begin{equation}\label{key}
	\Phi(\omega) = \frac{\omega}{\chi(\omega)}
\end{equation}
and 
\begin{equation}\label{key}
	\chi(\omega) = 1+\omega cb_{\mathbf{M}}(\omega)
\end{equation}
with $b_{\mathbf{M}}=\frac{1}{N}\sum_{i=1}^LN_i\frac{1}{\gamma_i-\omega}$ being the Stieltjes transform of the spectral density of $\mathbf{M}_N$.

In Fig. \ref{typical representations of Phi}, we depict three possible representations of $\Phi(\omega)$ as a function of $\omega$ at $L=3$. With reference to Fig. \ref{typical representations of Phi} and after a straightforward analysis, we may find that $\Phi(\omega)$ has the following properties.
\begin{itemize}
	\item \textit{Property 1:} $\Phi(w)$ has $L+1$ zeros: $\{\gamma_1,\gamma_2,\cdots,\gamma_L,0\}$.
	\item \textit{Property 2:} $\Phi(w)$ has $L$ poles, denoted by  $\eta_1>\eta_2>\cdots>\eta_L$, which are zeros of $\chi(\omega)$, i.e., $\chi(\eta_i)=0$ for $1\leqslant i\leqslant L$. 
	\item \textit{Property 3:} $\Phi(\omega)$ always has $2Q$ ($1\leqslant Q\leqslant L$) positive local extrema. If we denote their preimages by $\omega_i^-$, $\omega_i^{+}$, $1\leqslant i\leqslant Q$, which satisfy 
	$\omega_1^+>\omega_1^->\omega_2^+>\omega_2^->\cdots>\omega_Q^->\omega_Q^+>0$,  then $\omega_i^{\pm}$ are the real-valued solutions of equation $\xi(\omega)=0$ at $\omega$, where $\xi(\omega)$ is given by
	\begin{equation}\label{xi(w)}
		\begin{aligned}
			\xi(\omega) = 1-\frac{c}{N}\sum_{i=1}^LN_i\left(\frac{\omega}{\gamma_i-\omega}\right)^2.
		\end{aligned}
	\end{equation}
	\item \textit{Property 4:} Denote the interval $\mathcal{W}_i=(\omega_i^-,\omega_i^+)$, and denote $\mathcal{W} = \cup_{i=1}^Q\mathcal{W}_i$. Then the function $\Phi(\omega)$ is strictly increasing on $\mathbb{R}\setminus\mathcal{W}$, which leads to
	\begin{equation}\label{key}
		\tilde{x}_1^+>\tilde{x}_1^->\cdots>\tilde{x}_Q^+>\tilde{x}_Q^->0
	\end{equation}
	where $\tilde{x}_{i}^{\pm}=\Phi(\omega_i^{\pm})$ for all $i=1,2,...,Q$.
\end{itemize}

 From Fig. \ref{typical representations of Phi}, we observe that each eigenvalue $\gamma_k$ ($1\leqslant k\leqslant L$)  always locates at one, and only one of the intervals $\mathcal{W}_i$ for $1\leqslant i\leqslant Q$. However, it should be noted that this is not a one-to-one correspondence, since multiple consecutive eigenvalues of $\mathbf{M}_N$ may appear in the same interval  $\mathcal{W}_i$, as shown in Fig.\ref{typical representations of Phi}(a) and (b).

In this paper, we consider the situation of $Q=L$ (see Fig. \ref{typical representations of Phi}(c)), that is, each distinct eigenvalue $\gamma_i$ ($1\leqslant i\leqslant L$) of $\mathbf{M}_N$ locates in different interval $\mathcal{W}_i$ and hence the support of $\tilde{F}_K$, i.e., $\tilde{\mathcal{S}}$, will be divided into $L$ disjoint compacts. Now a question arises: how to guarantee that $\tilde{\mathcal{S}}$ is divided into $L$ disjoint compacts? From Fig. \ref{typical representations of Phi}(c), we can see that if the compact $\mathcal{W}$ is separated into $L$ different intervals, it is equivalent to ensure that each $\omega_i^{\pm}$ ($1\leqslant i \leqslant L$) is the local extreme of $\Phi(\omega)$. In other words, $\omega_i^{\pm}$ for $1\leqslant i \leqslant L$ are the zeros of $\xi(\omega)$.  In Fig. \ref{The function of derivative of Phi}, we depict the function $\xi(\omega)$. With reference to Fig. \ref{The function of derivative of Phi}, we formulate the exact spectral separation condition for $\tilde{F}_K$ in mathematical term:   
\begin{itemize}
	\item $\inf_{1\leqslant i\leqslant L-1}\xi(\bar{\omega}_i)>0$ where $\bar{\omega}_i$ for $1\leqslant i\leqslant L-1$ are the real-valued solutions to the equation $\xi'(\omega)=0$ at $\omega$ under convention $\bar{\omega}_1>\bar{\omega}_2>\cdots>\bar{\omega}_{L-1}$.
\end{itemize}

Observing the function $\Psi(\cdot)$ defined in \eqref{Psi(f)} and $\xi(\cdot)$ in \eqref{xi(w)}, and noticing the relation $\gamma_i = 1/\lambda_i$, it is easy to check that $\Psi(f) = \xi(1/f)$ and $\Psi'(f) =-\xi'(1/f)/f^2$. This implies $\bar{f}_i = 1/\bar{\omega}_i$ for $1\leqslant i\leqslant L-1$. Hence, we can conclude that the spectral separation condition for $\hat{\mathbf{M}}_N$ is equivalent to that of $\hat{\mathbf{R}}_N$.  In other words, the assumption (A4) also ensures the spectra separability of $\hat{\mathbf{M}}_N$. 

Now it is time to give an explicit characterization of the support of $\tilde{F}_K$: when assumption (A4) holds true, the support of $\tilde{F}_K$ is given by $\tilde{\mathcal{S}} =\cup_{k=1}^L\tilde{\mathcal{S}}_k $ with $\tilde{\mathcal{S}}_k  =(\tilde{x}_k^-,\tilde{x}_k^+) $ only being associated with $\gamma_k$. Furthermore, it is not difficult to find that $\tilde{\mathcal{S}}=\mathcal{S}^{-1}$, that is, for intervals $\tilde{\mathcal{S}}_k = (\tilde{x}_k^-,\tilde{x}_k^+)$ and $\mathcal{S}_k = (x_k^-,x_k^+)$, it holds true that $\tilde{x}_k^- = 1/x_k^+$ and $\tilde{x}_k^+ =1/x_k^-$.

\begin{figure}[]
	\centering
	\includegraphics[scale=0.5]{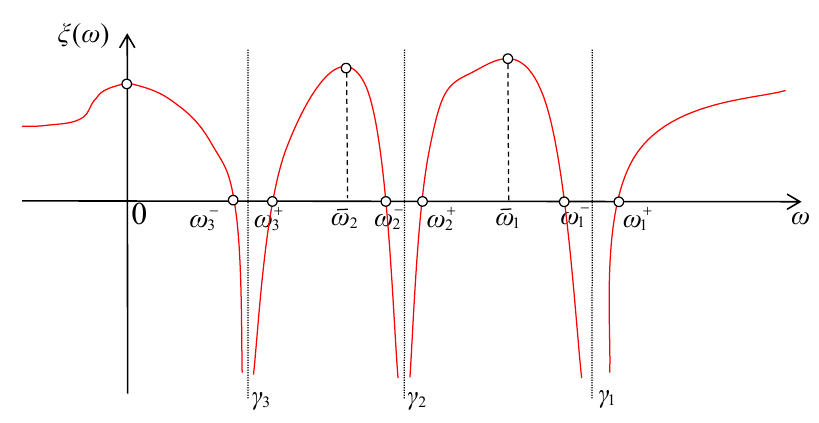}
	\caption{The function $\xi(\omega)$ at $Q=L=3$.}
	\label{The function of derivative of Phi}\vspace{-1em}
\end{figure}

\section{Main Results}
In this section, we present the main results of this paper. The main objective of this paper is to derive a consistent estimator for eigenvalues of $\mathbf{M}_N$. This improved estimator with its asymptotic properties are presented in Section III-B.   Before introducing this estimator, we first establish a preliminary theorem in Section III-A, which describes the convergence rate of the expected Stieltjes transform (with its derivative) of the spectra of $\hat{\mathbf{R}}_N$. We emphasize that this theorem serves as the cornerstone for accurately controlling the bias term of the proposed estimator.

\subsection{The Convergence Rates of Expected ${{b}_{\hat{\mathbf{R}}}}(z)$ and $b'_{\hat{\mathbf{R}}}(z)$}
As a preparatory step, we establish the following theorem to characterize the convergence rates of the expected ${{b}_{\hat{\mathbf{R}}}}(z)$ and ${{b}'_{\hat{\mathbf{R}}}}(z)$.  
\begin{mytheorem}\label{Theorem Convergence Rate of Expected Stieltjes Transform}
Under assumptions (A1)-(A3), it holds true that 
	\begin{align}
		&|Eb_{{\hat{\mathbf{R}}}}(z)-b_{K}(z)|\leqslant\frac{1}{K^2}P_9(|z|)P_{12}(|{\rm Im}z|^{-1})\label{Theorem Eb - b},
		\\&|Eb_{\hat{\mathbf{R}}}'(z) - b'_{K}(z)|\leqslant \frac{1}{K^2}P_{15}(|z|)P_{20}(|{\rm Im}z|^{-1})\label{Theorem Eb' - b'},
	\end{align}
for $z\in\mathbb{C}\setminus\mathbb{R}$ and all large $K$.
\end{mytheorem}
\begin{proof}
	See Appendix A in Supplementary Material.
\end{proof}
\textit{Remark 2}: We emphasize that \eqref{Theorem Eb - b} is a strengthen version of the result in \cite{Bai2004CLT}, which stated that $Eb_{{\hat{\mathbf{R}}}}(z)-{b}_{K}(z)$ is only an $O(K^{-1})$ term. 

\textit{Remark 3}: We note that \eqref{Theorem Eb - b} bears a resemblance to \cite[Lemma 8]{Vallet2012Improved} (see also \cite[Proposition 4]{Vallet2010ImprovedOnLine}), yet a crucial distinction lies in the differing assumptions made regarding the observation matrix $\mathbf{Y}$. In \cite{Vallet2012Improved}, $\mathbf{Y}$ is modeled as $\mathbf{B} + \mathbf{W}$, where $\mathbf{B}$ is a deterministic matrix containing the signals contribution and $\mathbf{W}$ is a complex Gaussian white noise matrix with i.i.d. elements having zero mean and variance $\sigma^2/K$. Conversely, we here assume that $\mathbf{Y} = \mathbf{R}_N^{1/2}\mathbf{X}$ with $\mathbf{R}_N$ being an arbitrary Hermitian matrix and $\mathbf{X}$ being a complex Gaussian matrix with i.i.d. elements having zero mean and variance $1/K$. Notably, this latter assumption is prevalent in Radar signal processing applications, such as adaptive detection \cite{Kelly1986An, Robey1992A, Orlando2022A} and filtering \cite{Rubio2012A, Ginolhac2013Performance}, where the target signal is usually surrounded by colored noise/clutter, and the noise/clutter covariance matrix is typically estimated using signal-free data (often referred to as the secondary data).

From Theorem \ref{Theorem Convergence Rate of Expected Stieltjes Transform}, we can get the following useful corollary.
\begin{corollary}\label{Corollaryb 1}
	Under the assumptions (A1)-(A3), it holds true that
\begin{align}
		&E\left[\left(b_{\hat{\mathbf{R}}}(z)-b_K(z)\right)^2\right]\leqslant \frac{1}{K^2}P_{18}(|z|)P_{24}(|{\rm Im}z|^{-1}), \label{E(bR-b)^2}
		\\& E\left[\left(b'_{\hat{\mathbf{R}}}(z)-b'_K(z)\right)^2\right]\leqslant \frac{1}{K^2}P_{30}(|z|)P_{40}(|{\rm Im}z|^{-1}),\label{E(bR'-b')^2}
\end{align}
for $z\in \mathbb{C}\setminus\mathbb{R}$ and all large $K$.
\end{corollary}
\begin{proof}
	See Appendix F in Supplementary Material.
\end{proof}

 We emphasize that, similar to \cite{Vallet2012Improved}, the results presented in Theorem \ref{Theorem Convergence Rate of Expected Stieltjes Transform} and Corollary \ref{Corollaryb 1} are only valid in the complex Gaussian case. This limitation arises because the proofs of these results employ the Poincare-Nash inequality (see e.g., \cite[Lemma 7]{Vallet2012Improved}) to rigorously control the bounds of the variance.  Notably, this mathematical tool is tailored specifically for the complex Gaussian scenario. 
\subsection{Improved Estimator with Its Asymptotic Behavior}
\subsubsection{Proposed estimator and its first-order asymptotic property}In the following theorem, we introduce the proposed estimator and present its first-order asymptotic property.
\begin{mytheorem}\label{Theorem first-order behavior}
	Consider the following estimator 
	\begin{equation}\label{proposed estimator}
		\breve{\gamma}_m = \frac{1}{N_m}\sum_{k\in\mathcal{N}_m}\hat{\rho}_k\left(1-c_K-\frac{c_K}{N}\sum_{i=1\ne k}^N\frac{\hat{\rho}_i}{\hat{\rho}_k-\hat{\rho}_i}\right)
	\end{equation}
	for all $1\leqslant m\leqslant L$, where $\mathcal{N}_m$ denotes the set of indices
	\begin{equation}\label{key}
		\mathcal{N}_m = \left\{\sum_{i=1}^{m-1}N_i+1,\sum_{i=1}^{m-1}N_i+2,...,\sum_{i=1}^mN_i\right\}.
	\end{equation}
	
	Then under the assumptions (A1)-(A4), $\breve{\gamma}_m$ is a consistent estimator of $\gamma_m$, i.e., the following convergence holds true:
	\begin{equation}\label{a.s. convergence of estimator}
		|\breve{\gamma}_m-\gamma_m|\overset{a.s.}{\longrightarrow}0.
	\end{equation}
 Moreover, the bias term of $\breve{\gamma}_m$ is $O(K^{-2})$ term, i.e.,
	\begin{equation}\label{bias term of estiamtor}
		|E\breve{\gamma}_m - \gamma_m|=O\left(\frac{1}{K^2}\right)
	\end{equation}
	for all $1\leqslant m\leqslant L$.
\end{mytheorem}
\begin{proof}
The proof  is deferred to Section IV-A. 
\end{proof}

\textit{Remark 4:} Theorem \ref{Theorem first-order behavior} states that the estimator $\breve{\gamma}_m$ is consistent under large dimensional regime. Actually, this consistency holds true even if $c\to 0$. In other words, the estimator $\breve{\gamma}_m$ is also consistent under the large sample size case.

\textit{Remark 5:} Notice that the assumption (A4) is quite strict and can be relaxed to some extent. For example, if only the cluster associated with $\gamma_m$ ($m\in\{1,2,...,L\}$) meets the spectral separation condition, then both the convergence $|\breve{\gamma}_m-\gamma_m|\overset{a.s.}{\longrightarrow}0$ and the convergence rate $O(1/K^2)$ still hold true. From \cite{Mestre2008On} and the arguments in Section II, we know that the spectral separation condition for $\gamma_m$ can be formulated as $\inf_K\phi_K(m)>0$ with $\phi_K(m)$ given by
		\begin{equation}\label{key}
			\phi_K(m) = \begin{cases}
				\Psi(\bar{f}_1), & m=1 \\
				\max\{\Psi\left(\bar{f}_{m-1}\right),\Psi\left(\bar{f}_{m}\right)\}, & 1<m<L \\
				\Psi(\bar{f}_{L-1}),& m=L
			\end{cases}
		\end{equation}
		where $\Psi(f)$ is defined in \eqref{Psi(f)}; $\bar{f}_i$ are the real-valued solutions to the equation $\Psi'(f) = 0$ at $f$ under convention $0<\bar{f}_1<\bar{f}_2<\cdots<\bar{f}_{L-1}$. 
	

\subsubsection{The second-order asymptotic behavior}
In the following theorem, we establish a CLT to describe the asymptotic fluctuation of the new estimator $\breve{\gamma}_m$ ($1\leqslant m\leqslant L$).
\begin{mytheorem}\label{Theorem Second-order asymptotic behavior}
	Let assumptions (A1)-(A4) hold true. Define the vectors $\breve{\bm{\gamma}} = [\breve{\gamma}_1,\breve{\gamma}_2,\cdots,\breve{\gamma}_L]^T$ and ${\bm \gamma} = [{\gamma}_1,{\gamma}_2,\cdots,{\gamma}_L]^T$. Then it holds true that
	\begin{equation}\label{Asymptotic Gaussianity}
		K\left(\breve{\bm{\gamma}} - {\bm \gamma}\right)\overset{d}{\rightarrow} \mathcal{N}_L \left(\mathbf{0}, \mathbf{\Theta } \right)
	\end{equation}
	where $\mathcal{N}_L \left(\mathbf{0}, \mathbf{\Theta }\right)$ denotes a real $L$-dimensional Gaussian distribution with mean $\mathbf{0}$ and covariance matrix $\mathbf{\Theta } \in \mathbb{R}^{L\times L}$. The $(m, n)$th element of $\mathbf{\Theta }$ is given by 
		\begin{equation}\label{28}
			\small\begin{aligned}
				\mathbf{\Theta }_{mn} =  -\frac{1}{4\pi^2c_mc_n}\oint_{{\Gamma}_m^+}\oint_{{\Gamma}_n^{'+}}\frac{{\underline{b}}_K\left(\frac{1}{z_1}\right){\underline{b}}_K\left(\frac{1}{z_2}\right)}{z_1^2z_2^2}\kappa\left(\frac{1}{z_1},\frac{1}{z_2}\right)dz_1dz_2
			\end{aligned}
		\end{equation}
	where the function $\kappa(z_1,z_2)$ is given by
	\begin{equation}\label{kappa}
		\kappa(z_1,z_2) = 	\frac{{\underline{b}}'_K(z_1){\underline{b}}'_K(z_2)}{\left({\underline{b}}_K(z_1)-{\underline{b}}_K(z_2)\right)^2}-\frac{1}{(z_1-z_2)^2}
	\end{equation}
	and the contours $\Gamma_{m}^+$ and $\Gamma_{n}^{'+}$ are positively oriented rectangles, symmetric with respect to real axis, and encloses the cluster $\tilde{\mathcal{S}}_m$ and $\tilde{\mathcal{S}}_n$, respectively. In addition, they verify
	\begin{equation}
		\label{30}
		\begin{aligned}
			&	\Gamma^{+}_{m} \cap \Gamma^{+}_{n} = 	\Gamma^{'+}_{m}\cap \Gamma^{'+}_{n} = \emptyset,&\text{\rm for $m\ne n$} \\
			&\Gamma^{+}_{m}\cap\Gamma^{'+}_{n}=\emptyset,	&\text{\rm for all $m, n$} 
		\end{aligned}
	\end{equation}
	and the families $\Gamma^{+}_{m}$ and $\Gamma^{'+}_{m}$ ($1\leqslant m\leqslant L$) are non-overlapping.
\end{mytheorem}

\begin{proof}
The proof is postponed to Section IV-B.
\end{proof}

\textit{Remark 6:} The assumption (A3), i.e., the Gaussianity of the observations,  is very strict. We remark that this assumption can be relaxed by the techniques introduced in  \cite{Pan2014Comparison} and \cite{Bodnar2016Spectral}.

\textit{Remark 7:} In Theorem \ref{Theorem Second-order asymptotic behavior},  assumption (A4) is also very strict and can be relaxed. Specifically, it is not necessary for all $L$ distinct eigenvalues to be separated. If only $M$ ($M\leqslant L$) distinct eigenvalues satisfy the spectral separation condition, the estimates obtained using \eqref{proposed estimator} will  still asymptotically follow an $M$-dimensional Gaussian distribution. To be more precise, define  $\bm{\gamma}'=[{\gamma}_{k_1},{\gamma}_{k_2},...,{\gamma}_{k_M}]^T$ ($\{k_1,k_2,...,k_M\}\subset\{1,2,...,L\}$), with each $\gamma_{k_i}$ satisfy the spectral separation condition $\inf_K\phi_K(k_i)>0$ for all $i=1,2,...,M$. Further define the estimates $\breve{\bm{\gamma}}' = [\breve{\gamma}_{k_1},\breve{\gamma}_{k_2},...,\breve{\gamma}_{k_M}]$. Then it holds that $K(\breve{\bm{\gamma}}'-\bm{\gamma}')\overset{d}{\longrightarrow}\mathcal{N}_M(\mathbf{0},\tilde{\mathbf{\Theta}})$ with $\tilde{\mathbf{\Theta}}\in\mathbb{C}^{M\times M}$. The $(m,n)$th entry of $\tilde{\mathbf{\Theta}}$ is given by
\begin{equation}\label{}\nonumber
	\small\begin{aligned}
		&\tilde{\mathbf{\Theta }}_{mn} = 
		\\& -\frac{1}{4\pi^2c_{k_m}c_{k_n}}\oint_{{\Gamma}_{k_m}^+}\oint_{{\Gamma}_{k_n}^{'+}}\frac{{\underline{b}}_K\left(\frac{1}{z_1}\right){\underline{b}}_K\left(\frac{1}{z_2}\right)}{z_1^2z_2^2}\kappa\left(\frac{1}{z_1},\frac{1}{z_2}\right)dz_1dz_2
	\end{aligned}
\end{equation}
for all $1\leqslant m,n \leqslant M$. 

\textit{Remark 8:} The convergence rate established in \eqref{bias term of estiamtor} is exclusively applicable to the complex Gaussian situation, as its derivation is based on Theorem \ref{Theorem Convergence Rate of Expected Stieltjes Transform} and Corollary \ref{Corollaryb 1}. Similarly,  the asymptotic Gaussianity in Theorem \ref{Theorem Second-order asymptotic behavior} is also limited to the complex Gaussian case, as its proof relies on \cite[Lemma 1]{Yao2013Fluctuations} which utilizes the Poincare-Nash inequality for complex Gaussian random vectors.  However,  we highlight that the estimator presented in \eqref{proposed estimator},  along with its almost sure convergence stated in \eqref{a.s. convergence of estimator}, is not restricted to the complex Gaussian case and can be naturally extended to the real Gaussian case. This broader applicability stems from the fact that both the derivation of \eqref{proposed estimator} and almost sure convergence in \eqref{a.s. convergence of estimator} are based on the Theorem \ref{Theorem Weak Convergence of u_N}, which remains valid in both the complex and real Gaussian case.

Theorem \ref{Theorem Second-order asymptotic behavior} characterizes the asymptotic performance of the new estimator $\breve{\gamma}_m$ under large dimensional regime. Unfortunately, the elements of  $\mathbf{\Theta}$ are defined by integrals, which rely on the prior knowledge of eigenvalues of  $\mathbf{M}_N$. In the following theorem, we provide a pointwise consistent estimator for $\mathbf{\Theta}$, which only uses the sample eigenvalues of $\hat{\mathbf{M}}_N$.

\begin{mytheorem}\label{Theorem the consistent estimator for Theta}
	Denote an estimator of $\mathbf{\Theta}$ by $\hat{\mathbf{\Theta}}$, whose $(m,n)$th element is given by
	\begin{equation}\label{the expression of thetakl}
		\hat{\mathbf{\Theta}}_{mn} = \begin{cases}
			\frac{1}{N_m^2}\sum\limits_{k\in\mathcal{N}_m}\sum\limits_{l\in\mathcal{N}^c_m}\frac{\hat{\rho}_k^2\hat{\rho}_l^2}{(\hat{\rho}_k-\hat{\rho}_l)^2}+\frac{K-N}{N_m^2}\sum\limits_{k\in\mathcal{N}_m}\hat{\rho}_k^2,&m=n\\	-\frac{1}{N_mN_n }\sum\limits_{k\in\mathcal{N}_m}\sum\limits_{l\in\mathcal{N}_n}\frac{\hat{\rho}_k^2\hat{\rho}_l^2}{(\hat{\rho}_k-\hat{\rho}_l)^2},&  m\ne n
		\end{cases}
	\end{equation}
	for all $1\leqslant m,n\leqslant L$, where $\mathcal{N}^c_m$ represents the set $\{1,2,...,N\}\setminus\mathcal{N}_m$.
	
	Then under the assumptions (A1)-(A4), $|\hat{\mathbf{\Theta}}_{mn}-\mathbf{\Theta}_{mn}|\overset{a.s.}{\longrightarrow}0$
	for all $1\leqslant m,n\leqslant L$.
\end{mytheorem}
\begin{proof}
	See Section IV-C.
\end{proof}

We emphasize that Theorem \ref{Theorem the consistent estimator for Theta} is useful in practical applications because it provides an estimation of the degree of confidence for each $\breve{\gamma}_m$.  
\section{Proofs of Theorems \ref{Theorem first-order behavior}-\ref{Theorem the consistent estimator for Theta}}
In this section, we give the proofs of Theorems \ref{Theorem first-order behavior}-\ref{Theorem the consistent estimator for Theta}.
\subsection{Proof of Theorem \ref{Theorem first-order behavior}}
We first prove \eqref{a.s. convergence of estimator}. Using the Cauchy integral formula, we can express $\gamma_m$ in an integral form:
\begin{equation}\label{integral in w domain}
	\gamma_m  = \frac{N}{N_m2\pi \mathsf{i}}\oint_{\mathcal{C}_m^+}-\omega\frac{1}{N}\sum_{i=1}^LN_i\frac{1}{\gamma_i-\omega}d\omega
\end{equation}
where $\mathcal{C}_m^+$ is a positively oriented contour taking values on $\mathbb{C}\setminus \{\gamma_1,\gamma_2,...,\gamma_L\}$ and only enclosing $\gamma_m$. Now we need to carefully choose the integration contour $\mathcal{C}_m^+$. We know from Section II that $\gamma_m$ always locates in interval $\mathcal{W}_m$. Hence We can take the contour $\mathcal{C}_m^+$ as a rectangle symmetric with respect to real axis:
\begin{equation}\label{key}\nonumber
	\mathcal{C}_m^{+} = \{\omega\in\mathbb{C}:\tau_1\leqslant{\rm Re}(\omega)\leqslant \tau_2, |{\rm Im}(z)|\leqslant \omega_y\}
\end{equation}
where $\omega_y$ is a positive value, and $\tau_1$ and $\tau_2$ are two real values such that $\mathcal{W}_m\subset\left(\tau_1,\tau_2\right)$.

With the change of variable $\omega = z(1+cz\tilde{b}_K(z))$, the integral \eqref{integral in w domain} is transformed into $z$-domain:
\begin{equation}\label{34}
	\begin{aligned}
		\gamma_m  &= \frac{N}{N_m2\pi \mathsf{i}}\oint_{\Gamma_m^+}-\frac{1}{N}\sum_{i=1}^LN_i\frac{z(1+cz\tilde{b}_K(z))}{\gamma_i-z(1+cz\tilde{b}_K(z))}\omega'dz
		\\& = \frac{N}{N_m2\pi \mathsf{i}}\oint_{\Gamma_m^+}g(z)dz
	\end{aligned}
\end{equation}
where
\begin{equation}\label{key}
	g(z)=	-z\tilde{b}_K(z)\left(1+2zc\tilde{b}_K(z)+z^2c\tilde{b}'_K(z)\right),
\end{equation}
and the contour $\mathcal{C}_m^+$ is transformed to $\Gamma_m^+$, which is a positively oriented rectangle symmetric with respect to real axis:
\begin{equation}\label{key}
	\Gamma_m^+ =\{z\in\mathbb{C}:t_1\leqslant{\rm Re}(z)\leqslant t_2, |{\rm Im}(z)|\leqslant y\}
\end{equation}
where $y$ is a positive value; $t_1$ and $t_2$  are two real values such that the interval $(t_1,t_2)$ only encloses $\tilde{\mathcal{S}}_m$. This can be ensured by the assumption (A4).

Now by using Theorem 2 and the dominated convergence theorem, we can obtain a pointwise strongly consistent estimator for $g(z)$:
\begin{equation}\label{hat g}
	\begin{aligned}
		\hat{g}(z) = -zb_{\hat{\mathbf{M}}}(z)\left(1+2c_Kzb_{\hat{\mathbf{M}}}(z)+z^2c_Kb'_{\hat{\mathbf{M}}}(z)\right)
	\end{aligned}
\end{equation}
which will be used to estimate $\gamma_m$. Before proceeding, we introduce the following useful lemma.
\begin{mylemma}\label{Lemma gamma}
	Let assumptions (A1)-(A4) hold true. Then it holds true that
	\begin{equation}\label{key}
		\left| \frac{N}{N_m2\pi \mathsf{i}}\oint_{\Gamma_m^+}\hat{g}(z)dz -\gamma_m\right|\overset{a.s.}{\longrightarrow}0
	\end{equation}
	for all $1\leqslant m\leqslant L$.
\end{mylemma}
\begin{proof}
	By using the integral representation of $\gamma_m$, we have
	\begin{equation}\label{key}
		\begin{aligned}
			\left| \frac{N}{N_m2\pi \mathsf{i}}\oint_{\Gamma_m^+}\hat{g}(z)dz -\gamma_m\right|
			&= \left|\frac{N}{N_m2\pi \mathsf{i}}\oint_{\Gamma_m^+}(\hat{g}(z)-g(z))dz \right|
			\\&\leqslant \frac{N}{N_m2\pi }\oint_{\Gamma_m^+}\left|\hat{g}(z)-g(z)\right||dz |.
		\end{aligned}
	\end{equation}
	
	Next we need to prove that $\left|\hat{g}(z)-g(z)\right|$ is bounded over the contour $\Gamma_m^+$, which is equivalent to show that 
	\begin{equation}\label{41}
		\sup_{z\in\Gamma_m^+} |\hat{g}(z)-g(z)|<\infty
	\end{equation}
	for all large $N,K$. 
	
	We note that $\hat{g}(z)$ is holomorphic on $\mathbb{C}$, except for a set of poles: $\{\hat{\rho}_1,\hat{\rho}_2,...,\hat{\rho}_N\}$. Hence the problem to prove \eqref{41} lies in the discontinuity of $\hat{g}(z)$ on the real axis. This can be handled by the exact spectral separation arguments in \cite{Bai1998No} and \cite{Bai1999Exact}. According to the results in \cite{Bai1998No} and \cite{Bai1999Exact}, we can easily check that only the eigenvalues \{$\hat{\rho}_k,k\in\mathcal{N}_m\}$ are located inside $\Gamma_m^+$ and there will be no sample eigenvalues outside the support $\tilde{\mathcal{S}}$ for all large $N,K$. This implies that $\sup_{z\in\Gamma_m^+}|\hat{g}(z)|<\infty$.
	
	Hence from the above analysis, we have
	\begin{equation}\label{key}
		\sup_{z\in\Gamma_m^+} |\hat{g}(z)-g(z)|\leqslant 	\sup_{z\in\Gamma_m^+}|\hat{g}(z)| + 	\sup_{z\in\Gamma_m^+}|g(z)|<\infty
	\end{equation}
	almost surely for all large $N,K$. Consequently, applying the dominated convergence theorem yields
	\begin{equation}\label{key}
		\frac{N}{N_m2\pi }\oint_{\Gamma_m^+}\left|\hat{g}(z)-g(z)\right||dz |\overset{a.s.}{\longrightarrow} 0
	\end{equation}
	for all large $N,K$. This completes the proof of Lemma \ref{Lemma gamma}.
\end{proof}

If we write
\begin{equation}\label{hat gamma m integral form}
	\breve{\gamma}_m = \frac{N}{N_m2\pi \mathsf{i}}\oint_{\Gamma_m^+}\hat{g}(z)dz,
\end{equation}
then Lemma \ref{Lemma gamma} implies that $\breve{\gamma}_m$ is a strongly consistent estimator of $\gamma_m$. To complete the proof of \eqref{a.s. convergence of estimator}, we only need to check that the integral \eqref{hat gamma m integral form}  is equal to \eqref{proposed estimator}. 

By substituting the expressions of $b_{\hat{\mathbf{M}}}(z)$ and $b'_{\hat{\mathbf{M}}}(z)$ into \eqref{hat g}, we expand the expression of $\hat{g}(z)$ as
\begin{equation}\label{key}
	\hat{g}(z) = -\left(\frac{1}{N}\sum_{i=1}^N\frac{z}{\hat{\rho}_i-z}\right) \left(1-c_K+\frac{c_K}{N}\sum_{r=1}^N\left(\frac{\hat{\rho}_r}{\hat{\rho}_r-z}\right)^2\right).
\end{equation}

Substituting this expression into \eqref{hat gamma m integral form} yields
\begin{equation}\label{45}
	\hat{\gamma}_m= \frac{1}{N_m}\sum_{i=1}^N\frac{1}{2\pi \mathsf{i}}\oint_{\Gamma_m^+}\hat{g}_i(z)dz
\end{equation}
where 
\begin{equation}\label{key}
		\hat{g}_i(z) = -\frac{z\left(1-c_K+\frac{c_K}{N}\sum_{r=1}^N\left(\frac{\hat{\rho}_r}{\hat{\rho}_r-z}\right)^2\right)}{\hat{\rho}_i-z}.
\end{equation}

A straightforward analysis shows that  the function $\hat{g}_i(z)$ has poles $\left\{\hat{\rho}_k,k\in\mathcal{N}_m\right\}$ within the contour $\Gamma_m^+$. Hence using the residue theorem, we can evaluate $\hat{\gamma}_m$ in \eqref{45} by
\begin{equation}\label{47}
	\breve{\gamma}_m = \frac{1}{N_m}\sum_{k\in\mathcal{N}_m}\sum_{i=1}^N{\rm Res}(g_i(z),\hat{\rho}_k).
\end{equation}

Next we need to carefully analyze these poles. 
\begin{itemize}
	\item When $i=k$, it is easy to find that $\left\{\hat{\rho}_k,k\in\mathcal{N}_m\right\}$ are the third-order poles of $\hat{g}_i(z)$. The corresponding residue is therefore calculated by
	\begin{equation}\label{48}
		{\rm Res}(\hat{g}_k(z),\hat{\rho}_k) = \hat{\rho}_k\left(1-c_K+\frac{c_K}{N}\sum_{r=1\ne k}^N\left(\frac{\hat{\rho}_r}{\hat{\rho}_r-\hat{\rho}_k}\right)^2\right).
	\end{equation}
	\item When $i\ne k$, we find that $\{\hat{\rho}_k,k\in\mathcal{N}_m\}$ are the second-order poles of $\hat{g}_i(z)$. The residue is computed by
	\begin{equation}\label{49}
		{\rm Res}(\hat{g}_i(z),\hat{\rho}_k) = -\frac{c_K}{N}\frac{\hat{\rho}_k^2\hat{\rho}_i}{\hat{\rho}_k-\hat{\rho}_i}.
	\end{equation}
\end{itemize}

Then by substituting \eqref{48} and \eqref{49} into \eqref{47} and after some simple manipulations, we get
\begin{equation}\label{key}
	\small	\begin{aligned}
		\breve{\gamma}_m &= \frac{1}{N_m}\sum_{k\in\mathcal{N}_m}\left({\rm Res}(g_k(z),\hat{\rho}_k) +\sum_{i=1\ne k}^{N}{\rm Res}(g_i(z),\hat{\rho}_k)\right)
		\\& = \frac{1}{N_m}\sum_{k\in\mathcal{N}_m}\hat{\rho}_k\left(1-c_K-\frac{c_K}{N}\sum_{i=1\ne k}^N\frac{\hat{\rho}_i}{\hat{\rho}_k-\hat{\rho}_i}\right)
	\end{aligned}
\end{equation}
which completes the proof of \eqref{a.s. convergence of estimator}.

Next our goal is to establish \eqref{bias term of estiamtor}. From the relation \eqref{relation 2}, we have
\begin{equation}\label{50}
	\tilde{b}_K(z) = - \frac{1}{zc}\left(\frac{1}{z}{\underline{b}_K(1/z)}+1\right).
\end{equation}

Taking derivative of both sides of \eqref{50} with respect to $z$ yields
\begin{equation}\label{51-1}
	\tilde{b}'_K(z) =  \frac{1}{c}\frac{{\underline{b}}'_K(1/z)+2z{\underline{b}}_K(1/z)+z^2}{z^4}.
\end{equation}

Substituting \eqref{50} and \eqref{51-1} into \eqref{34}, we obtain an alternative integral representation of $\gamma$:
\begin{equation}\label{gamma_m}
		\gamma_m = \frac{1}{c_m2\pi \mathsf{i}}\oint_{\Gamma_m^+}\frac{1}{z^2}{\underline{b}}'_K(1/z) \left(\frac{1}{z}{{\underline{b}}_K(1/z)}+1\right)dz.
\end{equation}

Likewise, we can re-express $\breve{\gamma}_m$ by the following integral from:
\begin{equation}\label{hat gamma_m}
		\breve{\gamma}_m =  \frac{K}{N_m2\pi \mathsf{i}}\oint_{\Gamma_m^+}\frac{1}{z^2}b'_{\underline{\hat{\mathbf{R}}}}(1/z) \left(\frac{1}{z}{b_{\underline{\hat{\mathbf{R}}}}(1/z)}+1\right)dz.
\end{equation}

Then combining \eqref{gamma_m} and \eqref{hat gamma_m}, and noticing from  assumption (A2) that $N_m/K\to c_m$, we have
\begin{equation}\label{hat gamma_m - gamma_m}
	\begin{aligned}
		&\breve{\gamma}_m-\gamma_m
		\\& = \frac{1}{c_m2\pi\mathsf{i}}\oint_{\Gamma_m^+}\frac{1}{z^3}\left[\left({\underline{b}}_K(1/z)+z\right)\left(b'_{\hat{\underline{\mathbf{R}}}}(1/z)-{\underline{b}}'_K(1/z)\right)\right.
		\\&\qquad\qquad\qquad\qquad \left.+b'_{\hat{\underline{\mathbf{R}}}}(1/z)\left(b_{\hat{\underline{\mathbf{R}}}}(1/z)-{\underline{b}}_K(1/z)\right)\right]dz.
	\end{aligned}
\end{equation}

If we define
\begin{equation}\label{key}
	\begin{aligned}
		&\nu_1(z) = \frac{1}{z^3}\left({\underline{b}}_K(1/z)+z\right)\left(b'_{\hat{\underline{\mathbf{R}}}}(1/z)-{\underline{b}}'_K(1/z)\right),
		\\&\nu_2(z) = \frac{1}{z^3}\left(b'_{\hat{\underline{\mathbf{R}}}}(1/z)- \underline{b}'_K(1/z)\right)\left(b_{\hat{\underline{\mathbf{R}}}}(1/z)-{\underline{b}}_K(1/z)\right),
		\\&\nu_3(z) = \frac{1}{z^3} \underline{b}'_K(1/z)\left(b_{\hat{\underline{\mathbf{R}}}}(1/z)-{\underline{b}}_K(1/z)\right),
	\end{aligned}
\end{equation}
then \eqref{hat gamma_m - gamma_m} can be transformed to
\begin{equation}\label{key}
\hat{\gamma}_m-\gamma_m = \frac{1}{c_m2\pi\mathsf{i}}\oint_{\Gamma_m^+}(\nu_1(z)+\nu_2(z)+\nu_3(z))dz.
\end{equation}

Using Fubini's theorem, we have
\begin{equation}\label{59}
	E\hat{\gamma}_m-\gamma_m =  \frac{1}{c_m2\pi\mathsf{i}}\oint_{\Gamma_m^+}(E\nu_1(z)+E\nu_2(z)+E\nu_3(z))dz
\end{equation}
where
\begin{equation}\label{key}
	\small\begin{aligned}
		&E\nu_1(z) = \frac{1}{z^3}\left({\underline{b}}_K(1/z)+z\right)\left(Eb'_{\hat{\underline{\mathbf{R}}}}(1/z)-{\underline{b}}'_K(1/z)\right),
		\\&E\nu_2(z) = \frac{1}{z^3}E\left[\left(b'_{\hat{\underline{\mathbf{R}}}}(1/z)- \underline{b}'_K(1/z)\right)\left(b_{\hat{\underline{\mathbf{R}}}}(1/z)-{\underline{b}}_K(1/z)\right)\right],
		\\&E\nu_3(z) = \frac{1}{z^3} \underline{b}'_K(1/z)\left(Eb_{\hat{\underline{\mathbf{R}}}}(1/z)-{\underline{b}}_K(1/z)\right).
	\end{aligned}
\end{equation}

Applying Theorem \ref{Theorem Convergence Rate of Expected Stieltjes Transform} and using the facts that $|\underline{b}_K(1/z)|\leqslant\frac{1}{|{\rm Im}(1/z)|}=\frac{|z|^2}{|{\rm Im}z|}$ and $|\underline{b}'_K(1/z)|\leqslant \frac{1}{|{\rm Im}(1/z)|^2} = \frac{|z|^4}{|{\rm Im}z|^2}$, it is easy to check that 
\begin{equation}\label{61}
	\begin{aligned}
	|E\nu_1(z)| &\leqslant \frac{1}{|z|^3}\left(\frac{|z|^2}{|{\rm Im}z|}+|z|\right)\frac{1}{K^2}P_{15}(|1/z|)P_{20}\left(|{\rm Im}(1/z)|^{-1}\right)
	\\&\leqslant\frac{1}{K^2}P_{42}(|z|)P_{39}(|{\rm Im}z|^{-1}),
\end{aligned}
\end{equation}
and 
\begin{equation}\label{62}
	\begin{aligned}
	|E\nu_3(z)|&\leqslant \frac{1}{|z|^3}\frac{|z|^4}{|{\rm Im}z|^2}\frac{1}{K^2}P_9(|1/z|)P_{12}(|{\rm Im}(1/z)|^{-1})
	\\ &\leqslant \frac{1}{K^2}P_{25}(|z|)P_{23}(|{\rm Im}z|^{-1}).
\end{aligned}
\end{equation}

As for the bound of $|E\nu_2(z)|$, we need to use Cauchy-Schwarz inequality and Corollary \ref{Corollaryb 1}, which yield
\begin{equation}\label{63}
	\small \begin{aligned}
	&|E\nu_2(z)| 
	\\&\leqslant \frac{1}{|z|^3}\sqrt{E\left[(b'_{\underline{\hat{\mathbf{R}}}}(1/z)-b'_K(1/z))^2\right]E\left[(b_{\underline{\hat{\mathbf{R}}}}(1/z)-b_K(1/z))^2\right]}
	\\&\leqslant\frac{1}{K^2}P_{64}(|z|)P_{59}(|{\rm Im}z|^{-1}).
\end{aligned}
\end{equation}  

Given these useful bounds, we are now available to find the convergence rate  of $E\hat{\gamma}_m-\gamma_m$. Before proceeding, we can use the same arguments in \cite{Mestre2008Improved} to write the expression in \eqref{59} as
\begin{equation}\label{64}
	\begin{aligned}
		E\hat{\gamma}_m-\gamma_m =  \frac{1}{c_m}\sum_{k=1}^3\frac{1}{2\pi\mathsf{i}}\lim_{y\to0^+}&\left[\int_{t_1}^{t_2}E\nu_k(x+\mathsf{i}y)dx\right.
		\\&\left.-\int_{t_1}^{t_2}E\nu_k(x-\mathsf{i}y)dx\right].
	\end{aligned}
\end{equation}

In order to obtain the bound of \eqref{64}, we need to use Lemma 9 in \cite{Vallet2012Improved}, which is originally proved in \cite{Haagerup2005A} and generalized in \cite{Capitaine2007Strong}. By using Lemma 9 in \cite{Vallet2012Improved}  and noticing from \eqref{61}-\eqref{63} that $|E\nu_k(z)|\leqslant \frac{1}{K^2}P_m(|z|)P_n(|{\rm Im}z|^{-1})$ for $k=1,2,3$ and some positive integers $m,n$, we obtain
\begin{equation}\label{65-1}
	\lim_{y\to0^+}\int_{t_1}^{t_2}|E\nu_k(x+\mathsf{i}y)|dx\leqslant \frac{a_k}{K^2}< \infty
\end{equation}
for some positive real numbers $a_k$ ($k=1,2,3$) independent of $K$. 

Likewise, using the same arguments as in \cite{Haagerup2005A} and \cite{Capitaine2007Strong}, we can also obtain
\begin{equation}\label{66-1}
	\lim_{y\to0^+}\int_{t_1}^{t_2}|E\nu_k(x-\mathsf{i}y)|dx \leqslant \frac{d_k}{K^2}<\infty
\end{equation}
for some positive real numbers $d_k$ ($k=1,2,3$) independent of $K$. 

Hence, by using \eqref{65-1} and \eqref{66-1} into \eqref{64}, we finally obtain
\begin{equation}\label{key}
	|E\hat{\gamma}_m-\gamma_m|\leqslant \frac{1}{c_m}\sum_{k=1}^3\frac{1}{2\pi}\left(\frac{a_k}{K^2}+\frac{d_k}{K^2}\right) = O\left(\frac{1}{K^2}\right)
\end{equation}
which completes the proof of \eqref{bias term of estiamtor}.
\subsection{Proof of Theorem \ref{Theorem Second-order asymptotic behavior} }
We first consider a process
\begin{align*}
	(M_K,M'_K,u'_K): \quad&\mathcal{K}\to\mathbb{C}^3\\
	&z\mapsto (M_K(z),M'_K(z),u'_K(z))
\end{align*}
where $M_K(z) = K\left(b_{\hat{\underline{\mathbf{R}}}}(z) - {\underline{b}}_K(z)\right)
$, $M'_K(z) = K\left(b'_{\hat{\underline{\mathbf{R}}}}(z) - {\underline{b}}'_K(z)\right)$, and $u'_K(z) = b'_{\hat{\underline{\mathbf{R}}}}(z)$.

Let $\mathcal{K}$ be $\Gamma_m^+$. Then from \eqref{hat gamma_m - gamma_m}, we have
\begin{equation}
	\begin{aligned}
		&K(\breve{\gamma}_m-\gamma_m)\\ &=\frac{1}{c_m2\pi\mathsf{i}}\oint_{{\Gamma}_m^+}\frac{1}{z^3}\left[\zeta(z)M'_K(1/z)+u'_K(1/z)M_K(1/  z)\right]dz\\
		&\triangleq \Delta_K(M_K,M'_K,u'_K,{\Gamma}_m^+)
	\end{aligned}
\end{equation}
where we define $\zeta(z) \triangleq {\underline{b}}_K(1/z)+{z}$.

In what follows, we will divide the proof of Theorem \ref{Theorem Second-order asymptotic behavior} into three steps: (i) to prove that the process $(M_K(1/z),M'_K(1/z),u'_K(1/z))$ converges in distribution to a Gaussian process; (ii) to transfer the convergence from the process $(M_K(1/z),M'_K(1/z),u'_K(1/z))$ to $\Delta_K(M_K,M'_K,u'_K,\Gamma_m^+)$ by using the continuous mapping theorem and to check that the limit of $\Delta_K(M_K,M'_K,u'_K,\Gamma_m^+)$ is Gaussian; and (iii) to calculate the limiting covariance between $\Delta_K(M_K,M'_K,u'_K,\Gamma_m^+)$ and $\Delta_K(M_K,M'_K,u'_K,\Gamma_n^+)$.

Now let us start the proof.
	\subsubsection{Convergence of $(M_K(1/z),M'_K(1/z),u'_K(1/z))$}
In order to obtain the convergence of $(M_K,M'_K,u'_K)$, we need to use the following lemma that was proved in \cite{Yao2013Fluctuations} and \cite{Yao2012Eigenvalue}.
\begin{mylemma}\label{convergence of process M and diff of M}
	Consider a compact set $\mathcal{K}\subset\mathbb{C}$, which is symmetric with respect to the real axis. Then the process
	\begin{align*}
		(M_K,M'_K): \quad&\mathcal{K}\to\mathbb{C}^2\\
		&z\mapsto (M_K(z),M'_K(z))
	\end{align*}
	converges in distribution to a stochastic process $(M,M')$, satisfying $\overline{M(z)}=M(\bar{z})$ and $\overline{M'(z)}=M'(\bar{z})$. The process $(M,M')$ is a centered real Gaussian process with mean zero and covariance function defined as follows
	\begin{align*}
		\begin{matrix}
		&EM(z)M(\tilde{z}) = \kappa(z,\tilde{z});&EM(z)M'(\tilde{z}) = \frac{\partial{\kappa(z,\tilde{z})}}{\partial \tilde{z}}\\
		&EM'(z)M(\tilde{z}) = \frac{\partial{\kappa(z,\tilde{z})}}{\partial {z}}; &EM'(z)M'(\tilde{z}) = \frac{\partial^2}{\partial z\partial \tilde{z}}\kappa(z,\tilde{z})
		\end{matrix}
	\end{align*}
	where the function $\kappa(\cdot,\cdot)$ has been given by \eqref{kappa}.
\end{mylemma}

A direct consequence of the Lemma \ref{convergence of process M and diff of M} indicates that the process $(M_K(1/z),M'_K(1/z))$ converges in distribution to a Gaussian process $(M(1/z),M'(1/z))$ with mean zero and covariance function:
\begin{equation}\label{65}
	\begin{aligned}
		&EM(1/z)M(1/\tilde{z}) = \kappa(1/z,1/\tilde{z}),\\
		&EM(1/z)M'(1/\tilde{z}) = -\tilde{z}^2\frac{\partial{\kappa(1/z,1/\tilde{z})}}{\partial \tilde{z}},\\
		&EM'(1/z)M(1/\tilde{z}) = -z^2\frac{\partial{\kappa(1/z,1/\tilde{z})}}{\partial {z}},\\
		&EM'(1/z)M'(1/\tilde{z}) = z^2\tilde{z}^2\frac{\partial^2}{\partial z\partial \tilde{z}}\kappa(1/z,1/\tilde{z}).
	\end{aligned}
\end{equation}

In addition, by using the dominated convergence theorem, $u'_K(1/z)$ converges to ${\underline{b}}'_K(1/z)$ almost surely. Hence, the process $\left(M_K(1/z),M'_K(1/z),u'_K(1/z)\right)$ converges in distribution to a Gaussian process $\left(M(1/z),M'(1/z),{\underline{b}}'_K(1/z)\right)$ with mean $(0,0,{\underline{b}}'_K(1/z))$ and covariance structure given by \eqref{65}. Now step (i) is established.
\subsubsection{Convergence of $\Delta_K(M_K,M'_K,u'_K,\Gamma_m^+)$}
Now our goal is to prove that $\Delta_K(M_K,M'_K,u'_K,\Gamma_m^+)$ converges to a Gaussian distribution. Because $M(1/z),M'(1/z),\underline{b}'_K(1/z)$ are continuous over the contour $\Gamma_m^+$ ($1\leqslant m\leqslant L$), we can transfer the convergence of $\left(M_K(1/z),M'_K(1/z),u'_K(1/z)\right)$ to $\Delta_K(M_K,M'_K,u'_K,\Gamma_m^+)$ by using the continuous mapping theorem \cite[Theorem 4]{Yao2013Fluctuations}. Therefore, we obtain the following convergence
\begin{equation}\label{66}
	\Delta_K(M_K,M'_K,u'_K,{\Gamma}_m^+)\overset{d}{\longrightarrow}\Delta(M,M',{\underline{b}}'_K,{\Gamma}_m^+)
\end{equation}
where
\begin{equation}\label{67}
	\begin{aligned}
		&\Delta(M,M',{\underline{b}}'_K,{\Gamma}_m^+)
		\\&= \frac{1}{c_m2\pi\mathsf{i}}\oint_{{\Gamma}_m^+}\frac{1}{z^3}\left[\zeta(z)M'\left(\frac{1}{z}\right)+\underline{b}'_K\left(\frac{1}{z}\right)M\left(\frac{1}{z}\right)\right]dz.
	\end{aligned}
\end{equation}

The convergence of \eqref{66} only indicates that  $\Delta_K(M_K,M'_K,u'_K,\Gamma_m^+)$ converges in distribution to a random variable $\Delta(M,M',\underline{b}'_K,{\Gamma}_m^+)$, but its Gaussianity is still not verified. Next we need to show the Gaussianity of $\Delta(M,M',\underline{b}'_K,{\Gamma}_m^+)$. By noticing that the integrand of \eqref{67} is a linear combination of Gaussian random variables $M(1/z)$ and $M'(1/z)$, it is easy to check that the integrand of \eqref{67} still follows a Gaussian distribution. Hence the integral \eqref{67} can be written as the limit of a finite Riemann sum of a series of Gaussian random variables, which results in the Gaussianity of $\Delta(M,M',\underline{b}'_K,{\Gamma}_m^+)$.

Applying the convergence of $\Delta_K(M_K,M'_K,u'_K,\Gamma_m^+)$ to the whole vector $K\left(\bm{\breve{\gamma}} -\bm{ {\gamma}}\right)$ yields
\begin{equation}\label{key}
	\begin{pmatrix}
		\Delta_K(M_K,M'_K,u'_K,\Gamma_1^+)\\
		\vdots 
		\\
		\Delta_K(M_K,M'_K,u'_K,\Gamma_L^+)
	\end{pmatrix}\overset{d}{\longrightarrow}{\bm \Delta} = \begin{pmatrix}
		\Delta(M,M',\underline{b}'_K,\Gamma_1^+)\\
		\vdots 
		\\
		\Delta(M,M',\underline{b}'_K,\Gamma_L^+)
	\end{pmatrix}
\end{equation}
where ${\bm{\Delta}}\sim \mathcal{N}_L(\mathbf{s},{\bm \Theta})$ with $\mathbf{s}$ being an $L\times 1$ vector and $\bm \Theta$ being an $L\times L$ real-valued matrix.

Denote the $m$th element of $\mathbf{s}$ by $\mathbf{s}_m$. Then $\mathbf{s}_m$ is computed by	
\begin{equation}\label{key}
	\begin{aligned}
		\mathbf{s}_m &={E}\Delta(M,M',\underline{b}'_K,{\Gamma}_m^+)
\\&{=} \frac{1}{c_m2\pi\mathsf{i}}\oint_{{\Gamma}_m^+}\frac{1}{z^3}\left[\zeta(z){E}M'\left(\frac{1}{z}\right)+\underline{b}'_K\left(\frac{1}{z}\right){E}M\left(\frac{1}{z}\right)\right]dz
		\\& = 0
	\end{aligned}
\end{equation}
where we use the Fubini's theorem and the fact that ${E}M\left(1/z\right)={E}M'\left(1/z\right)=0$. 
Hence we have $\mathbf{s}=0$. Now we complete the proof of step (ii).
\subsubsection{Calculation of Covariance} We are now in position to calculate the covariance between $\Delta(M,M',\underline{b}'_K,{\Gamma}_m^+)$ and $\Delta(M,M',\underline{b}'_K,{\Gamma}_n^+)$ for all $1\leqslant m,n\leqslant L$. According to the definition of covariance, we have
\begin{equation}\label{70}
	\begin{aligned}
		\bm{\Theta}_{mn} &= {E}\Delta(M,M',\underline{b}'_K,{\Gamma}_m^+)\Delta(M,M',\underline{b}'_K,{\Gamma}_n^+)
		\\& =  \beta_{mn} \mathbb{E}\left(\oint_{{\Gamma}_m^+}\frac{1}{z_1^3}\left[\zeta(z_1)M'(\frac{1}{z_1})+\underline{b}'_K(\frac{1}{z_1})M(\frac{1}{z_1})\right]dz_1\right.
		\\&\qquad\quad\left.\times\oint_{{\Gamma}_n^+}\frac{1}{z_2^3}\left[\zeta(z_2)M(\frac{1}{z_2})+\underline{b}'_K(\frac{1}{z_2})M(\frac{1}{z_2})\right]dz_2\right)
	\end{aligned}
\end{equation}
where we define $\beta_{mn}\triangleq -\frac{1}{4\pi^2c_mc_n}$ in order to save space.

Note that we can take two different contours $\Gamma_n^+$ and $\Gamma_n^{'+}$ that satisfy \eqref{30} and only enclose the $n$th cluster, such that the contour integral \eqref{70} remains unchanged. Hence we can arbitrarily choose two non-overlapping contours of the same cluster $\tilde{\mathcal{S}}_n$ for calculation. By replacing $\Gamma_n^+$ with another non-overlapping contour $\Gamma_n^{'+}$ and using the Fubini's theorem and the covariance variance given in \eqref{65}, we obtain \eqref{covariance} at the top of the following page, where $\partial_1$, $\partial_2$ and $\partial^2_{12}$ represent $\partial/\partial z_1$,  $\partial/\partial z_2$ and  $\partial^2/\partial z_1\partial z_2$, respectively.

\begin{figure*}
	\begin{equation}\label{covariance}
		\begin{aligned}
			\bm{\Theta}_{mn}  &=\beta_{mn}\oint_{{\Gamma}_m^+}\oint_{{\Gamma}_n^{'+}}\frac{1}{z_1z_2}\left[\zeta (z_1)\zeta (z_2)\partial^2_{12}\kappa(1/z_1,1/z_2) + \zeta (z_1)\frac{\partial \underline{b}_K(1/z_2)}{\partial z_2}\partial_1\kappa(1/z_1,1/z_2)\right.
			\\&\qquad\qquad\qquad\qquad\qquad \left.+ \zeta (z_2)\frac{\partial \underline{b}_K(1/z_1)}{\partial z_1}\partial_2\kappa(1/z_1,1/z_2) + \frac{\partial \underline{b}_K(1/z_1)}{\partial z_1}\frac{\partial \underline{b}_K(1/z_2)}{\partial z_2}\kappa(1/z_1,1/z_2)\right] dz_1dz_2
		\end{aligned}
	\end{equation}
	\hrulefill
\end{figure*}

Next our task is to simplify \eqref{covariance}. By integration by parts, we have
\begin{equation}\label{73}
	\small	\begin{aligned}
		&\oint \frac{1}{z_1z_2}\zeta (z_1)\frac{\partial \underline{{b}}_K\left(1/z_2\right)}{\partial z_2}\partial_1\kappa(1/z_1,1/z_2) dz_1
		\\&=- \oint \frac{1}{z_1^2z_2}\frac{\partial \underline{b}_K(\frac{1}{z_2})}{\partial z_2}\left(z_1\frac{\partial \underline{b}_K(\frac{1}{z_1})}{\partial z_1} -\underline{b}_K\left(\frac{1}{z_1}\right)  \right)\kappa\left(\frac{1}{z_1},\frac{1}{z_2}\right)dz_1.
	\end{aligned}
\end{equation}

Likewise,
\begin{equation}\label{74}
	\small\begin{aligned}
		&\oint \frac{1}{z_1z_2}\zeta (z_1)\zeta (z_2)\partial^2_{12}\kappa(1/z_1,1/z_2)dz_1
		\\& = -\oint \frac{1}{z_1^2z_2}\zeta(z_2)\left(z_1\frac{\partial \underline{b}_K(\frac{1}{z_1})}{\partial z_1} -\underline{b}_K\left(\frac{1}{z_1}\right)  \right)\partial_2\kappa\left(\frac{1}{z_1},\frac{1}{z_2}\right)dz_1.
	\end{aligned}
\end{equation}

Substituting \eqref{73} and \eqref{74} into \eqref{covariance}, we get
\begin{equation}\label{75}
	\begin{aligned}
		&\bm{\Theta}_{mn} =
		\beta_{mn}\oint_{{\Gamma}_m^+}\oint_{{\Gamma}_n^{'+}} \frac{1}{z_1^2z_2}\left[ \zeta (z_2)\underline{b}_K(1/z_1)\partial_2\kappa(1/z_1,1/z_2) \right.
		\\&\qquad\qquad\qquad
		\left.+\frac{\partial \underline{b}_K(1/z_2)}{\partial z_2}  \underline{b}_K(1/z_1)\kappa(1/z_1,1/z_2)\right]dz_1dz_2.
	\end{aligned}
\end{equation}

Proceeding similarly, the integration by parts yields
\begin{equation}\label{76}
	\small	\begin{aligned}
		&\oint \frac{1}{z_1^2z_2}\zeta (z_2)\tilde{\underline{b}}_N(1/z_1)\partial_2\kappa(1/z_1,1/z_2)dz_2
		\\& = -\oint \frac{{\underline{b}}_K\left(\frac{1}{z_1}\right)}{z_1^2z_2^2} \left(z_2\frac{\partial \underline{b}_K\left(\frac{1}{z_2}\right)}{\partial z_2} -\underline{b}_K\left(\frac{1}{z_2}\right) \right)\kappa\left(\frac{1}{z_1},\frac{1}{z_2}\right)dz_2.
	\end{aligned}
\end{equation}

Finally, substituting \eqref{76} into \eqref{75}, we obtain \eqref{28}, which completes the proof of Theorem \ref{Theorem Second-order asymptotic behavior}.
\subsection{Proof of Theorem \ref{Theorem the consistent estimator for Theta}}
Reviewing the integral representation of $\bm{\Theta}_{mn}$ given in \eqref{28}, we can replace $\underline{b}_K(z)$ and $\underline{b}'_K(z)$ with their consistent estimators $b_{\underline{\hat{\mathbf{R}}}}(z)$ and $b'_{\underline{\hat{\mathbf{R}}}}(z)$. After this, we naturally obtain a consistent estimator for $\bm{\Theta}_{mn}$, given by 
\begin{equation}\label{77}
	\hat{\bm{\Theta}}_{mn} = \hat{\beta}_{mn}\oint_{{\Gamma}_m^+}\oint_{{\Gamma}_n^{'+}}\hat{h}(z_1,z)dz_1dz_2
\end{equation}
where $\hat{\beta}_{mn}\triangleq -\frac{K^2}{4\pi^2N_mN_n}$, and the integrand is given by
\begin{equation}
\begin{aligned}
		&\hat{h}(z_1,z_2) =
		\\& b_{\underline{\hat{\mathbf{R}}}}\left(\frac{1}{z_1}\right)b_{\underline{\hat{\mathbf{R}}}}\left(\frac{1}{z_2}\right)\left(\frac{ \frac{\partial b_{\underline{\hat{\mathbf{R}}}}(\frac{1}{z_1})}{\partial z_1} \frac{\partial b_{\underline{\hat{\mathbf{R}}}}(\frac{1}{z_2})}{\partial z_2} }{(b_{\underline{\hat{\mathbf{R}}}}(\frac{1}{z_1})-b_{\underline{\hat{\mathbf{R}}}}(\frac{1}{z_2}))^2} - \frac{1}{(z_1-z_2)^2}\right).
	\end{aligned}
\end{equation}
Note that the consistency of $\hat{\bm{\Theta}}_{mn}$ is guaranteed by the dominated convergence theorem.

The remaining work is to evaluate the integral \eqref{77} by using the residue theorem. We can split the integral \eqref{77} into two parts, i.e., $\hat{\bm{\Theta}}_{mn} = I_1+I_2$, where
\begin{equation}\label{79}
	\begin{aligned}
		I_1 = \hat{\beta}_{mn}\oint_{{\Gamma}_m^+}\oint_{{\Gamma}_n^{'+}} b_{\underline{\hat{\mathbf{R}}}}(\frac{1}{z_1})b_{\underline{\hat{\mathbf{R}}}}(\frac{1}{z_2})\frac{ \frac{\partial b_{\underline{\hat{\mathbf{R}}}}(\frac{1}{z_1})}{\partial z_1} \frac{\partial b_{\underline{\hat{\mathbf{R}}}}(\frac{1}{z_2})}{\partial z_2} }{(b_{\underline{\hat{\mathbf{R}}}}(\frac{1}{z_1})-b_{\underline{\hat{\mathbf{R}}}}(\frac{1}{z_2}))^2}dz_1dz_2
	\end{aligned}
\end{equation}
and
\begin{equation}\label{80}
	I_2 =  \hat{\beta}_{mn}\oint_{{\Gamma}_m^+}\oint_{{\Gamma}_n^{'+}}-\frac{ b_{\underline{\hat{\mathbf{R}}}}(\frac{1}{z_1})b_{\underline{\hat{\mathbf{R}}}}(\frac{1}{z_2})}{(z_1-z_2)^2}dz_1dz_2.
\end{equation}

We give the explicit results of $I_1$ and $I_2$ in the following proposition.

\begin{myproposition}\label{I1 and I2}
	The integrals $I_1$ and $I_2$ are respectively calculated by $I_1=0$ and 
	\begin{equation}\label{81}
		I_2= \begin{cases}
			\frac{1}{N_m^2}\sum\limits_{k\in\mathcal{N}_m}\sum\limits_{l\in{\mathcal{N}}^c_m}\frac{\hat{\rho}_k^2\hat{\rho}_l^2}{(\hat{\rho}_k-\hat{\rho}_l)^2}+\frac{K-N}{N_m^2}\sum\limits_{k\in\mathcal{N}_m}\hat{\rho}_k^2,& m=n\\	-\frac{1}{N_mN_n}\sum\limits_{k\in\mathcal{N}_m}\sum\limits_{l\in\mathcal{N}_n}\frac{\hat{\rho}_k^2\hat{\rho}_l^2}{(\hat{\rho}_k-\hat{\rho}_l)^2},& m\ne n.
			\end{cases}
	\end{equation}
\end{myproposition}
\begin{proof}
	See Appendix G in Supplementary Material.
\end{proof}

Substituting the results of Proposition \ref{I1 and I2} into $\hat{\bm{\Theta}}_{mn} = I_1+I_2$, we obtain \eqref{the expression of thetakl}, which completes the proof of Theorem \ref{Theorem the consistent estimator for Theta}.

\section{Numerical Examples}
In this section, we will present numerical examples to demonstrate the superiority of the proposed estimator and validate the accuracy of the Gaussian approximation. We consider two scenarios to assess the estimation performance of the proposed estimator, depending on whether the multiplicities of the eigenvalues of $\mathbf{M}_N$ are known. In Section V-A, assuming that the multiplicities are known, we compare the estimation performance of the proposed estimator against the maximum likelihood (ML) estimator and Mestre's estimator \cite{Mestre2008Improved}. Additionally, we compare the computational complexity of the proposed estimator with that of Mestre's estimator. In Section V-B, assuming that the multiplicities of $\mathbf{M}_N$ are unknown, we conduct a comparative study on the estimation performance of proposed estimator with the ML estimator, Mestre's estimator, and the moment method developed in \cite{Yao2012Eigenvalue}. Finally, in Section V-C, we verify by simulations the accuracy of the CLT and the Gaussian approximation respectively stated in Theorems \ref{Theorem Second-order asymptotic behavior} and \ref{Theorem the consistent estimator for Theta}. 


\subsection{ Performance Comparisons with Known Multiplicities}
\subsubsection{Performance comparisons of the proposed estimator against Mestre's estimator and ML estimator}
The ML estimator (see for instance \cite[Theorem 9.3.1]{Muirhead1982Aspects}, \cite[Theorem 2]{Anderson1963Asymptotic}, and  \cite[Theorem 2]{Mestre2008On}) and Mestre's improved estimator (see \cite[Theorem 3]{Mestre2008Improved}) are traditionally employed to estimate the eigenvalues of covariance matrix. In this study, we invert them to estimate the eigenvalues of precision matrix. It is worth noting that the inverted versions of these two estimators are still referred to as the ML estimator and Mestre's estimator, respectively. In the subsequent comparisons, we will utilize two key metrics: bias and mean square error (MSE),  to quantitatively evaluate the performance of the proposed estimator against ML and Mestre's estimators. In our experiments, the bias and MSE are respectively calculated by
\begin{equation}\label{key}
	{\rm bias} \triangleq \sum_{i=1}^L \left|\hat{m}_i-\gamma_i\right|,
\end{equation}
and 
\begin{equation}\label{key}
	{\rm MSE} \triangleq \sum_{i=1}^L\left(\frac{1}{N_S}\sum_{k=1}^{N_S}(\hat{\gamma}_i^{(k)}-\gamma_i)^2\right)
\end{equation}
where $\hat{m}_i=\frac{1}{N_S}\sum_{k=1}^{N_S}\hat{\gamma}_i^{(k)}$; $N_S$ denotes the numbers of simulation trials, and $\hat{\gamma}_i^{(k)}$ denotes the estimated $i$th eigenvalues at $k$th simulation. 
\begin{figure}[]
	\centering
	\subfloat[bias]{
		\centering
		\includegraphics[trim=1cm 0cm 1cm 1.8cm, scale=0.4]{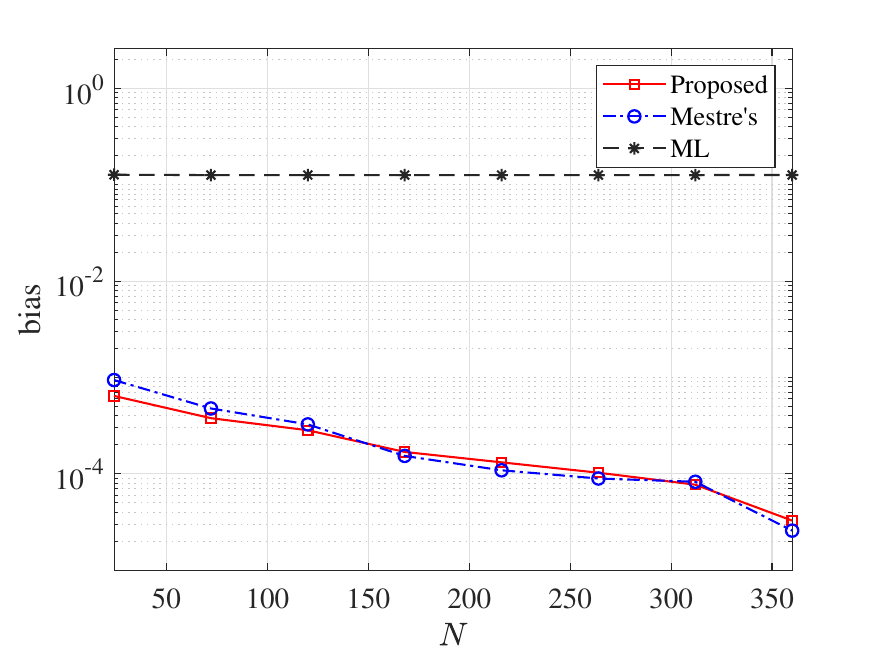}
	}\\
	\subfloat[MSE]{
		\centering
		\includegraphics[trim=0.5cm 0 1cm 1cm, scale=0.4]{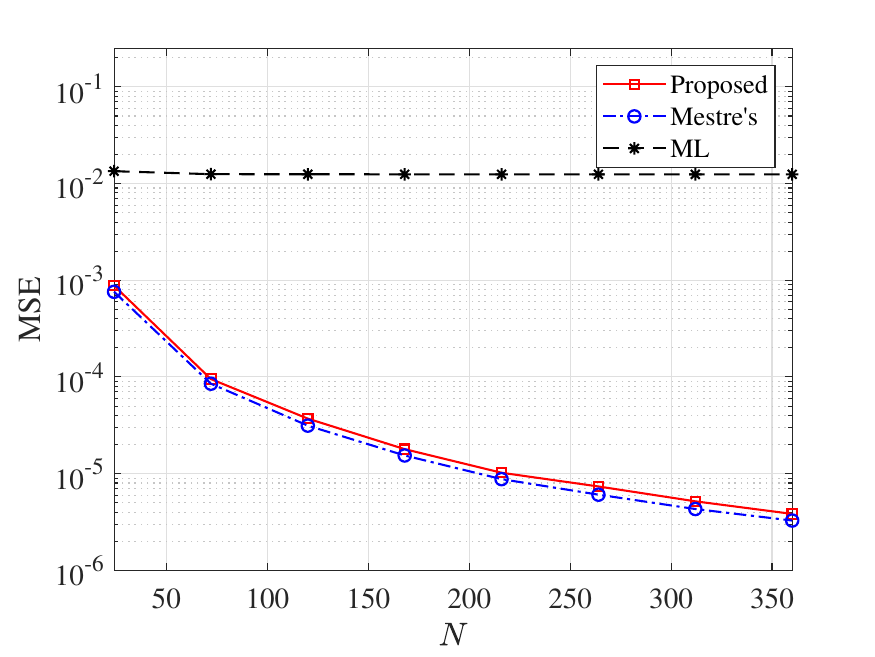}
	}\\
	\caption{Performance comparisons when the separability condition is satisfied. The experiment parameters are as follows: $\frac{N}{K}=\frac{3}{20}$, $(\rho_1,\rho_2,\rho_3) = (1,1/3,1/7)$, $(\frac{N_1}{N},\frac{N_2}{N},\frac{N_3}{N}) = \left(0.5,0.25,0.25\right)$, and $N$ increases from 24 to 360.}
	\label{Performance comparisons when the separability condition is satisfied.}\vspace{-0.5em}
\end{figure}

In what follows, we set $N_S = 1000$, that is, we will conduct 1000 simulation trials. We distinguish two cases to discuss depending on weather the separability condition (i.e., assumption (A4)) holds true or not.

We first consider the case where the separability condition is satisfied. The comparisons results are depicted in Fig. \ref{Performance comparisons when the separability condition is satisfied.}. We assume that the population matrix $\mathbf{R}_N$ has three distinct eigenvalues $(\lambda_1,\lambda_2,\lambda_3)=(1,3,7)$ with respective multiplicities $N_1$, $N_2$ and $N_3$ such that $N_1/N=0.5$, $N_2/N=N_3/N=0.25$. Additionally, we set $\frac{N}{K} = \frac{3}{20}$. Notice that these parameters are the same as those of the first experiment in  \cite{Yao2012Eigenvalue}. Correspondingly, the eigenvalues of  precision matrix $\mathbf{M}_N$ are $(\gamma_1,\gamma_2,\gamma_3)=(1,1/3,1/7)$ with the multiplicities $N_1$, $N_2$ and $N_3$. From Fig. \ref{Performance comparisons when the separability condition is satisfied.}, it can be observed that, the ML estimator, as an inconsistent estimator under large dimensional regime (as shown in \cite{Mestre2008On}), suffers from the worst performance in terms of bias. On the contrary, the proposed estimator along with Mestre's estimator stands out in terms of their accuracy, evidenced by their low biases. This means that both the proposed estimator and Mestre's estimator achieve the consistency when $N$ and $K$ increase at the same rate. If we use MSE as a metric to assess the overall performance, we find that the Mestre's estimator achieves a marginally lower MSE than our proposed estimator. However, as we shall further discuss in the subsequent subsection, the superiority of the Mestre's estimator actually comes with the cost of relatively high computational complexity.
\begin{figure}[]
	\centering
	\subfloat[bias]{
		\centering
		\includegraphics[trim=1cm 0 1cm 1.8cm, scale=0.4]{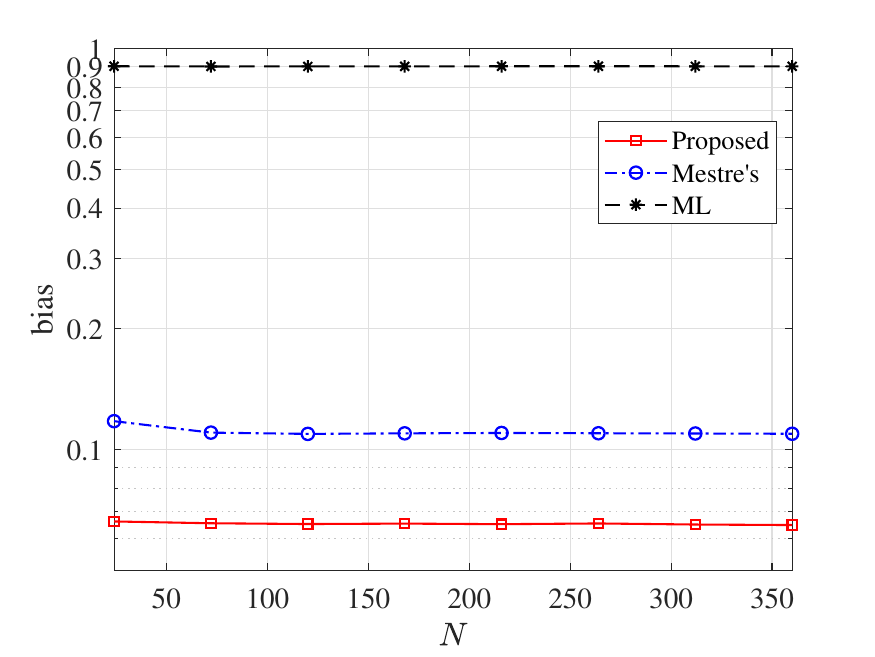}
	}\\
	\subfloat[MSE]{
		\centering
		\includegraphics[trim=0.5cm 0 1cm 1cm, scale=0.4]{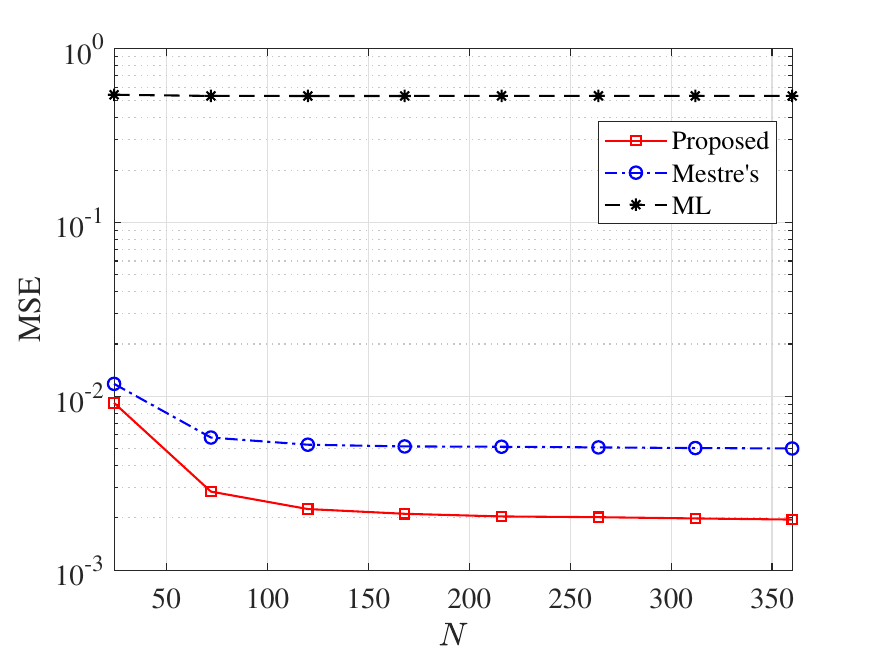}
	}\\
	\caption{Performance comparisons when the separability condition does not hold true. The experiment parameters are as follows: $\frac{N}{K}=\frac{3}{8}$, $(\rho_1,\rho_2,\rho_3) = (1,1/2,1/3)$, $(\frac{N_1}{N},\frac{N_2}{N},\frac{N_3}{N}) = \left(1/3,1/3,1/3\right)$, and $N$ increases from 24 to 360.}
	\label{Performance comparisons when the separability condition does not hold true.}\vspace{-1em}
\end{figure}

Next we consider the case where the separability condition does not hold true. Actually, this case is much more realistic in practical applications, as stated in \cite{Yao2012Eigenvalue}. Here, we use the same parameters as those of the second experiment in \cite{Yao2012Eigenvalue}. Specifically, we assume that the covariance matrix $\mathbf{R}_N$ has 3 different eigenvalues $(\lambda_1,\lambda_2,\lambda_3) = (1,2,3)$ with respective multiplicities $N_1$, $N_2$ and $N_3$ such that $\frac{N_1}{N} = \frac{N_2}{N}=\frac{N_3}{N}=\frac{1}{3}$. Accordingly, the precision matrix has three distinct eigenvalues $(\gamma_1,\gamma_2,\gamma_3)=(1,1/2,1/3)$ with 
respective multiplicities $N_1$, $N_2$ and $N_3$. The performance comparison results in terms of bias and MSE are given in Fig. \ref{Performance comparisons when the separability condition does not hold true.}. In this experiment, we set $\frac{N}{K}=\frac{3}{8}$.  As expected, Fig. \ref{Performance comparisons when the separability condition does not hold true.} reveals that the ML estimator has the worst performance. It exhibits an unacceptably large bias, resulting in a high MSE. In contrast, the proposed estimator exhibits the smallest bias.  We can see from Fig. \ref{Performance comparisons when the separability condition does not hold true.}(b) that the proposed estimator outperforms Mestre's estimator, if using MSE as a metric to evaluate the overall performance. This fact illustrates the  superiority of the proposed estimator in the case where the separability condition is not met.

The above numerical experiments focus on the estimation performance in estimating the individual eigenvalue. However, in many applications, such as the hypothesis test and discriminant analysis \cite{Bodnar2016Direct}, estimating functionals of the precision matrix, such as $\frac{1}{N}{\rm tr}(\mathbf{M}_N)$ and/or $\frac{1}{N}{\rm tr}(\mathbf{M}_N^2)$, is of significant interest. Hence, in the following, we consider estimating $\frac{1}{N}{\rm tr}(\mathbf{M}_N)$ to compare the performance of the proposed estimator, Mestre's estimator and ML estimator. We denote $g_1 = \frac{1}{N}{\rm tr}(\mathbf{M}_N)$. By employing the estimated eigenvalues, we estimate $m_1$ by
	\begin{equation}\label{key}\nonumber
		\hat{g}_1= \frac{1}{N}\sum_{i=1}^LN_i\hat{\gamma}_i.
	\end{equation}

	We still distinguish two cases to discuss based on whether the separability condition holds. We use the MSE to evaluate the estimation performance. The MSE for $\hat{g}_1$ is calculated by 
	\begin{equation}\label{key}\nonumber
		{\rm MSE} = \frac{1}{N_s}\sum_{k=1}^{N_s}|\hat{g}_1^{(k)}-g_1|^2
	\end{equation}
where $N_s$ is the number of simulation trials and $\hat{g}_1^{(k)}$ is the estimated $\hat{g}_1$ at $k$th simulation. We set $N_s=1000$ in our experiments.

The results are illustrated in Fig. \ref{Performance comparisons in estimating trM}(a) and (b), respectively. It should be pointed out that the experiment parameters are identical to those in Figs. \ref{Performance comparisons when the separability condition is satisfied.} and \ref{Performance comparisons when the separability condition does not hold true.}. As expected, the ML estimator, as an inconsistent estimator, performs the worst regardless of whether the separability condition is met. The proposed estimator and Mestre's estimator exhibits comparable performance when the separability condition holds true, as shown in Fig. \ref{Performance comparisons in estimating trM}(a). However, when the separability condition is not satisfied, as shown in Fig. \ref{Performance comparisons in estimating trM}(b), the performance of Mestre's estimator deteriorates significantly, while the proposed estimator still maintains excellent performance.

\subsubsection{Comparison of computational complexity between the proposed estimator and Mestre's estimator}
Given that ML estimator shows poor performance in large dimensional regime, we here will only compare the computational complexity of the proposed estimator and Mestre's estimator. Reviewing Mestre's estimator (see Theorem 3 in \cite{Mestre2008Improved}), we find that it involves a quantity $\hat{\mu}$, which necessitates solving a polynomial equation  of degree $N$ (refer to Eq. (19) in \cite{Mestre2008Improved}).  The computational complexity of solving a high-degree polynomial equation rises significantly with the increase of degree $N$.  In practice, the resolution of this high-degree polynomial equation relies heavily on  iterative methods like Newton's method, which typically require numerous iterations to converge, making the entire process extremely inefficient and time consuming. Conversely, our estimator simply combines sample estimates, thereby achieving high computational efficiency.
\begin{figure}[] 
	\centering
	\subfloat[Separability condition holds true]{
		\centering
		\includegraphics[trim=1cm 0 1cm 1.8cm, scale=0.4]{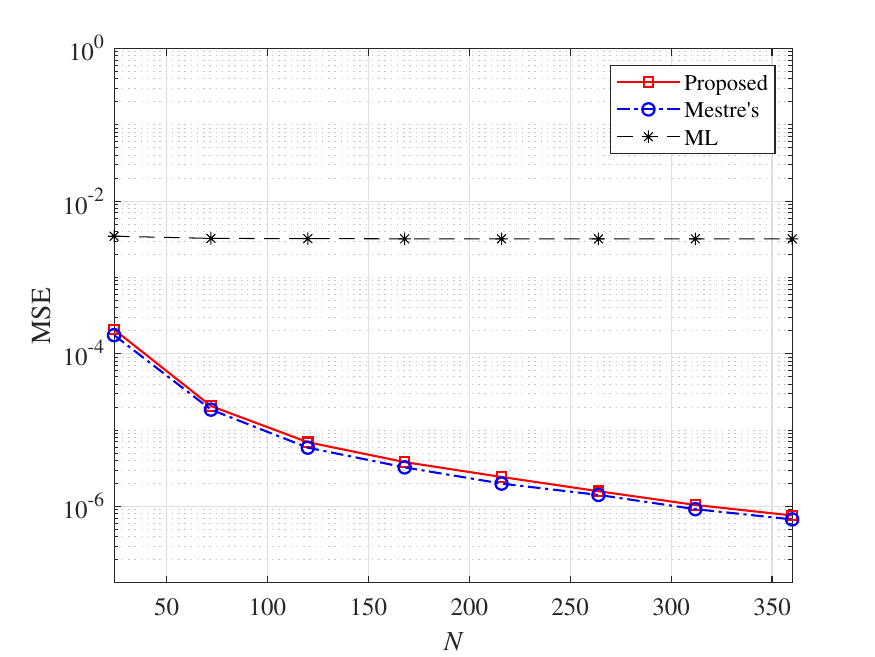}
	}\\
	\subfloat[Separability condition does not hold true]{
		\centering
		\includegraphics[trim=0.5cm 0 1cm 1cm, scale=0.4]{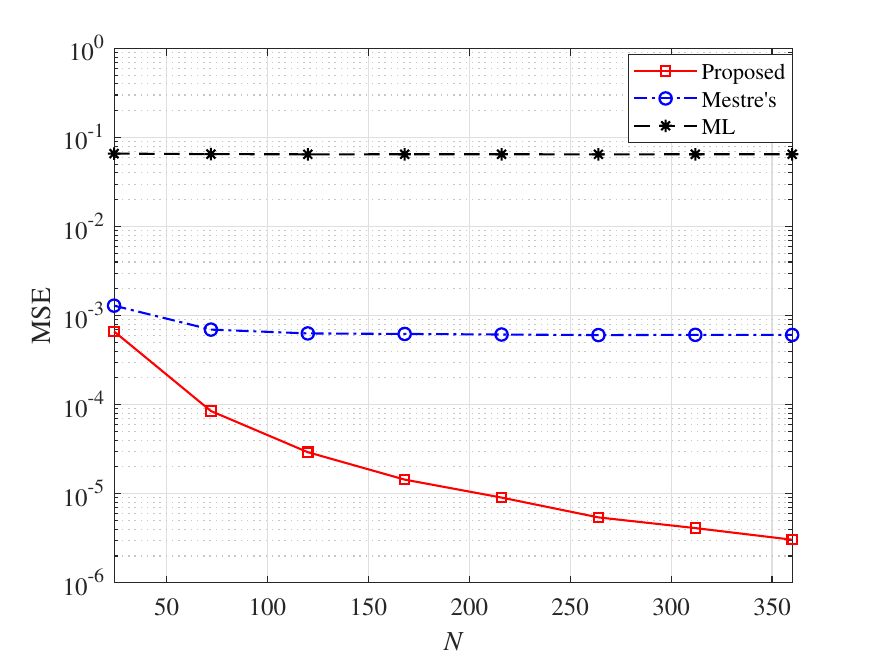}
	}\\
	\caption{Performance comparisons in estimating $g_1=\frac{1}{N}{\rm tr}(\mathbf{M}_N)$ when the multiplicities of eigenvalues of $\mathbf{M}_N$ are known. The experiment parameters are identical to those in Figs. \ref{Performance comparisons when the separability condition is satisfied.} and \ref{Performance comparisons when the separability condition does not hold true.}.}
	\label{Performance comparisons in estimating trM}
	\vspace{-1em}
\end{figure}
\begin{figure}[]
	\centering
	\subfloat[Separability condition holds true]{
		\centering
		\includegraphics[trim=1cm 0 1cm 1.8cm, scale=0.4]{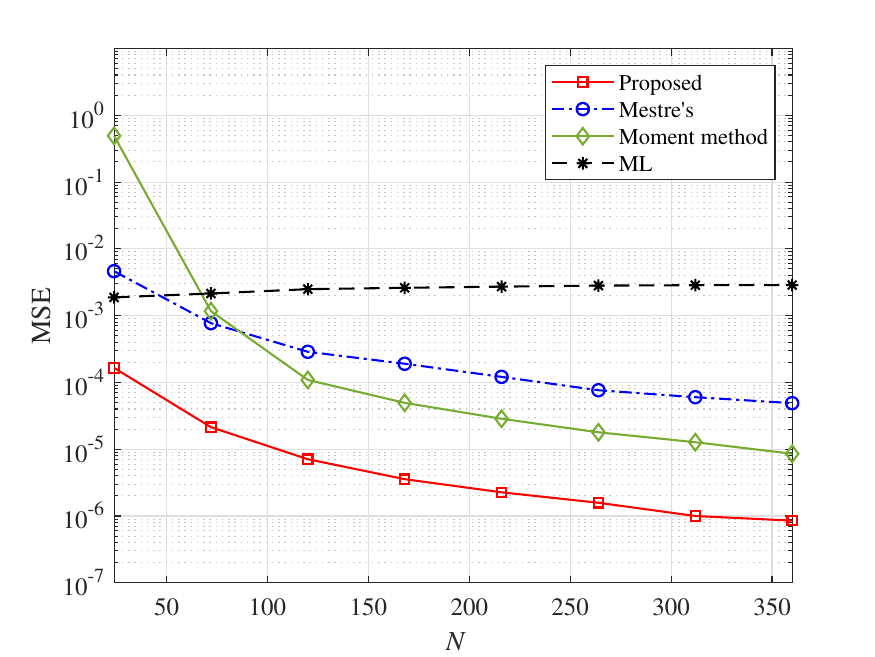}
	}\\
	\subfloat[Separability condition does not hold true]{
		\centering
		\includegraphics[trim=0.5cm 0 1cm 1cm, scale=0.4]{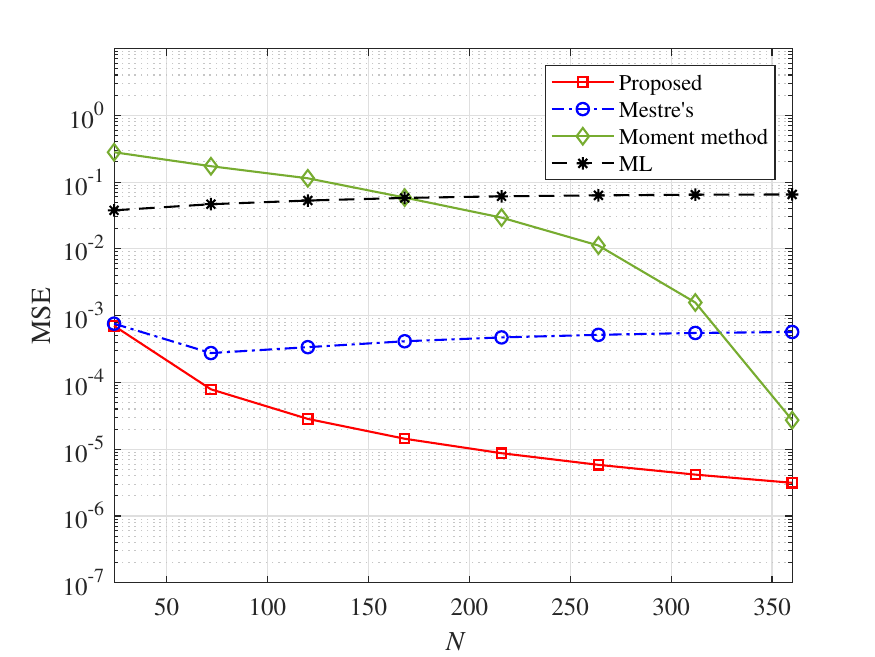}
	}\\
	\caption{Performance comparisons in estimating $g_1=\frac{1}{N}{\rm tr}(\mathbf{M}_N)$ in the case where the multiplicities of eigenvalues of $\mathbf{M}_N$ are unknown. The experiment parameters are identical to those in Figs. \ref{Performance comparisons when the separability condition is satisfied.} and \ref{Performance comparisons when the separability condition does not hold true.}.}
	\label{Performance comparisons in estimating trM in the case where the multiplicities of eigenvalues of M are unknown}
\end{figure}

In Table \ref{Execution Time}, we present comparisons of execution times between Mestre's estimator and the proposed estimator. The parameters are identical to those in the experiment in Fig. \ref{Performance comparisons when the separability condition does not hold true.}. For Mestre's estimator, the polynomial equation is solved utilizing the  \texttt{vpasolve} function in MATLAB.  Evidently, our estimator offers a substantial reduction in computational complexity compared to Mestre's estimator. Notably, the complexity of Mestre's estimator tends to increase extremely as $N,K$ increase.

\begin{table}[ht]
	\centering
	\caption{\textsc{Execution Time of the Proposed Estimator and Mestre's Estimator for One Realization (in Seconds).}}
	\label{Execution Time}
	\renewcommand{\arraystretch}{1.2}
	\resizebox{8cm}{!}{%
		\begin{tabular}{ccc}
			\toprule[1pt]
			\multicolumn{1}{c}{$N,K$}         & \multicolumn{1}{c}{Proposed estimator} & \multicolumn{1}{c}{Mestre's estimator} \\ \hline
			$N=30,K=80$   & 0.047s & 1.13s \\
			$N=120,K=320$ & 0.054s & 7.66s \\ 
			\multicolumn{1}{c}{$N=300,K=800$} & \multicolumn{1}{c}{0.066s}             & \multicolumn{1}{c}{73.41s}             \\ \bottomrule[1pt]
		\end{tabular}%
	}\vspace{-1em}
\end{table}
\subsection{Performance Comparisons with Unknown Multiplicities}
In practical applications, the multiplicities of eigenvalues of covariance/precision matrix are usually unknown and hence need to be estimated from the observations. In \cite{Yao2012Eigenvalue}, Yao \textit{et al}. proposed a moment-based estimator for the eigenvalues of covariance matrix, along with the corresponding multiplicities. It is worth mentioning that the moment method does not rely on the separability assumption and hence outperforms Mestre's estimator when the separability does not hold true (see \cite{Yao2012Eigenvalue} for details). 
	
In what follows, we will make a comparative study on the estimation performance of the proposed estimator, Mestre's estimator, moment method and ML estimator. We consider estimating $g_1=\frac{1}{N}{\rm tr}(\mathbf{M}_N)$ by using the following formula:
\begin{equation}\nonumber
	\hat{g}_1= \frac{1}{N}\sum_{i=1}^L\hat{N}_i\hat{\gamma}_i.
\end{equation}
Notice that different from the numerical examples presented in Section V-A, here we do not utilize the prior information about the multiplicities but to estimate them using the moment method. In other words, the estimated multiplicities $\hat{N}_i$ ($i=1,2,..,L$) are obtained by the moment method. 

The comparison results are presented in Fig. \ref{Performance comparisons in estimating trM in the case where the multiplicities of eigenvalues of M are unknown}. It can be seen that, regardless of whether the separability condition holds, the proposed estimator consistently outperforms Mestre's estimator, the moment method and ML estimator. This demonstrates the superiority of the proposed estimator, even in the scenario where the multiplicities are unknown. However, we want to emphasize that the proposed estimator lacks the capability to estimate the multiplicities. In practical applications, it may be beneficial to consider integrating the moment method with the proposed estimator. Specifically, one could initially use the moment method to estimate the unknown multiplicities and subsequently apply these estimated multiplicities to the proposed estimator in order to enhance the estimation performance.

\subsection{Accuracy of the CLT and Gaussian Approximation}
Finally, we conduct simulations to verify the results in Theorems \ref{Theorem Second-order asymptotic behavior} and \ref{Theorem the consistent estimator for Theta}. We consider the case where $N=60,K=120$ and the precision matrix has two different eigenvalues $(\gamma_1,\gamma_2)=(1,1/5)$ with respective multiplicities $N_1=40$ and $N_2=20$.
Fig. \ref{Comparisons between the histograms of estimated eigenvalues and the limiting spectral densities} depicts the histograms for each estimated eigenvalues and the corresponding limiting spectral densities based on Theorems \ref{Theorem Second-order asymptotic behavior} and \ref{Theorem the consistent estimator for Theta}. We observe that the histograms are in perfect agreement with the theoretical results, thus validating the accuracy of Theorems \ref{Theorem Second-order asymptotic behavior} and \ref{Theorem the consistent estimator for Theta}.
\begin{figure}[]
	\vspace{1em}
	\centering
	\subfloat[Fluctuation of $\hat{\gamma}_1$]{
		\centering
		\includegraphics[trim=1cm 0 1cm 1cm, scale=0.4]{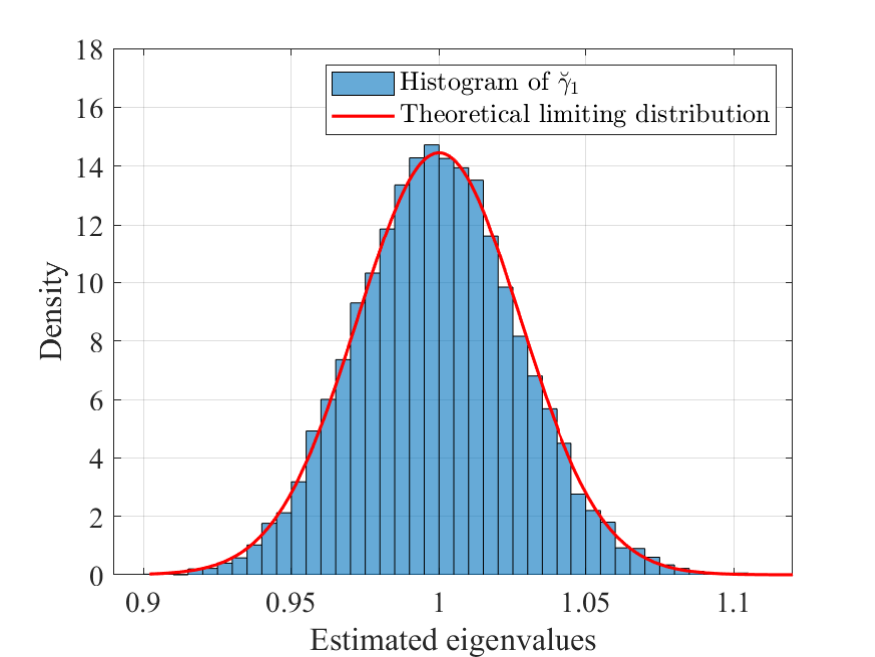}
	}\\
	\subfloat[Fluctuation of $\hat{\gamma}_2$]{
		\centering
		\includegraphics[trim=0.5cm 0 1cm 0.8cm, scale=0.4]{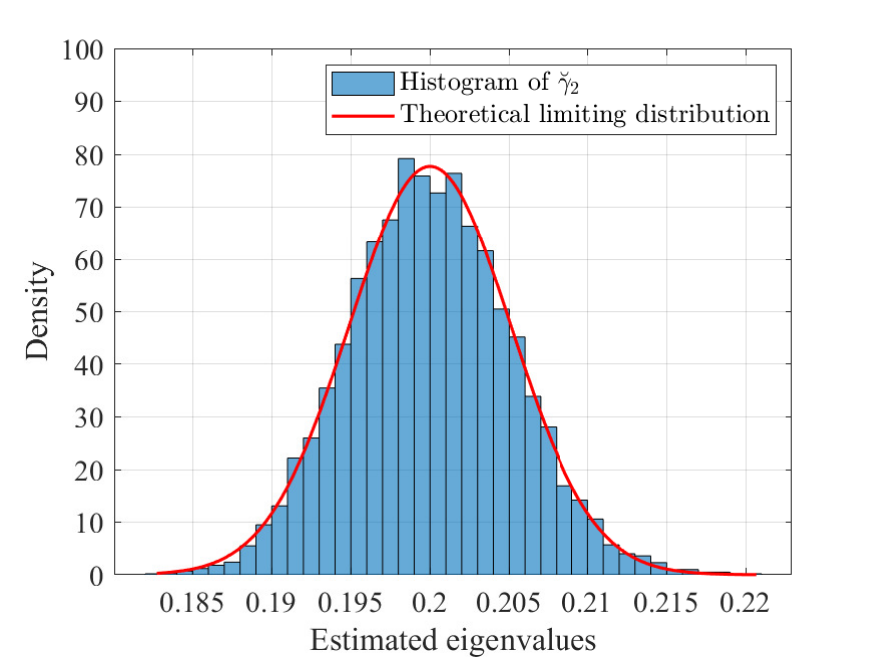}
	}
	\caption{Comparisons between the histograms of estimated eigenvalues and the limiting spectral densities based on Theorems \ref{Theorem Second-order asymptotic behavior} and \ref{Theorem the consistent estimator for Theta}.}
	\label{Comparisons between the histograms of estimated eigenvalues and the limiting spectral densities}\vspace{-1em}
\end{figure}

\section{Conclusion}
In this paper, we construct an improved estimator for eigenvalues of precision matrix, without inverting the eigenvalue estimator of covariance matrix. We prove the consistency of the proposed estimator under large dimensional regime. In order to obtain the asymptotic bias of the proposed estimator, we establish a theoretical result that characterizes the convergence rate of expected Stieltjes transform (with its derivative) of the spectra of SCM. This is a theoretical contribution of this paper.  Using this result, we prove that the bias of the proposed estimator is $O\left(\frac{1}{K^2}\right)$ term. We also highlight that using this result, one can also control the bias term of Mestre's improved estimator by following the same routine presented in this paper. We further  provide a CLT to describe  the asymptotic fluctuation of the proposed estimator. We also give an empirical but asymptotically accurate approximation of the covariance for the CLT. The simulation results show that the proposed estimator exhibits excellent performance, especially in the case where the spectral separation condition is not satisfied. In addition, the advantage of the proposed estimator lies in its high computational efficiency, making it a more viable choice for practical implementations.

We emphasize that, in this work, we restrict our analysis to the case $N/K<1$ (see assumption (A2)) to ensure the invertibility of the sample covariance matrix. However, in practical applications, the case $N/K>1$ may occur, where the sample covariance matrix becomes non-invertible. In future work, it would be interesting to explore the extension of our methodology to handle this case. A potential solution is to use the Moore-Penrose inverse of the sample covariance matrix, whose limiting spectral behavior udder large dimensional regime has been established by Bodnar, Dette, and Parolya \cite{Bodnar2016Spectral}. By applying their results and adapting the methodology developed in this paper, a similar estimator could be derived. 

\section*{Acknowledgment}
The authors would like to thank the editors and anonymous reviewers who pointed out a number of constructive improvements. The authors would also like to thank Ms. Yusi Fu and Dr. Xinzhong Hu for their assistance in validating part of the simulation results.
\bibliographystyle{IEEEtran}
\bibliography{IEEEexample}

\clearpage
\section*{Supplementary Material}
This document provides some additional proofs of results in the article ``Direct Estimation of Eigenvalues of Large Dimensional Precision Matrix”.
\section*{Appendix A\\Proof of Theorem 3}
Since we assume that $\mathbf{R}_N$ is an Hermitian matrix, it can be decomposed as $\mathbf{R}_N= \mathbf{U}^H\mathbf{\Lambda}\mathbf{U}$ where $\mathbf{\Lambda} = {\rm diag}(\underbrace{\lambda_1}_{N_1},\underbrace{\lambda_2}_{N_2},\cdots,\underbrace{\lambda_L}_{N_L})$ and $\mathbf{U}$ is a unitary matrix composed of eigenvectors of $\mathbf{R}_N$. Taking the unitary transformations $\bm{\mathcal{Y} }   = \mathbf{U}\mathbf{Y}$ and $\bm{\mathcal{X} } = \mathbf{U}\mathbf{X}$, we have $\bm{\mathcal{Y} }  = \mathbf{\Lambda}^{1/2}\bm{\mathcal{X} } $. By using $\bm{\mathcal{Y}}$, we obtain a new SCM defined by $\breve{\mathbf{R}}_N = \bm{\mathcal{Y} } \bm{\mathcal{Y} }^H$. It is easy to check that $\hat{\mathbf{R}}_N = \mathbf{U}^H\breve{\mathbf{R}}_N\mathbf{U}$, which implies that $\breve{\mathbf{R}}_N$ and $\hat{\mathbf{R}}_N$ have the same eigenvalues. Hence we have $b_{\hat{\mathbf{R}}}(z) = b_{\breve{\mathbf{R}}}(z) $, or equivalently $Eb_{\hat{\mathbf{R}}}(z) = Eb_{\breve{\mathbf{R}}}(z) $. Therefore, in order to prove Theorem 3, it is sufficient to prove 
\setcounter{equation}{87}
\begin{equation}\label{77}
	\left|Eb_{\breve{\mathbf{R}}}(z) - b_K(z)\right| \leqslant\frac{1}{K^2}P_9(|z|)P_{12}(|{\rm Im}z|^{-1}),
\end{equation}
and
\begin{equation}\label{78}
	\left|Eb_{\breve{\mathbf{R}}}'(z) - b'_K(z)\right| \leqslant\frac{1}{K^2}P_{15}(|z|)P_{20}(|{\rm Im}z|^{-1}).
\end{equation}

We remark that our proof relies on the ideas developed by Haagerup and Thorbjornsen \cite{Haagerup2005A}, and extensively used by Capitaine \textit{et al.} \cite{Capitaine2007Strong}, and Vallet, Loubaton and Mestre \cite{Vallet2012Improved} (see also \cite{Vallet2010ImprovedOnLine}).

We first establish \eqref{77}. We define  two matrices
\begin{equation}\label{key}
	\mathbf{Q}_N(z) = \left(\breve{\mathbf{R}}-z\mathbf{I}_N\right)^{-1} = \left(\bm{\mathcal{Y} }{\bm{\mathcal{Y} }^H}-z\mathbf{I}_N\right)^{-1}
\end{equation}
and
\begin{equation}\label{key}
	\tilde{\mathbf{Q}}_N(z) = \left(\bm{\mathcal{Y} }^H{\bm{\mathcal{Y} }}-z\mathbf{I}_K\right)^{-1}.
\end{equation}

In what follows, the subscript $N$ will be omitted when there is no danger of ambiguity. 

According to the definition of Stieltjes transform, we have
\begin{equation}\label{key}
	b_{\breve{\mathbf{R}}}(z) = \frac{1}{N}{\rm tr}{\mathbf{Q}(z)}.
\end{equation}

We also define the matrix
\begin{equation}\label{T(z)}
	\mathbf{T}(z) = -\frac{1}{z}\left(\tilde{\alpha}(z)\mathbf{\Lambda}+\mathbf{I}_N\right)^{-1}
\end{equation}
where
\begin{equation}\label{key}
	\tilde{\alpha}(z) = \alpha(z)- \frac{1-c_K}{z}
\end{equation}
and 
\begin{equation}\label{83}
	\alpha(z) = 	c_K{b}_K(z) = \frac{1}{K}{\rm tr}\mathbf{T}(z).
\end{equation}

We also define a matrix
\begin{equation}\label{V(z)}
	\mathbf{V}(z) = -\frac{1}{z}\left(\tilde{\beta}(z)\mathbf{\Lambda}+\mathbf{I}_N\right)^{-1}
\end{equation}
where
\begin{equation}\label{key}
	\tilde{\beta}(z) =\beta(z)- \frac{1-c_K}{z}
\end{equation}
and 
\begin{equation}\label{key}
	\beta(z)=c_KEb_{\breve{\mathbf{R}}}(z) =\frac{1}{K}E{\rm  tr}\mathbf{Q}(z)
\end{equation}
Apparently, $\tilde{\beta}(z) = \frac{1}{K}E{\rm tr}\tilde{\mathbf{Q}}(z)$. Then in order to prove \eqref{77}, it suffices to prove 
\begin{equation}\label{sup beta-alpha}
	|\beta(z)-\alpha(z)|=|\tilde{\beta}(z)-\tilde{\alpha}(z)| \leqslant\frac{1}{K^2}P_9(|z|)P_{12}(|{\rm Im}z|^{-1}).
\end{equation}

In what follows, we will denote  $\alpha(z)$,  $\beta(z)$, $\tilde{\alpha}(z)$ and $\tilde{\beta}(z)$ by $\alpha$,  $\beta$, $\tilde{\alpha}$ and $\tilde{\beta}$ respectively for the sake of clarity of presentation.

We express $\tilde{\beta}-\tilde{\alpha}$ by
\begin{equation}\label{beta-alpha}
	\begin{aligned}
		\tilde{\beta}-\tilde{\alpha} &= E\left(\frac{1}{K}{\rm tr}\mathbf{Q}(z)\right) - \frac{1}{K}{\rm tr}\mathbf{T}(z)
		\\& = \frac{1}{K}{\rm tr}\left(\mathbf{V}(z) - \mathbf{T}(z)\right) + \epsilon(z)
	\end{aligned}
\end{equation}
where
\begin{equation}\label{key}
	\epsilon(z) = \frac{1}{K}{\rm tr}\left(E\mathbf{Q}(z) - \mathbf{V}(z)\right).
\end{equation}

Notice that 
\begin{equation}\label{V-T}
	\begin{aligned}
		\mathbf{V}(z) - \mathbf{T}(z) &=\mathbf{V}(z)\left(\mathbf{T}^{-1}(z) - \mathbf{V}^{-1}(z)\right) \mathbf{T}(z)
		\\& = z(\tilde{\beta}-\tilde{\alpha})\mathbf{V}(z)\mathbf{\Lambda}\mathbf{T}(z).
	\end{aligned}
\end{equation}

Substituting \eqref{V-T} into \eqref{beta-alpha} yields
\begin{equation}\label{103}
	\tilde{\beta}-\tilde{\alpha} = (\tilde{\beta}-\tilde{\alpha})\frac{z}{K}{\rm tr}(\mathbf{V}(z)\mathbf{\Lambda}\mathbf{T}(z) )+ \epsilon(z)
\end{equation}
which is equivalent to
\begin{equation}\label{beta-alpha 2}
	\tilde{\beta}-\tilde{\alpha} = \frac{\epsilon(z)}{1-\frac{z}{K}{\rm tr}(\mathbf{V}(z)\mathbf{\Lambda}\mathbf{T}(z))}.
\end{equation}

In what follows,  we will follow the tricks in \cite{Vallet2012Improved} to find the bound of \eqref{beta-alpha 2}.  From \eqref{83}, we have
\begin{equation}\label{Im alpha}
	{\rm Im}(\alpha) = \frac{1}{K}{\rm tr}\left({\rm Im}\mathbf{T}(z)\right)
\end{equation}

Notice that 
\begin{equation}\label{Im T}
	\begin{aligned}
		{\rm Im}\mathbf{T}_N(z) &= \frac{1}{2\mathsf{i}}\left(\mathbf{T}(z)-\mathbf{T}(z)^{H}\right)
		\\& = \frac{1}{2\mathsf{i}}\left(\mathbf{T}(z)\left(\mathbf{T}(z)^{-H} - \mathbf{T}^{-1}(z)\right)\mathbf{T}(z)^H\right)
		\\& = \mathbf{T}(z)\mathbf{\Lambda}\mathbf{T}(z)^H{\rm Im}(z\tilde{\alpha}) + \mathbf{T}(z)\mathbf{T}(z)^H{\rm Im}(z).
	\end{aligned}
\end{equation}

Using \eqref{Im T} in \eqref{Im alpha} and noticing that ${\rm Im}(z\tilde{\alpha}) = {\rm Im}(z\alpha)$, we have
\begin{equation}\label{Im alpha2}
	{\rm Im}({\alpha}) = \frac{{\rm Im}(z{\alpha})}{K}{\rm tr}(\mathbf{T}(z)\mathbf{\Lambda}\mathbf{T}(z)^H) +\frac{{\rm Im}(z)}{K}{\rm tr}(\mathbf{T}(z)\mathbf{T}(z)^H).
\end{equation}

Likewise, we have
\begin{equation}\label{key}
	{\rm Im}(z\alpha) = \frac{1}{K}{\rm tr}\left({\rm Im}\left(z\mathbf{T}(z)\right)\right).
\end{equation}

By using the facts that
\begin{equation}\label{Im zT}
	\begin{aligned}
		&{\rm Im}\left(z\mathbf{T}(z)\right) 
		\\&= \frac{1}{2\mathsf{i}}\left(z\mathbf{T}(z) - \left(z\mathbf{T}(z)\right)^H\right)
		\\&=\frac{1}{2\mathsf{i}}\left(z\mathbf{T}(z)\left(\left(z\mathbf{T}(z)\right)^{-H} - \left(z\mathbf{T}(z)\right)^{-1} \right) \left(z\mathbf{T}(z)\right)^H  \right)
		\\& =|z|^2 \mathbf{T}(z)\mathbf{\Lambda}\mathbf{T}(z)^H{\rm Im}(\tilde{\alpha})
	\end{aligned}
\end{equation}
and 
\begin{equation}\label{key}
	\begin{aligned}
		{\rm Im}(\tilde{\alpha}) &= {\rm Im}(\alpha) - (1-c_K){\rm Im}\left(\frac{1}{z}\right)
		\\& = {\rm Im}(\alpha) +\frac{1-c_K}{|z|^2}{\rm Im}(z),
	\end{aligned}
\end{equation}
we have
\begin{equation}\label{Im zalpha}
	\begin{aligned}
		{\rm Im}(z\alpha) &= \frac{{\rm Im}(\alpha)}{K}|z|^2{\rm tr}(\mathbf{T}(z)\mathbf{\Lambda}\mathbf{T}(z)^H) 
		\\&\quad+ \frac{{\rm Im}(z)}{K}(1-c_K){\rm tr}(\mathbf{T}(z)\mathbf{\Lambda}\mathbf{T}(z)^H).
	\end{aligned}
\end{equation}

Hence from \eqref{Im alpha2} and \eqref{Im zalpha}, we obtain the following linear system
\begin{equation}\label{key}
	\begin{pmatrix}
		{\rm Im}(\alpha)\\{\rm Im}(z\alpha)
	\end{pmatrix} = \mathbf{G}_1\begin{pmatrix}
		{\rm Im}(\alpha)\\{\rm Im}(z\alpha)
	\end{pmatrix} + \begin{pmatrix}
		v_1\\(1-c_K)u_1
	\end{pmatrix}{\rm Im}(z)
\end{equation}
where 
\begin{equation}\label{key}
	\mathbf{G}_1	=\begin{pmatrix}
		0,&u_1\\|z|^2u_1,&0
	\end{pmatrix}
\end{equation}
with 
\begin{align}\label{key}
	&u_1=\frac{1}{K}{\rm tr}(\mathbf{T}(z)\mathbf{\Lambda}\mathbf{T}(z)^H)
	\\& v_1 = \frac{1}{K}{\rm tr}(\mathbf{T}(z)\mathbf{T}(z)^H).
\end{align}

We find that ${\rm det}(\mathbf{I}-\mathbf{G}_1) = 1-|z|^2u_1^2$ coincides with
\begin{equation}\label{key}
	\begin{aligned}
		{\rm det}(\mathbf{I}-\mathbf{G}_1) &= \left((1-c_K)u_1^2+v_1\right)\frac{{\rm Im}z}{{\rm Im}\alpha}
		\\&\geqslant	v_1\frac{{\rm Im}z}{{\rm Im}\alpha}.
	\end{aligned}
\end{equation}
Since $\alpha = c_K{b}_K(z)$ with $b_K(z)$ being a Stieltjes transform of a density, we have ${\rm Im}\alpha\leqslant \frac{c_K}{{\rm Im}z}$. Then $\frac{{\rm Im}z}{{\rm Im}\alpha}\geqslant\frac{({\rm Im}z)^2}{c_K}$, which implies that
\begin{equation}\label{key}
	{\rm det}(\mathbf{I}-\mathbf{G}_1)\geqslant v_1\frac{({\rm Im}z)^2}{c_K}.
\end{equation}

In order to obtain the lower bound of $v_1$, we remark that $\frac{1}{N}{\rm tr}(\mathbf{T}(z)\mathbf{T}(z)^H)\geqslant|{b}_K(z)|^2$ by the Jensen inequality. Hence $v_1\geqslant c_K|{b}_K(z)|^2\geqslant c_K|{\rm Im}(b_K(z))|^2$. Further, ${\rm Im}(b_K(z))$ can be written as
\begin{equation}\label{key}
	{\rm Im}(b_K(z)) = {\rm Im}(z)\int_{\mathbb{R}}\frac{dF_K(\lambda)}{|\lambda-z|^2}.
\end{equation}

There exists a constant $\eta>0$ for which
\begin{equation}\label{107}
	F_K([0,\eta])>1/2
\end{equation}
for each integer $K$. In addition, it holds that
\begin{equation}\label{key}
	\int_{\mathbb{R}}\frac{dF_K(\lambda)}{|\lambda-z|^2} > \int_{0}^\eta\frac{dF_K(\lambda)}{|\lambda-z|^2}>\frac{F_K([0,\eta])}{2(\eta^2+|z|^2)}>\frac{1}{4(\eta^2+|z|^2)}
\end{equation}
Hence $v_1>\frac{c_K({\rm Im}z)^2}{16(\eta^2+|z|^2)^2}$. Finally we get ${\rm det}(\mathbf{I}-\mathbf{G}_1)>\frac{({\rm Im}z)^4}{16(\eta^2+|z|^2)^2}$, which is equivalent to
\begin{equation}\label{bound of u1}
	|u_1|<\frac{\left(1-\frac{({\rm Im}z)^4}{16(\eta^2+|z|^2)^2}\right)^{1/2}}{{|z|}}.
\end{equation}

Similarly, by noticing that $\beta$ can be expressed as
\begin{equation}\label{key}
	\beta = \frac{1}{K}{\rm tr}({\rm Im}\mathbf{V}(z)) +\epsilon(z),
\end{equation}
we have
\begin{equation}\label{Im beta}
	{\rm Im}(\beta) = \frac{1}{K}{\rm tr}\left({\rm Im}\mathbf{V}(z)\right) + {\rm Im}(\epsilon(z))
\end{equation}
and
\begin{equation}\label{key}
	{\rm Im}(z\beta) = \frac{1}{K}{\rm tr}\left({\rm Im}(z\mathbf{V}(z))\right) + {\rm Im}(z\epsilon(z)).
\end{equation}

Using the procedures similar to \eqref{Im T} and \eqref{Im zT} yields
\begin{equation}\label{key}
	\begin{aligned}
		{\rm Im}\mathbf{V}(z) 
		= \mathbf{V}(z)\mathbf{\Lambda}\mathbf{V}(z)^H{\rm Im}(z\tilde{\beta})+\mathbf{V}(z)\mathbf{V}(z)^H{\rm Im}(z)
	\end{aligned}
\end{equation}
and
\begin{equation}\label{key}
	\begin{aligned}
		{\rm Im}\left(z\mathbf{V}(z)\right) = |z|^2\mathbf{V}(z)\mathbf{\Lambda}\mathbf{V}(z)^H{\rm Im}(\tilde{\beta}).
	\end{aligned}
\end{equation}

Meanwhile, remark that ${\rm Im}(z\tilde{\beta}) ={\rm Im}(z\beta)$ and 
\begin{equation}\label{Im beta tilde}
	{\rm Im}(\tilde{\beta}) = {\rm Im}(\beta)+\frac{1-c_K}{|z|^2}{\rm Im}(z).
\end{equation}

From \eqref{Im beta}-\eqref{Im beta tilde}, we obtain the following linear system
\begin{align}\label{key}
	\begin{aligned}
		\begin{pmatrix}
			{\rm Im}(\beta)\\{\rm Im}(z\beta)
		\end{pmatrix}= \mathbf{G}_2\begin{pmatrix}
			{\rm Im}(\beta)\\{\rm Im}(z\beta)
		\end{pmatrix} &+ \begin{pmatrix}
			v_2\\(1-c_K)u_2
		\end{pmatrix}{\rm Im}(z) 
		\\&+ \begin{pmatrix}
			{\rm Im}\epsilon(z)\\{\rm Im}(z\epsilon(z))
		\end{pmatrix}
	\end{aligned}
\end{align}
where
\begin{equation}\label{key}
	\mathbf{G}_2 = \begin{pmatrix}
		&	0,&u_2\\&|z|^2u_2,&0
	\end{pmatrix}
\end{equation}
with
\begin{align}\label{key}
	&u_2 = \frac{1}{K}{\rm tr}(\mathbf{V}(z)\mathbf{\Lambda}\mathbf{V}(z)^H),\\
	&v_2 = \frac{1}{K}{\rm tr}(\mathbf{V}(z)\mathbf{V}(z)^H) .
\end{align}

We can calculate
\begin{equation}\label{key}
	{\rm det}(\mathbf{I}-\mathbf{G}_2) = 1-|z|^2u_2^2
\end{equation}
which coincides with
\begin{equation}\label{key}
	\begin{aligned}
		{\rm det}(\mathbf{I}-\mathbf{G}_2)  &= \left((1-c_K)u_2^2+v_2\right)\frac{{\rm Im}z}{{\rm Im}\beta} + \frac{{\rm Im}\epsilon(z)}{{\rm Im}\beta}+\frac{{\rm Im}(z\epsilon(z))}{{\rm Im}\beta}
		\\& >v_2\frac{{\rm Im}z}{{\rm Im}\beta} + \frac{{\rm Im}\epsilon(z)}{{\rm Im}\beta}+\frac{{\rm Im}(z\epsilon(z))}{{\rm Im}\beta}
	\end{aligned}
\end{equation}

Since ${\rm Im}\beta\leqslant \frac{c_K}{{\rm Im}z}$, we have
\begin{equation}\label{det I-D2}
	\begin{aligned}
		{\rm det}(\mathbf{I}-\mathbf{G}_2)   &>v_2\frac{({\rm Im}z)^2}{c_K} + \frac{{\rm Im}\epsilon(z)}{{\rm Im}\beta}+\frac{{\rm Im}(z\epsilon(z))}{{\rm Im}\beta}
		\\& >v_2\frac{({\rm Im}z)^2}{c_K}  -\frac{|\epsilon(z)| }{{\rm Im}\beta}-\frac{|z||\epsilon(z)|}{{\rm Im}\beta} 
	\end{aligned}
\end{equation}

Note that $v_2 = c_K\frac{1}{N}{\rm tr}(\mathbf{V}(z)\mathbf{V}(z)^H)\geqslant c_K|\frac{1}{N}{\rm tr}\mathbf{V}(z)|^2$. Since $\frac{1}{N}{\rm tr}\mathbf{V}(z)$ can be regarded as the Stieltjes transform of a probability measure $\hat{F}_K$. Using the same arguments as before, the sequence $(\hat{F}_K-F_K)\to 0$ weakly as $K\to\infty$. Thus given $\eta$ defined in \eqref{107}, there exists an integer $K_1$ such that 
\begin{equation}\label{key}
	{\hat{F}}_K([0,\eta])>\frac{1}{4}
\end{equation}
for  each $K>K_1$. Hence we have
\begin{equation}\label{lower bound of v2}
	\begin{aligned}
		v_2> c_K\left|{\rm Im}z\int_{\mathbb{R}}\frac{d\hat{F}_K(\lambda)}{|\lambda-z|^2} \right|^2&>c_K\left|{\rm Im}z\int_{0}^\eta\frac{d\hat{F}_K(\lambda)}{|\lambda-z|^2}\right|^2
		\\&>c_K\left|\frac{{\rm Im}z}{2(\eta^2+|z|^2)}\hat{F}_K([0,\eta])\right|^2
		\\& = \frac{c_K({\rm Im}z)^2}{64(\eta^2+|z|^2)^2}.
	\end{aligned}
\end{equation}

Using \eqref{lower bound of v2} in \eqref{det I-D2} gives
\begin{equation}\label{125}
	{\rm det}(\mathbf{I}-\mathbf{G}_2)>\frac{({\rm Im}z)^4}{64(\eta^2+|z|^2)^2}-\frac{|\epsilon(z)|}{{\rm Im}\beta}  - \frac{|z||\epsilon(z)|}{{\rm Im}\beta} 
\end{equation}
for sufficiently large $K>K_1$.

Since $\beta = c_K\frac{1}{N}{\rm tr}E\mathbf{Q}(z)$, and $\frac{1}{N}{\rm tr}E\mathbf{Q}(z)$ can be regarded as the Stieltjes transform of a probability measure $\bar{F}_K$. It is shown in \cite{Bai2004CLT} that $|\frac{1}{N}{\rm tr}E\mathbf{Q}(z) - {b}_K(z)|\to 0$ for all $z\in \mathbb{C}\setminus\mathbb{R}^+$. Therefore, the measure $\bar{F}_K$ converges weakly to $F_K$ for large $K$. This implies there exists an integer $K_2\geqslant K_1$ such that for each large $K>K_2$
\begin{equation}\label{key}
	\bar{F}_K([0,\eta])>\frac{1}{4}
\end{equation}
with $\eta$ defined in \eqref{107}. Thus we have
\begin{equation}\label{key}
	{\rm Im}\beta > c_K{\rm Im}z\int_{\mathbb{R}}\frac{d\bar{F}_K(\lambda)}{|\lambda-z|^2}>\frac{c_K{\rm Im}z}{8(\eta^2+|z|^2)}.
\end{equation}

Using this inequality in \eqref{125} gives
\begin{equation}\label{128}
	\begin{aligned}
		{\rm det}(\mathbf{I}-\mathbf{G}_2)&>\frac{({\rm Im}z)^4}{64(\eta^2+|z|^2)^2}-\frac{8(\eta^2+|z|^2)}{c_K{\rm Im}z} |\epsilon(z)| \\&\quad- \frac{8(\eta^2+|z|^2)}{c_K{\rm Im}z}|z||\epsilon(z)|
	\end{aligned}
\end{equation}

In order to obtain the lower bound of \eqref{128}, we need to use the following Lemma.
\setcounter{lemma}{3} 
\begin{mylemma}\label{convergence rate of epsilon}
	For $z\in\mathbb{C}\setminus\mathbb{R}$, it holds true that
	\begin{equation}\label{key}
		|\epsilon(z)| \leqslant \frac{1}{K^2}P_2(|z|)P_6(|{\rm Im}z|^{-1})
	\end{equation}
\end{mylemma}
\begin{proof}
	See Appendix B in Supplementary Material (pages 4-6).
\end{proof}

Using this lemma, \eqref{128} can be arranged as
\begin{equation}\label{key}
	{\rm det}(\mathbf{I}-\mathbf{G}_2)>\frac{({\rm Im}z)^4}{64(\eta^2+|z|^2)^2}\left(1-\frac{1}{K^2}P_{9}(|z|)P_{11}(|{\rm Im}z|^{-1})\right).
\end{equation}
If we denote the subset $\varpi$ by
\begin{equation}\label{set varpi}
	\varpi = \left\{z\in\mathbb{C}^+:1-\frac{1}{K^2}P_{9}(|z|)P_{11}(|{\rm Im}z|^{-1})>0\right\},
\end{equation}
then for $z\in\varpi$ we have ${\rm det}(\mathbf{I}-\mathbf{G}_2) = 1-|z|^2u_2^2>0$ for all large $K>K_2$, which is equivalent to
\begin{equation}\label{bound of u2}
	|u_2|<\frac{1}{|z|}.
\end{equation}

Hence for $z\in \varpi$, from \eqref{bound of u1} and \eqref{bound of u2}, we have 
\begin{equation}\label{Lower bound of 1-z/Ktr}
	\begin{aligned}
		|1-\frac{z}{K}{\rm tr}(\mathbf{V}(z)\mathbf{\Lambda}\mathbf{T}(z))|&\geqslant|1-z\sqrt{u_1}\sqrt{u_2}|
		\\&\geqslant1-|z||\sqrt{u_1}||\sqrt{u_2}|
		\\&\geqslant 1-\left(1-\frac{({\rm Im}z)^4}{16(\eta^2+|z|^2)^2}\right)^{1/4}
		\\&\geqslant \frac{({\rm Im}z)^4}{64(\eta^2+|z|^2)^2}
	\end{aligned}
\end{equation}

Using \eqref{Lower bound of 1-z/Ktr} in \eqref{beta-alpha 2} and applying Lemma \ref{convergence rate of epsilon}, we obtain
\begin{equation}\label{sup beta-alpha 1}
	\begin{aligned}
		|\tilde{\beta}-\tilde{\alpha}|\leqslant\frac{64(\eta^2+|z|^2)^2}{({\rm Im}z)^4}|\epsilon|
		&	\leqslant \frac{1}{K^2}P_6(|z|)P_{10}(|{\rm Im}z|^{-1})
		\\&\leqslant \frac{1}{K^2}P_{9}(|z|)P_{12}(|{\rm Im}z|^{-1})
	\end{aligned}
\end{equation}
for $z\in \varpi$.

As for $z\in\varpi^c = \{z\in\mathbb{C}^+:1-\frac{1}{K^2}P_{9}(|z|)P_{11}(|{\rm Im}z|^{-1})\leqslant0\}$, we have 
\begin{equation}\label{sup beta-alpha 2}
	\begin{aligned}
		|\tilde{\beta}-\tilde{\alpha}|\leqslant\frac{2c_K}{|{\rm Im}z|}
		&\leqslant \frac{2c_K}{|{\rm Im}z|}\frac{1}{K^2}P_{9}(|z|)P_{11}(|{\rm Im}z|^{-1})
		\\&\leqslant \frac{1}{K^2}P_{9}(|z|)P_{12}(|{\rm Im}z|^{-1})
	\end{aligned}
\end{equation}

Combining \eqref{sup beta-alpha 1} and \eqref{sup beta-alpha 2}, we obtain \eqref{sup beta-alpha}. Now we completes the proof of \eqref{77}.

Next we establish \eqref{78}. Notice that $	Eb_{\breve{\mathbf{R}}}'(z) - b'_K(z) = \tilde{\beta}'-\tilde{\alpha}'$. Hence in order to prove \eqref{78}, it suffices to prove 
\begin{equation}\label{key}
	|\tilde{\beta}'-\tilde{\alpha}'|\leqslant \frac{1}{K^2}P_{15}(|z|)P_{20}(|{\rm Im}z|^{-1}).
\end{equation}

Taking the derivative with respect to $z$ on both sides of \eqref{103} gives
\begin{equation}\label{tilde beta - alpha}
	\begin{aligned}
		&	\tilde{\beta}'-\tilde{\alpha}'= \\&\qquad (\tilde{\beta}'-\tilde{\alpha}')\frac{z}{K}{\rm tr}\left(\mathbf{V}(z)\mathbf{\Lambda}\mathbf{T}(z)\right) 
		\\&\quad+(\tilde{\beta}-\tilde{\alpha})\frac{1}{K}{\rm tr}\left(\mathbf{V}(z)\mathbf{\Lambda}\mathbf{T}(z)+z\mathbf{V}'(z)\mathbf{\Lambda}\mathbf{T}(z)+z\mathbf{V}(z)\mathbf{\Lambda}\mathbf{T}'(z)\right)
		\\&\quad+\epsilon'(z)
	\end{aligned}
\end{equation}
where
\begin{equation}\label{key}
	\epsilon'(z) = 	\frac{1}{K}{\rm tr}\left(E\mathbf{Q}'(z) - \mathbf{V}'(z)\right).
\end{equation}
Taking derivatives with respect to $z$ of \eqref{T(z)} and \eqref{V(z)}, we have
\begin{equation}\label{T'(z)}
	\mathbf{T}'(z) = z\tilde{\alpha}'\mathbf{T}(z)\mathbf{\Lambda}\mathbf{T}(z) - \frac{1}{z}\mathbf{T}(z)
\end{equation}
and
\begin{equation}\label{V'(z)}
	\mathbf{V}'(z) = z\tilde{\beta}'\mathbf{V}(z)\mathbf{\Lambda}\mathbf{V}(z) - \frac{1}{z}\mathbf{V}(z)
\end{equation}

Substituting \eqref{T'(z)} and \eqref{V'(z)} into \eqref{tilde beta - alpha}, we have
\begin{equation}\label{key}
	\begin{aligned}
		\tilde{\beta}'-\tilde{\alpha}' &= (\tilde{\beta}'-\tilde{\alpha}')\frac{z}{K}{\rm tr}\left(\mathbf{V}(z)\mathbf{\Lambda}\mathbf{T}(z)\right) 
		\\&\quad+(\tilde{\beta}-\tilde{\alpha})\frac{1}{K}{\rm tr}\left[z^2\tilde{\alpha}'\mathbf{T}(z)\mathbf{\Lambda}\mathbf{V}(z)\mathbf{\Lambda}\mathbf{T}(z)\right.
		\\&\left.\qquad\qquad\qquad+z^2\tilde{\beta}'\mathbf{V}(z)\mathbf{\Lambda}\mathbf{T}(z)\mathbf{\Lambda}\mathbf{V}(z) - \mathbf{V}(z)\mathbf{\Lambda}\mathbf{T}(z)\right]
		\\&\quad+\epsilon'(z).
	\end{aligned}
\end{equation}

Hence we have
\begin{equation}\label{definiyion of beta'-alpha'}
	\tilde{\beta}'-\tilde{\alpha}' = \frac{u_4(\tilde{\beta}-\tilde{\alpha})+\epsilon'(z)}{1-\frac{z}{K}{\rm tr}(\mathbf{V}(z)\mathbf{\Lambda}\mathbf{T}(z))}
\end{equation}
where
\begin{equation}\label{key}
	\small	\begin{aligned}
		u_4 &= \frac{z^2}{K}\tilde{\alpha}'{\rm tr}\left[\mathbf{T}(z)\mathbf{\Lambda}\mathbf{V}(z)\mathbf{\Lambda}\mathbf{T}(z)\right]+\frac{z^2}{K}\tilde{\beta}'{\rm tr}\left[\mathbf{V}(z)\mathbf{\Lambda}\mathbf{T}(z)\mathbf{\Lambda}\mathbf{V}(z) \right]
		\\&\quad- \frac{1}{K}{\rm tr}\left[\mathbf{V}(z)\mathbf{\Lambda}\mathbf{T}(z)\right].
	\end{aligned}
\end{equation}

Using the facts that $|\tilde{\alpha}'|\leqslant\frac{c_K}{|{\rm Im}z|^2}$, $|\tilde{\beta}'|\leqslant\frac{c_K}{|{\rm Im}z|^2}$, 
\begin{equation}\label{key}
	\begin{aligned}
		\left|	\frac{1}{K}{\rm tr}\left(\mathbf{T}(z)\mathbf{\Lambda}\mathbf{V}(z)\mathbf{\Lambda}\mathbf{T}(z)\right)\right|
		&\leqslant\frac{1}{K}{\rm tr}\left(\mathbf{T}(z)\mathbf{\Lambda}\mathbf{V}(z)\mathbf{T}(z)^H\mathbf{\Lambda}^H\right)
		\\& \leqslant c_K{\Vert\mathbf{V}(z)\Vert}\Vert\mathbf{\Lambda}\Vert^2\Vert\mathbf{T}(z)\Vert
		\\&\leqslant c_K\frac{\lambda_{\rm max}^2}{|{\rm Im}z|^2},
	\end{aligned}
\end{equation}
\begin{equation}\label{key}
	\begin{aligned}
		\left|\frac{1}{K}{\rm tr}\left(\mathbf{V}(z)\mathbf{\Lambda}\mathbf{T}(z)\mathbf{\Lambda}\mathbf{V}(z)\right)\right|
		&\leqslant\frac{1}{K}{\rm tr}\left(\mathbf{V}(z)\mathbf{\Lambda}\mathbf{T}(z)\mathbf{V}(z)^H\mathbf{\Lambda}^H\right)
		\\& \leqslant c_K{\Vert\mathbf{T}(z)\Vert}\Vert\mathbf{\Lambda}\Vert^2\Vert\mathbf{V}(z)\Vert
		\\&\leqslant c_K\frac{\lambda_{\rm max}^2}{|{\rm Im}z|^2},
	\end{aligned}
\end{equation}
and
\begin{equation}\label{key}
	\left|\frac{1}{K}{\rm tr}\left(\mathbf{V}(z)\mathbf{\Lambda}\mathbf{T}(z)\right) \right|\leqslant c_K\frac{\lambda_{\rm max}}{|{\rm Im}z|},
\end{equation}
we have
\begin{equation}\label{217}
	\begin{aligned}
		|u_4|&\leqslant 2c_K^2|z|^2\frac{\lambda_{\rm max}^2}{|{\rm Im}z|^4} + c_K\frac{\lambda_{\rm max}}{|{\rm Im}z|}
		\\& \leqslant (2c_K^2\lambda_{\rm max}^2|z|^2+c_K\lambda_{\rm max})\left(|{\rm Im}z|^{-4}+|{\rm Im}z|^{-1}\right)
		\\&\leqslant P_2(|z|)P_4(|{\rm Im}z|^{-1})
	\end{aligned}
\end{equation}

In addition,  we need the following lemma to control the bound of $\epsilon'(z)$.
\begin{mylemma}\label{lemma bound of epsilon'(z)}
	For $z\in\mathbb{C}\setminus\mathbb{R}$, it holds true that
	\begin{equation}\label{bound of epsilon'(z)}
		|\epsilon'(z)| \leqslant \frac{1}{K^2}P_3(|z|)P_{11}(|{\rm Im}z|^{-1}).
	\end{equation}
\end{mylemma}
\begin{proof}
	See Appendix C in Supplementary Material (pages 6-8).
\end{proof}

For $z\in\varpi$ (the set $\varpi$ is defined in \eqref{set varpi}), by using \eqref{sup beta-alpha}, \eqref{Lower bound of 1-z/Ktr},  \eqref{217} and \eqref{bound of epsilon'(z)} in \eqref{definiyion of beta'-alpha'}, we obtain 
\begin{equation}\label{key}
	|\tilde{\beta}'-\tilde{\alpha}'|\leqslant \frac{1}{K^2}P_{15}(|z|)P_{20}(|{\rm Im}z|^{-1})
\end{equation}

For $z\in \varpi^c$,
\begin{equation}\label{key}
	\begin{aligned}
		|\tilde{\beta}'-\tilde{\alpha}'|\leqslant 2\frac{c_K}{|{\rm Im}z|^2}&\leqslant  \frac{1}{K^2}P_9(|z|)P_{13}(|{\rm Im}z|^{-1})\\&
		\leqslant  \frac{1}{K^2}P_{15}(|z|)P_{20}(|{\rm Im}z|^{-1})
	\end{aligned}
\end{equation}
which completes the proof of \eqref{78}.

\section*{Appendix B\\Proof of Lemma \ref{convergence rate of epsilon}}
We here follow the method developed by Dumont, Hachem and Loubaton \cite{Dumont2010On}. We note that $\bm{\mathcal{Y}}=\mathbf{\Lambda}^{1/2}\bm{\mathcal{X}}$ is a matrix model with a separable variance profile. Before proceeding, we introduce the following lemma, which is useful to control the upper bound of the variances.
\begin{lemma}\label{Lemma inequality}
	For  $z\in\mathbb{C}\setminus\mathbb{R}$, and deterministic matrices $\mathbf{C}_1$, $\mathbf{C}_2$, $\mathbf{D}_1\in\mathbb{C}^{N\times N}$ and $\mathbf{D}_2\in\mathbb{C}^{K\times K}$  with uniformly bounded spectral norms, it holds true that
	\begin{align}
		&{\rm Var}\left(\frac{1}{K}{\rm tr}\left(\mathbf{C}_1\mathbf{Q}(z)\mathbf{C}_2\right)\right)\leqslant \frac{1}{K^2}P_1(|z|)P_4(|{\rm Im}z|^{-1})\label{Var0},
		\\&	{\rm Var}\left(\frac{1}{K}{\rm tr}\tilde{\mathbf{Q}}(z)\right)\leqslant \frac{1}{K^2}P_1(|z|)P_4(|{\rm Im}z|^{-1})\label{Var1},
		\\& {\rm Var}\left(\frac{1}{K}{\rm tr}\left(\tilde{\mathbf{Q}}(z)\bm{\mathcal{Y}}^H\mathbf{D}_1\bm{\mathcal{Y}}\mathbf{D}_2\right)\right)\leqslant \frac{1}{K^2}P_3(|z|)P_4(|{\rm Im}z|^{-1})\label{Var2}.
	\end{align}
\end{lemma} 
\begin{proof}
	See Appendix D in Supplementary Material (pages 8-9).
\end{proof}

We define scalars
\begin{equation}\label{138}
	\begin{aligned}
		&\delta(z) =\frac{1}{K} E{\rm tr}\left(\mathbf{\Lambda}\mathbf{Q}(z)\right)
		\\&\tilde{\delta} (z)= \frac{1}{K}E{\rm tr}\tilde{\mathbf{Q}}(z)
	\end{aligned}
\end{equation}
and random matrices
\begin{equation}\label{key}
	\begin{aligned}
		&\mathbf{S}(z) = -\frac{1}{z}\left(\mathbf{I}_N+\tilde{\delta}(z)\mathbf{\Lambda}\right)^{-1}
		\\&\tilde{\mathbf{S}}(z) = -\frac{\mathbf{I}_K}{z(1+\delta(z))}
	\end{aligned}
\end{equation}

Furthermore we define two scalar parameters
\begin{equation}\label{139}
	\begin{aligned}
		&\tau(z) = -\frac{1}{zK}{\rm tr}\left(\mathbf{\Lambda}(\mathbf{I}_N+\tilde{\delta}(z)\mathbf{\Lambda})^{-1}\right)
		\\& \tilde{\tau}(z) = -\frac{1}{z(1+\delta(z))}
	\end{aligned}
\end{equation}
and two matrices
\begin{equation}\label{140}
	\begin{aligned}
		&\mathbf{\Delta}(z) = \frac{1}{1+\delta(z)}E\left(  \overset{\circ }{\varsigma}(z)\mathbf{Q}(z)\bm{\mathcal{Y}}\bm{\mathcal{Y}}^H\right)
		\\&\tilde{\mathbf{\Delta}}(z) = E\left(\overset{\circ}{\tilde{\varsigma}}(z)\tilde{\mathbf{Q}}(z)\bm{\mathcal{Y}}^H\mathbf{\Lambda}(\mathbf{I}_N+\tilde{\delta}(z)\mathbf{\Lambda})^{-1}\bm{\mathcal{Y}}\right)
	\end{aligned}
\end{equation}

where
\begin{equation}\label{key}
	\begin{aligned}
		&\overset{\circ }{\varsigma}(z)=	\frac{1}{K}{\rm tr}(\mathbf{\Lambda}\mathbf{Q}(z))-E\left(\frac{1}{K}{\rm tr}(\mathbf{\Lambda}\mathbf{Q}(z))\right)
		\\&  \overset{\circ }{\tilde{\varsigma}}(z)= \frac{1}{K}{\rm tr}\tilde{\mathbf{Q}}(z)-E\left(\frac{1}{K}{\rm tr}\tilde{\mathbf{Q}}(z)\right)
	\end{aligned}
\end{equation}
In what follows, we will drop the dependence on $z$ of $\delta(z)$, $\tilde{\delta}(z)$, $\tau(z)$, $\tilde{\tau}(z)$, $\overset{\circ }{\varsigma}(z)$, and $\overset{\circ }{\tilde{\varsigma}}(z)$ for the clarity of presentation.

It has been shown in \cite{Dumont2010On} that the following equality holds true
\begin{equation}\label{142}
	E\mathbf{Q}(z) = \mathbf{S}(z) -z(\tilde{\delta}-\tilde{\tau})E\mathbf{Q}(z)\mathbf{\Lambda}\mathbf{S}(z)+\mathbf{\Delta}(z)\mathbf{S}(z).
\end{equation}

Multiplying \eqref{142} from both sides by $\frac{1}{K}\mathbf{\Lambda}$ and taking the trace, we obtain
\begin{equation}\label{143}
	\begin{aligned}
		&\frac{1}{K}{\rm tr}\left(\mathbf{\Lambda}E\mathbf{Q}(z)\right) \\&=\frac{1}{K}{\rm tr}\left(\mathbf{\Lambda}\mathbf{S}(z)\right) - \frac{z}{K}(\tilde{\delta}-\tilde{\tau}){\rm tr}\left(\mathbf{\Lambda}E\mathbf{Q}(z)\mathbf{\Lambda}\mathbf{S}(z)\right) 
		\\&\quad\quad+ \frac{1}{K}{\rm tr}\left(\mathbf{\Lambda}\mathbf{\Delta}(z)\mathbf{S}(z)\right).
	\end{aligned}
\end{equation}

We remark that $\frac{1}{K}{\rm tr}\left(\mathbf{\Lambda}E\mathbf{Q}(z)\right)  =\frac{1}{K} E{\rm tr}\left(\mathbf{\Lambda}\mathbf{Q}(z)\right)=\delta$ (using \eqref{138}) and $ \frac{1}{K}{\rm tr}\left(\mathbf{\Lambda}\mathbf{S}(z)\right) =-\frac{1}{zK}{\rm tr}\left(\mathbf{\Lambda}\left(\mathbf{I}_N+\tilde{\delta}\mathbf{\Lambda}\right)^{-1}\right)=\tau$ (using \eqref{139} and \eqref{140}). Therefore, \eqref{143} is equivalent to
\begin{equation}\label{145}
	\begin{aligned}
		\delta = \tau -  \frac{z}{K}(\tilde{\delta}-\tilde{\tau}){\rm tr}\left(\mathbf{\Lambda}E\mathbf{Q}(z)\mathbf{\Lambda}\mathbf{S}(z)\right) 
		+ \frac{1}{K}{\rm tr}\left(\mathbf{\Lambda}\mathbf{\Delta}(z)\mathbf{S}(z)\right).
	\end{aligned}
\end{equation}

In addition, the following equality is proved in \cite{Dumont2010On} 
\begin{equation}\label{key}
	E\tilde{\mathbf{Q}}(z) = \tilde{\mathbf{S}}(z) - z(\delta-\tau)E\tilde{\mathbf{Q}}(z)\tilde{\mathbf{S}}(z) + \tilde{\mathbf{\Delta}}(z)\tilde{\mathbf{S}}(z).
\end{equation}

Multiplying $1/K$ and taking trace on both sides, and using \eqref{138}-\eqref{140}, we obtain
\begin{equation}\label{147}
	\tilde{\delta} = \tilde{\tau} - \frac{z}{K}(\delta-\tau){\rm tr}\left(E\tilde{\mathbf{Q}}(z)\tilde{\mathbf{S}}(z)\right) + \frac{1}{K}{\rm tr}\left( \tilde{\mathbf{\Delta}}(z)\tilde{\mathbf{S}}(z)\right).
\end{equation}

Then from \eqref{145} and \eqref{147}, we obtain the following linear system
\begin{equation}\label{148}
	\begin{pmatrix}
		\delta-\tau\\
		\tilde{\delta} - \tilde{\tau}
	\end{pmatrix}=\mathbf{G}_3\begin{pmatrix}
		\delta-\tau\\
		\tilde{\delta} - \tilde{\tau}
	\end{pmatrix}+ \begin{pmatrix}
		v_3(z)\\
		\tilde{v}_3(z)
	\end{pmatrix}
\end{equation}
where
\begin{equation}\label{key}
	\mathbf{G}_3 = \begin{pmatrix}
		0,&u_3(z)\\
		\tilde{u}_3(z),&0
	\end{pmatrix}
\end{equation}
with 
\begin{align}
	&u_3 (z)=- \frac{z}{K}{\rm tr}\left(\mathbf{\Lambda}E\mathbf{Q}(z)\mathbf{\Lambda}\mathbf{S}(z)\right) ,
	\\&\tilde{u}_3(z)=  -\frac{z}{K}{\rm tr}\left(E\tilde{\mathbf{Q}}(z)\tilde{\mathbf{S}}(z)\right),
	\\& v_3(z) = \frac{1}{K}{\rm tr}\left(\mathbf{\Lambda}\mathbf{\Delta}(z)\mathbf{S}(z)\right) \label{152},
	\\& \tilde{v}_3(z) = \frac{1}{K}{\rm tr}\left( \tilde{\mathbf{\Delta}}(z)\tilde{\mathbf{S}}(z)\right)\label{153}.
\end{align}

It is shown in \cite{Dumont2010On} that $0<{\rm det}\left(\mathbf{I}-\mathbf{G}_3\right)<1$ and there exists an integer $K_3$ such that $1/{\rm det}\left(\mathbf{I}-\mathbf{G}_3\right)<\infty$ for each $K>K_3$. Hence from the linear system \eqref{148} we have
\begin{equation}\label{154}
	\tilde{\delta}-\tilde{\tau} = \frac{\tilde{u}_3(z)v_3(z)+\tilde{v}_3(z)}{{\rm det}(\mathbf{I}-\mathbf{G}_3)}
\end{equation}

Next we need to study $v_3(z)$ and $\tilde{v}_3(z)$. Substituting the expression of $\mathbf{\Delta}(z)$ (see \eqref{140}) into \eqref{152}, we obtain
\begin{equation}\label{key}
	\begin{aligned}
		v_3(z) = E\left(\frac{1}{K(1+\delta)}{\rm tr}\left(\overset{\circ}{\varsigma}\mathbf{\Lambda}\mathbf{Q}(z)\bm{\mathcal{Y}}\bm{\mathcal{Y}}^H\mathbf{S}(z)\right)\right).
	\end{aligned}
\end{equation}

Using the identity
\begin{equation}\label{key}
	\mathbf{Q}(z)\bm{\mathcal{Y}}\bm{\mathcal{Y}}^H  = \mathbf{I}_N+z\mathbf{Q}(z),
\end{equation}  
we have
\begin{equation}\label{157}
	\begin{aligned}
		v_3(z) &= E\left(\frac{z}{K(1+\delta)}{\rm tr}\left(\overset{\circ}{\varsigma}\mathbf{\Lambda}\mathbf{Q}(z)\mathbf{S}(z)\right)\right)
		\\& = \frac{z}{1+\delta}E\left[\overset{\circ}{\varsigma}\left(\frac{1}{K}{\rm tr}(\mathbf{\Lambda}\mathbf{Q}(z)\mathbf{S}(z))\right.\right.
		\\&\left.\left.\qquad\qquad\qquad- E\left(\frac{1}{K}{\rm tr}(\mathbf{\Lambda}\mathbf{Q}(z)\mathbf{S}(z)) \right)\right)\right]
		\\&{\leqslant} \frac{z}{1+\delta}\sqrt{{\rm Var}\left(\frac{1}{K}{\rm tr}\left(\mathbf{Q}(z)\mathbf{\Lambda}\right)\right)}
		\\&\qquad\quad\times\sqrt{{\rm Var}\left(\frac{1}{K}{\rm tr}(\mathbf{\Lambda}\mathbf{Q}(z)\mathbf{S}(z))\right)}
		.\end{aligned}
\end{equation}

Hence by applying \eqref{Var0} to \eqref{157}, we obtain 
\begin{equation}\label{159}
	v_3(z) \leqslant  \frac{1}{K^2}P_2(|z|)P_4(|{\rm Im}z|^{-1}).
\end{equation}

Likewise, we can obtain the upper bound of $\tilde{v}_3(z)$. For this, we can write $\tilde{v}_3'$ by using \eqref{140} in \eqref{153}
\begin{equation}\label{160}
	\begin{aligned}
		\tilde{v}_3(z) &= E\left[\frac{1}{K}{\rm tr}\left(\overset{\circ}{\tilde{\varsigma}}\tilde{\mathbf{Q}}(z)\bm{\mathcal{Y}}^H\mathbf{\Lambda}(\mathbf{I}_N+\tilde{\delta}\mathbf{\Lambda})^{-1}\bm{\mathcal{Y}}\tilde{\mathbf{S}}(z)\right)\right]
		\\& = E\left[\overset{\circ}{\tilde{\varsigma}}\left(\frac{1}{K}{\rm tr}(\tilde{\mathbf{Q}}(z)\bm{\mathcal{Y}}^H\mathbf{\Lambda}(\mathbf{I}_N+\tilde{\delta}\mathbf{\Lambda})^{-1}\bm{\mathcal{Y}}\tilde{\mathbf{S}}(z)) \right.\right.
		\\&\left.\left.\qquad\quad- E\left(\frac{1}{K}{\rm tr}(\tilde{\mathbf{Q}}(z)\bm{\mathcal{Y}}^H\mathbf{\Lambda}(\mathbf{I}_N+\tilde{\delta}\mathbf{\Lambda})^{-1}\bm{\mathcal{Y}}\tilde{\mathbf{S}}(z))\right)\right) \right]
		\\& \leqslant \sqrt{{\rm Var}\left(\frac{1}{K}{\rm tr}(\tilde{\mathbf{Q}}(z)\bm{\mathcal{Y}}^H\mathbf{\Lambda}(\mathbf{I}_N+\tilde{\delta}\mathbf{\Lambda})^{-1}\bm{\mathcal{Y}}\tilde{\mathbf{S}}(z))\right)}
		\\&\quad\times	\sqrt{{\rm Var}\left(\frac{1}{K}{\rm tr}\tilde{\mathbf{Q}}(z)\right)}
	\end{aligned}
\end{equation}

We note that $||\mathbf{\Lambda}(\mathbf{I}_N+{\tilde{\delta}}\mathbf{\Lambda})^{-1}||\leqslant \lambda_{\rm max}<\infty$ and $||\tilde{\mathbf{S}}_N(z)||\leqslant \frac{1}{|{\rm Im}z|}$. Hence by applying \eqref{Var1} and \eqref{Var2} to \eqref{160}, we have
\begin{equation}\label{163-1}
	\tilde{v}_3(z)\leqslant \frac{1}{K^2}P_2(|z|)P_4(|{\rm Im}z|^{-1})
\end{equation}

Then using \eqref{159} and \eqref{163-1} in \eqref{154} and noting that $|\tilde{u}_3|\leqslant\frac{1}{|{\rm Im}z|}$, we obtain
\begin{equation}\label{164-1}
	\begin{aligned}
		|\tilde{\delta}-\tilde{\tau}| \leqslant \frac{1}{K^2}P_2(|z|)P_5(|{\rm Im}z|^{-1})
	\end{aligned}
\end{equation}

Now let us review the equality \eqref{142}. After some manipulations, \eqref{142} is equivalent to 
\begin{equation}\label{177}
	\begin{aligned}
		\frac{1}{K}{\rm tr}(E\mathbf{Q}(z) - \mathbf{S}(z)) &= -z(\tilde{\delta}-\tilde{\tau})\frac{1}{K}{\rm tr}(E\mathbf{Q}(z)\mathbf{\Lambda}\mathbf{S}(z)) 	
		\\&\quad+\frac{1}{K}{\rm tr}(\mathbf{\Delta}(z)\mathbf{S}(z))
	\end{aligned}
\end{equation}

Then we have
\begin{equation}\label{key}
	\begin{aligned}
		|\frac{1}{K}{\rm tr}(E\mathbf{Q}(z) - \mathbf{S}(z))|&\leqslant |\tilde{\delta}-\tilde{\tau}|	\left|\frac{z}{K}{\rm tr}(E\mathbf{Q}(z)\mathbf{\Lambda}\mathbf{S}(z))\right|
		\\&\quad+\left|\frac{1}{K}{\rm tr}(\mathbf{\Delta}(z)\mathbf{S}(z))\right|
	\end{aligned}
\end{equation} 

By using \eqref{164-1}  and the facts that $\left|\frac{z}{K}{\rm tr}(E\mathbf{Q}_N(z)\mathbf{\Lambda}\mathbf{S}_N(z))\right|\leqslant c_K\frac{\lambda_{\rm max}}{|{\rm Im}z|}$, and $\left|\frac{1}{K}{\rm tr}(\mathbf{\Delta}_N(z)\mathbf{S}_N(z))\right|\leqslant K^{-2}P_2(|z|)P_4(|{\rm Im}z|^{-1})$ (the proof, based on the Poincare-Nash inequality,  is omitted.), we obtain
\begin{equation}\label{key}
	\left|\frac{1}{K}{\rm tr}(E\mathbf{Q}(z) - \mathbf{S}(z))\right| \leqslant \frac{1}{K^2}P_2(|z|)P_6(|{\rm Im}z|^{-1})
\end{equation}

By observing \eqref{V(z)} and \eqref{139}, we find that $\mathbf{V}(z)=\mathbf{S}(z)$. Hence we obtain
\begin{equation}\label{key}
	|\epsilon(z)| \leqslant \frac{1}{K^2}P_2(|z|)P_6(|{\rm Im}z|^{-1})
\end{equation}
which completes the proof of Lemma \ref{convergence rate of epsilon}.
\section*{Appendix C\\ Proof of Lemma \ref{lemma bound of epsilon'(z)} }

From \eqref{177}, we know that
\begin{equation}\label{219}
	\begin{aligned}
		&\epsilon'(z) =
		\\&\quad  -(\tilde{\delta}'-\tilde{\tau}')\frac{z}{K}{\rm tr}\left[E\mathbf{Q}'(z)\mathbf{\Lambda}\mathbf{S}(z)+E\mathbf{Q}(z)\mathbf{\Lambda}\mathbf{S}'(z)\right]
		\\&\quad-\frac{\tilde{\delta}-\tilde{\tau}}{K}{\rm tr}\left[(E\mathbf{Q}(z)+E\mathbf{Q}'(z))\mathbf{\Lambda}\mathbf{S}(z)+E\mathbf{Q}(z)\mathbf{\Lambda}\mathbf{S}'(z)\right]
		\\&\quad+\frac{1}{K}{\rm tr}\left[\mathbf{\Delta}'(z)\mathbf{S}(z)+\mathbf{\Delta}(z)\mathbf{S}'(z)\right].	
	\end{aligned}
\end{equation}
One can easily check that 
\begin{equation}\label{220}
	\frac{z}{K}{\rm tr}\left[E\mathbf{Q}'(z)\mathbf{\Lambda}\mathbf{S}(z)+E\mathbf{Q}(z)\mathbf{\Lambda}\mathbf{S}'(z)\right]\leqslant c_K|z|\frac{\lambda_{\rm max}}{|{\rm Im}z|^3},
\end{equation} 
and 
\begin{equation}\label{221}
	\begin{aligned}
		&\frac{1}{K}{\rm tr}\left[(E\mathbf{Q}(z)+E\mathbf{Q}'(z))\mathbf{\Lambda}\mathbf{S}(z)+E\mathbf{Q}(z)\mathbf{\Lambda}\mathbf{S}'(z)\right]
		\\&\leqslant c_K\left(\frac{\lambda_{\rm max}}{|{\rm Im}z|^3}+\frac{\lambda_{\rm max}}{|{\rm Im}z|^2}\right)
	\end{aligned}
\end{equation}

In addition, by the same analysis for $v_3'$ (see \eqref{definition v3'} below), we have
\begin{equation}\label{222}
	\frac{1}{K}{\rm tr}\left[\mathbf{\Delta}'(z)\mathbf{S}(z)+\mathbf{\Delta}(z)\mathbf{S}'(z)\right]\leqslant \frac{1}{K^2}P_2(|z|)P_5(|{\rm Im}z|^{-1}).
\end{equation}

Furthermore, in what follows we will prove
\begin{equation}\label{223}
	|\tilde{\delta}'-\tilde{\tau}'|\leqslant \frac{1}{K^2}P_2(|z|)P_8(|{\rm Im}z|^{-1}).
\end{equation}

Using \eqref{164-1} and \eqref{220}-\eqref{223} in \eqref{219}, we obtain
\begin{equation}\label{key}
	|\epsilon'(z)|\leqslant \frac{1}{K^2}P_3(|z|)P_{11}(|{\rm Im}z|^{-1})
\end{equation}

The remaining part is dedicated to proving \eqref{223}. Taking the derivative of both sides of \eqref{148},  we have
\begin{equation}\label{key}
	\begin{pmatrix}
		\delta'-\tau'\\
		\tilde{\delta}' - \tilde{\tau}'
	\end{pmatrix}=\mathbf{G}'_3\begin{pmatrix}
		\delta-\tau\\
		\tilde{\delta} - \tilde{\tau}
	\end{pmatrix}+ \mathbf{G}_3 \begin{pmatrix}
		\delta'-\tau'\\
		\tilde{\delta}' - \tilde{\tau}'
	\end{pmatrix}+\begin{pmatrix}
		v'_3\\
		\tilde{v}'_3
	\end{pmatrix}
\end{equation}
where 
\begin{equation}\label{key}
	\mathbf{G}_3' = \begin{pmatrix}
		0,&u_3'\\
		\tilde{u}_3',&0
	\end{pmatrix}
\end{equation}
with
\begin{align}
	&u_3'=- \frac{z}{K}{\rm tr}\left[\mathbf{\Lambda}E\mathbf{Q}'(z)\mathbf{\Lambda}\mathbf{S}(z)+\mathbf{\Lambda}E\mathbf{Q}(z)\mathbf{\Lambda}\mathbf{S}'(z)\right]
	\\&\tilde{u}_3'=  -\frac{z}{K}{\rm tr}\left(E\tilde{\mathbf{Q}}'(z)\tilde{\mathbf{S}}(z)+E\tilde{\mathbf{Q}}(z)\tilde{\mathbf{S}}'(z)\right)
	\\& v_3'= \frac{1}{K}{\rm tr}\left[\mathbf{\Lambda}\mathbf{\Delta}'(z)\mathbf{S}(z)+\mathbf{\Lambda}\mathbf{\Delta}(z)\mathbf{S}'(z)\right] \label{definition v3'}
	\\& \tilde{v}_3' = \frac{1}{K}{\rm tr}\left[ \tilde{\mathbf{\Delta}}'(z)\tilde{\mathbf{S}}(z)+\tilde{\mathbf{\Delta}}(z)\tilde{\mathbf{S}}'(z)\right]
\end{align}

Then 
\begin{equation}\label{224}
	\tilde{\delta}' - \tilde{\tau}' = \frac{1}{{\rm det}(\mathbf{I}-\mathbf{G}_3)}\left(\tilde{u}_3'(\delta-\tau) + \tilde{u}_3u_3'(	\tilde{\delta} - \tilde{\tau})+ \tilde{u}_3v_3'+\tilde{v}_3'\right)
\end{equation}  

It is easy to check that $\Vert\mathbf{S}_N(z)\Vert, \Vert\tilde{\mathbf{S}}_N(z)\Vert\leqslant \frac{1}{|{\rm Im}z|}$ and $\Vert\mathbf{S}'_N(z)\Vert, \Vert\tilde{\mathbf{S}}'_N(z)\Vert\leqslant \frac{1}{|{\rm Im}z|^2}$. Hence we obtain $|u_3'|\leqslant 2c_K\frac{\lambda_{\rm max}^2}{|{\rm Im}z|^2}$ and $|\tilde{u}_3'|\leqslant \frac{2}{|{\rm Im}z|^2}$. In addition, using the same analysis of $|\tilde{\delta}-\tilde{\tau}|$ given in \eqref{154}, we have
\begin{equation}\label{bound of delta-tau}
	|\delta-\tau|\leqslant \frac{1}{K^2}P_2(|z|)P_5(|{\rm Im}z|^{-1})
\end{equation}

From these upper bounds together with $|\tilde{u}_3|\leqslant \frac{c_K}{|{\rm Im}z|}$, we obtain the upper bound of the first two terms in brackets of \eqref{224}:
\begin{equation}\label{key}
	\frac{\tilde{u}_3'(\delta-\tau) + \tilde{u}_3u_3'(	\tilde{\delta} - \tilde{\tau})}{{\rm det}(\mathbf{I}-\mathbf{G}_3)}\leqslant \frac{1}{K^2}P_2(|z|)P_8(|{\rm Im}z|^{-1})
\end{equation}

Next we need to study $v_3'$ and $\tilde{v}_3'$. Taking derivative with respect to $z$ of \eqref{152} and \eqref{153}, we have
\begin{align}
	&v_3' = \frac{1}{K}{\rm tr}\left(\mathbf{\Lambda}\mathbf{\Delta}'(z)\mathbf{S}(z)+\mathbf{\Lambda}\mathbf{\Delta}(z)\mathbf{S}'(z)\right),\label{v3'-122}
	\\&\tilde{v}_3' = \frac{1}{K}{\rm tr}\left(\tilde{\mathbf{\Delta}}'(z)\tilde{\mathbf{S}}(z)+\tilde{\mathbf{\Delta}}(z)\tilde{\mathbf{S}}'(z)\right)\label{tilde v3'-122}.
\end{align}

Taking derivative of both sides of \eqref{140} yields
\begin{equation}\label{Delta_N' and tilde Delta_N'}
	\begin{aligned}
		&\mathbf{\Delta}'(z) = \mathbf{\Delta}_1(z) + \mathbf{\Delta}_2(z) + \mathbf{\Delta}_3(z),
		\\& \tilde{\mathbf{\Delta}}'(z) = \tilde{\mathbf{\Delta}}_1(z) + \tilde{\mathbf{\Delta}}_2(z) + \tilde{\mathbf{\Delta}}_3(z),
	\end{aligned}
\end{equation}
where
\begin{equation}\label{key}
	\begin{aligned}
		&\mathbf{\Delta}_1(z) = \frac{1}{1+\delta}E\left(\overset{\circ}{\varsigma}'\mathbf{Q}(z)\bm{\mathcal{Y}}\bm{\mathcal{Y}}^H\right),
		\\& \mathbf{\Delta}_2(z) =\frac{1}{1+\delta}E\left(\overset{\circ}{\varsigma}\mathbf{Q}'(z)\bm{\mathcal{Y}}\bm{\mathcal{Y}}^H\right),
		\\& \mathbf{\Delta}_3(z) = -\mathbf{\Delta}_N(z)\delta',
	\end{aligned}
\end{equation}
and
\begin{equation}\label{key}
	\begin{aligned}
		&\tilde{\mathbf{\Delta}}_1(z) = E\left(\overset{\circ}{\tilde{\varsigma}}'\tilde{\mathbf{Q}}(z)\bm{\mathcal{Y}}^H\mathbf{\Lambda}(\mathbf{I}_N+\tilde{\delta}\mathbf{\Lambda})^{-1}\bm{\mathcal{Y}}\right),
		\\	&\tilde{\mathbf{\Delta}}_2(z) = E\left(\overset{\circ}{\tilde{\varsigma}}\tilde{\mathbf{Q}}'(z)\bm{\mathcal{Y}}^H\mathbf{\Lambda}(\mathbf{I}_N+\tilde{\delta}\mathbf{\Lambda})^{-1}\bm{\mathcal{Y}}\right),
		\\&	\tilde{\mathbf{\Delta}}_3(z) =-   E\left(\overset{\circ}{\tilde{\varsigma}}\tilde{\delta}'\tilde{\mathbf{Q}}(z)\bm{\mathcal{Y}}^H\mathbf{\Lambda}(\mathbf{I}_N+\tilde{\delta}\mathbf{\Lambda})^{-1}\mathbf{\Lambda}(\mathbf{I}_N+\tilde{\delta}\mathbf{\Lambda})^{-1}\bm{\mathcal{Y}}\right).
	\end{aligned}
\end{equation}

Then substituting \eqref{Delta_N' and tilde Delta_N'} into \eqref{v3'-122} gives
\begin{equation}\label{v3'}
	\begin{aligned}
		v_3' &= \frac{1}{K}{\rm tr}\left[\mathbf{\Lambda}\mathbf{\Delta}_1(z)\mathbf{S}(z)+\mathbf{\Lambda}\mathbf{\Delta}_2(z)\mathbf{S}(z)+\mathbf{\Lambda}\mathbf{\Delta}_3(z)\mathbf{S}(z)\right.
		\\&\left.\qquad\quad+\mathbf{\Lambda}\mathbf{\Delta}(z)\mathbf{S}'(z)\right] .
	\end{aligned}
\end{equation}

Let us now find the upper bound of the first term of RHS of \eqref{v3'}. Substituting the expression of $\mathbf{\Delta}_1$ and using the identity $\mathbf{Q}(z)\bm{\mathcal{Y}}\bm{\mathcal{Y}}^H = \mathbf{I}_N+z\mathbf{Q}(z)$ yield
\begin{equation}\label{key}
	\small\begin{aligned}
		&\frac{1}{K}{\rm tr}(\mathbf{\Lambda}\mathbf{\Delta}_1(z)\mathbf{S}(z)) 
		\\&\quad= E\left[\frac{z}{K(1+\delta)}{\rm tr}\left(\overset{\circ}{\varsigma}'\mathbf{\Lambda}\mathbf{Q}(z)\mathbf{S}(z)\right)\right]
		\\&\quad = \frac{z}{1+\delta}E\left[\overset{\circ}{\varsigma}'\left(\frac{1}{K}{\rm tr}(\mathbf{\Lambda}\mathbf{Q}(z)\mathbf{S}(z)) -E\left(\frac{1}{K}{\rm tr}(\mathbf{\Lambda}\mathbf{Q}(z)\mathbf{S}(z)) \right)\right)\right]
		\\&\quad\leqslant \frac{z}{1+\delta}\sqrt{{\rm Var}\left(\frac{1}{K}{\rm tr}(\mathbf{\Lambda}\mathbf{Q}'(z))\right){\rm Var}\left(\frac{1}{K}{\rm tr}(\mathbf{\Lambda}\mathbf{Q}(z)\mathbf{S}(z))\right)}
	\end{aligned}
\end{equation}

We need the following lemma.
\begin{lemma}\label{Lemma Var Q'}
	For arbitrary matrices $\mathbf{C}_1,\mathbf{C}_2,\mathbf{D}_1\in\mathbb{C}^{N\times N} $, $\mathbf{D}_2\in\mathbb{C}^{K\times K}$ with uniformly bounded spectral norms, it holds true that
	\begin{align}
		&	{\rm Var}\left(\frac{1}{K}{\rm tr}\left(\mathbf{C}_1\mathbf{Q}'(z)\mathbf{C}_2\right)\right)\leqslant\frac{1}{K^2}P_{1}(|z|)P_{6}(|{\rm Im}z|^{-1}), \label{Lemma Var Q' 1}
		\\& {\rm Var}\left(\frac{1}{K}{\rm tr}\tilde{\mathbf{Q}}'(z)\right)\leqslant \frac{1}{K^2}P_{1}(|z|)P_{6}(|{\rm Im}z|^{-1}),\label{Lemma Var Q' 3}
		\\& {\rm Var}\left(\frac{1}{K}{\rm tr}(\tilde{\mathbf{Q}}'(z)\bm{\mathcal{Y}}^H\mathbf{D}_1\bm{\mathcal{Y}}\mathbf{D}_2)\right)\leqslant\frac{1}{K^2}P_{3}(|z|)P_{6}(|{\rm Im}z|^{-1})\label{Lemma Var Q' 2}.
	\end{align}
\end{lemma}
\begin{proof}
	See Appendix E in Supplementary Material (pages 10-11).
\end{proof}

Noticing that $|\frac{z}{1+\delta}|\leqslant |z|$  and using \eqref{Var0} and \eqref{Lemma Var Q' 1}, we obtain
\begin{equation}\label{v3'-1}
	\frac{1}{K}{\rm tr}(\mathbf{\Lambda}\mathbf{\Delta}_1(z)\mathbf{S}_N(z))\leqslant \frac{1}{K^2}P_{2}(|z|)P_{5}(|{\rm Im}z|^{-1})
\end{equation}

Noticing that $\mathbf{Q}'(z) = \mathbf{Q}(z)\mathbf{Q}(z)$ and using the same analysis, we obtain
\begin{equation}\label{v3'-2}\small
	\begin{aligned}
		&\frac{1}{K}{\rm tr}(\mathbf{\Lambda}\mathbf{\Delta}_2(z)\mathbf{S}(z)) 
		\\&= \frac{z}{1+\delta}E\left[{\frac{1}{K}\rm tr}\left(\overset{\circ}{\varsigma}\mathbf{\Lambda}\mathbf{Q}'(z)\mathbf{S}_N(z)\right)\right]
		\\& = \frac{z}{1+\delta}E\left[\overset{\circ}{\varsigma}\left(\frac{1}{K}{\rm tr}(\mathbf{\Lambda}\mathbf{Q}'(z)\mathbf{S}(z)) - E\left(\frac{1}{K}{\rm tr}(\mathbf{\Lambda}\mathbf{Q}'(z)\mathbf{S}(z))\right)\right)\right]
		\\& \leqslant \frac{z}{1+\delta}\sqrt{{\rm Var}\left(\frac{1}{K}{\rm tr}(\mathbf{\Lambda}\mathbf{Q}(z))\right){\rm Var}\left(\frac{1}{K}{\rm tr}(\mathbf{\Lambda}\mathbf{Q}'(z)\mathbf{S}(z))\right)}
		\\&
		\leqslant \frac{1}{K^2}P_{2}(|z|)P_{5}(|{\rm Im}z|^{-1})
	\end{aligned}
\end{equation}

Similarly, we have
\begin{equation}\label{v3'-3}\small
	\begin{aligned}
		&\left|	\frac{1}{K}{\rm tr}(\mathbf{\Lambda}\mathbf{\Delta}_3(z)\mathbf{S}(z)) \right|
		\\&= \left|\frac{z\delta'}{1+\delta}E\left[\frac{1}{K}{\rm tr}(\overset{\circ}{\varsigma}\mathbf{\Lambda}\mathbf{Q}(z)\mathbf{S}(z))\right]\right|
		\\& = \left|\frac{z\delta'}{1+\delta}E\left[\overset{\circ}{\varsigma}\left(\frac{1}{K}{\rm tr}(\mathbf{\Lambda}\mathbf{Q}(z)\mathbf{S}(z))-E\left(\frac{1}{K}{\rm tr}(\mathbf{\Lambda}\mathbf{Q}(z)\mathbf{S}(z))\right)\right)\right]\right|
		\\&\leqslant \frac{z\delta'}{1+\delta}\sqrt{{\rm Var}\left(\frac{1}{K}{\rm tr}(\mathbf{\Lambda}\mathbf{Q}(z))\right){\rm Var}\left(\frac{1}{K}{\rm tr}(\mathbf{\Lambda}\mathbf{Q}(z)\mathbf{S}(z))\right)}
		\\&\leqslant \frac{1}{K^2}P_1(|z|)P_5(|{\rm Im}z|^{-1})
	\end{aligned}
\end{equation}
and
\begin{equation}\label{v3'-4}
	\begin{aligned}
		&	\frac{1}{K}{\rm tr}\left(\mathbf{\Lambda}\mathbf{\Delta}(z)\mathbf{S}'_N(z)\right) 
		\\&= \frac{z}{1+\delta}E\left[\frac{1}{K}{\rm tr}(\overset{\circ}{\varsigma}\mathbf{\Lambda}\mathbf{Q}(z)\mathbf{S}'(z))\right]
		\\&\leqslant \frac{z}{1+\delta}\sqrt{{\rm Var}\left(\frac{1}{K}{\rm tr}(\mathbf{\Lambda}\mathbf{Q}(z))\right){\rm Var}\left(\frac{1}{K}{\rm tr}(\mathbf{\Lambda}\mathbf{Q}(z)\mathbf{S}'(z))\right)}
		\\&\leqslant \frac{1}{K^2}P_2(|z|)P_4(|{\rm Im}z|^{-1})
	\end{aligned}
\end{equation}

From \eqref{v3'-1}, \eqref{v3'-2}, \eqref{v3'-3} and \eqref{v3'-4}, we have
\begin{equation}\label{bound of v3'}
	|	v_3'|\leqslant \frac{1}{K^2}P_2(|z|)P_5(|{\rm Im}z|^{-1})
\end{equation}

Next  we find the  bound of $\tilde{v}_3'$. Using \eqref{Delta_N' and tilde Delta_N'}in \eqref{tilde v3'-122}, we can expand $\tilde{v}_3'$ as
\begin{equation}\label{136}                                                                    
	\tilde{v}_3'=\frac{1}{K}{\rm tr}\left(\tilde{\mathbf{\Delta}}_1\tilde{\mathbf{S}}(z)+\tilde{\mathbf{\Delta}}_2\tilde{\mathbf{S}}(z)+\tilde{\mathbf{\Delta}}_3\tilde{\mathbf{S}}(z) +\tilde{\mathbf{\Delta}}\tilde{\mathbf{S}}'(z)\right)
\end{equation}

For the first term of \eqref{136}, we have
\begin{equation}\label{tilde v3'-1}
	\begin{aligned}
		&\frac{1}{K}{\rm tr}(\tilde{\mathbf{\Delta}}_1\tilde{\mathbf{S}}(z)) 
		\\&= \frac{1}{K}E{\rm tr}\left(\overset{\circ}{\tilde{\varsigma}}'\tilde{\mathbf{Q}}(z)\bm{\mathcal{Y}}^H\mathbf{\Lambda}(\mathbf{I}_N+\tilde{\delta}\mathbf{\Lambda})^{-1}\bm{\mathcal{Y}}\tilde{\mathbf{S}}(z)\right)
		\\& = E\left[\overset{\circ}{\tilde{\varsigma}}'\left(\frac{1}{K}{\rm tr}(\tilde{\mathbf{Q}}(z)\bm{\mathcal{Y}}^H\mathbf{\Lambda}(\mathbf{I}_N+\tilde{\delta}\mathbf{\Lambda})^{-1}\bm{\mathcal{Y}}\tilde{\mathbf{S}}(z))
		\right.\right.\\&\qquad\quad\left.\left.-E\left(\frac{1}{K}{\rm tr}(\tilde{\mathbf{Q}}(z)\bm{\mathcal{Y}}^H\mathbf{\Lambda}(\mathbf{I}_N+\tilde{\delta}\mathbf{\Lambda})^{-1}\bm{\mathcal{Y}}\tilde{\mathbf{S}}(z))\right)\right)\right]
		\\&\leqslant \sqrt{{\rm Var}\left(\frac{1}{K}{\rm tr}(\tilde{\mathbf{Q}}(z)\bm{\mathcal{Y}}^H\mathbf{\Lambda}(\mathbf{I}_N+\tilde{\delta}\mathbf{\Lambda})^{-1}\bm{\mathcal{Y}}\tilde{\mathbf{S}}(z))\right)}
		\\&\quad\times\sqrt{{\rm Var}\left(\frac{1}{K}{\rm tr}\tilde{\mathbf{Q}}'(z)\right)}
		\\&\overset{(a)}{\leqslant } \frac{1}{K^2}P_2(|z|)P_5(|{\rm Im}z|^{-1})
	\end{aligned}
\end{equation}
where $(a)$ follows from \eqref{Lemma Var Q' 1}.

For the second term, we have
\begin{equation}\label{tilde v3'-2}
	\begin{aligned}
		&\frac{1}{K}{\rm tr}(\tilde{\mathbf{\Delta}}_2\tilde{\mathbf{S}}(z)) 
		\\&= \frac{1}{K}E{\rm tr}\left(\overset{\circ}{\tilde{\varsigma}}\tilde{\mathbf{Q}}'
		(z)\bm{\mathcal{Y}}^H\mathbf{\Lambda}(\mathbf{I}_N+\tilde{\delta}\mathbf{\Lambda})^{-1}\bm{\mathcal{Y}}\tilde{\mathbf{S}}(z)\right)
		\\&\leqslant  \sqrt{{\rm Var}\left(\frac{1}{K}{\rm tr}(\tilde{\mathbf{Q}}'(z)\bm{\mathcal{Y}}^H\mathbf{\Lambda}(\mathbf{I}_N+\tilde{\delta}\mathbf{\Lambda})^{-1}\bm{\mathcal{Y}}\tilde{\mathbf{S}}(z))\right)}
		\\&\quad\times\sqrt{{\rm Var}\left(\frac{1}{K}{\rm tr}\tilde{\mathbf{Q}}(z)\right)}
		\\&\overset{(a)}{\leqslant } \frac{1}{K^2}P_2(|z|)P_5(|{\rm Im}z|^{-1})
	\end{aligned}
\end{equation}
where $(a)$ follows from \eqref{Var1} and \eqref{Lemma Var Q' 2}.

For the third term, noticing that $\tilde{\delta}'\leqslant \frac{c_K}{|{\rm Im}z|^2}$ and by using the same analysis, we have
\begin{equation}\label{tilde v3'-3}
	\left|	\frac{1}{K}{\rm tr}(\tilde{\mathbf{\Delta}}_3\tilde{\mathbf{S}}(z))\right|\leqslant\frac{1}{K^2}P_2(|z|)P_7(|{\rm Im}z|^{-1})
\end{equation}

Likewise, for the fourth term, we have
\begin{equation}\label{tilde v3'-4}
	\frac{1}{K}{\rm tr}(\tilde{\mathbf{\Delta}}_3\tilde{\mathbf{S}}(z))\leqslant \frac{1}{K^2}P_2(|z|)P_4(|{\rm Im}z|^{-1})
\end{equation}\

Reviewing \eqref{v3'} and using \eqref{tilde v3'-1}, \eqref{tilde v3'-2}, \eqref{tilde v3'-3} and \eqref{tilde v3'-4}, we have
\begin{equation}\label{bound of tilde v3'}
	|\tilde{v}_3'|\leqslant\frac{1}{K^2}P_2(|z|)P_7(|{\rm Im}z|^{-1})
\end{equation}

Finally, using \eqref{164-1}, \eqref{bound of delta-tau}, \eqref{bound of v3'} and \eqref{bound of tilde v3'} in \eqref{224}, and noticing that $|\tilde{u}_3'|\leqslant\frac{2}{|{\rm Im}z|^2},|\tilde{u}_3|\leqslant \frac{1}{|{\rm Im}z|}$, and $|u_3'|\leqslant 2c_K\frac{\lambda^2_{\rm max}}{|{\rm Im}z|^2}$, we obtain 
\begin{equation}\label{key}
	|\tilde{\delta}'-\tilde{\tau}'|\leqslant \frac{1}{K^2}P_2(|z|)P_8(|{\rm Im}z|^{-1})
\end{equation}
which completes the proof of \eqref{223}.
\section*{Appendix D \\ Proof of Lemma \ref{Lemma inequality}}
The proof of  the inequalities in Lemma \ref{Lemma inequality}  relies on the Poincare-Nash inequality (see for instance, \cite{Dumont2010On,Vallet2012Improved}) that has been widely used to deal with the variances in RMT.

For convenience, we write ${\mathbf{Q}}(z)$ and $\tilde{\mathbf{Q}}(z)$ as ${\mathbf{Q}}$ and $\tilde{\mathbf{Q}}$,  respectively.

We first establish \eqref{Var0}.  We remark that
\begin{equation}\label{key}
	\small\begin{aligned}
		\frac{\partial{(\mathbf{C}_1\mathbf{Q}\mathbf{C}_2)_{pq}}}{\partial \bm{\mathcal{Y}}_{ij}} &= \sum_{r,s}\frac{\partial (\mathbf{C}_1)_{pr}\mathbf{Q}_{rs}(\mathbf{C}_2)_{sq}}{\partial \bm{\mathcal{Y}}_{ij}}
		\\& = \sum_{r,s}(\mathbf{C}_1)_{pr}\left(-\mathbf{Q}\frac{\partial \bm{\mathcal{Y}}\bm{\mathcal{Y}}^H}{\partial \bm{\mathcal{Y}}_{ij}}\mathbf{Q}\right)_{rs}(\mathbf{C}_2)_{sq}
		\\& =  \sum_{r,s}(\mathbf{C}_1)_{pr}\left(-\sum_{m,n,t}\mathbf{Q}_{rm}\frac{\partial \bm{\mathcal{Y}}_{mn}\overline{\bm{\mathcal{Y}}_{tn}}}{\partial \bm{\mathcal{Y}}_{ij}}\mathbf{Q}_{ts}\right)(\mathbf{C}_2)_{sq}
		\\& = -(\mathbf{D}_1\mathbf{Q})_{pi}(\bm{\mathcal{Y}}^H\mathbf{Q}\mathbf{D}_2)_{jq}
	\end{aligned}
\end{equation}
Likewise, we can obtain
\begin{equation}\label{key}
	\begin{aligned}
		\frac{\partial{(\mathbf{C}_1\mathbf{Q}\mathbf{C}_2)_{pq}}}{\partial \bar{\bm{\mathcal{Y}}}_{ij}} = -(\mathbf{C}_1\mathbf{Q}\bm{\mathcal{Y}})_{pj} (\mathbf{Q}\mathbf{C}_2)_{iq}
	\end{aligned}
\end{equation}

Then using the Poincare-Nash inequality,  we have
\begin{equation}\label{key}
	\begin{aligned}
		&{\rm Var}\left(\frac{1}{K}{\rm tr}(\mathbf{C}_1\mathbf{Q}\mathbf{C}_2)\right)
		\\&\leqslant \frac{\lambda_{\rm max}}{K}\sum_{i,j}E\left(\left|\frac{1}{K}\sum_{p,q}\frac{\partial{(\mathbf{C}_1\mathbf{Q}\mathbf{C}_2)_{pq}}}{\partial \bm{\mathcal{Y}}_{ij}}\right|^2\right.
		\\&\left.+\left|\frac{1}{K}\sum_{p,q}\frac{\partial{(\mathbf{C}_1\mathbf{Q}\mathbf{C}_2)_{pq}}}{\partial \overline{\bm{\mathcal{Y}}_{ij}}}\right|^2\right)
		\\&\leqslant \frac{\lambda_{\rm max}}{K^3}\sum_{i,j}E\left(\left|(\bm{\mathcal{Y}}^H\mathbf{Q}\mathbf{C}_2\mathbf{C}_1\mathbf{Q})_{ji}\right|^2+\left|(\mathbf{Q}\mathbf{C}_2\mathbf{C}_1\mathbf{Q}\bm{\mathcal{Y}})_{ij}\right|^2\right)
		\\&\leqslant \frac{\lambda_{\rm max}}{K^3}E\left({\rm tr}(\bm{\mathcal{Y}}^H\mathbf{Q}\mathbf{D}_2\mathbf{D}_1\mathbf{Q}\mathbf{Q}^H\mathbf{D}_1^H\mathbf{D}_2^H\mathbf{Q}^H\bm{\mathcal{Y}})
		\right.\\&\quad\left.+{\rm tr}(\mathbf{Q}\mathbf{C}_2\mathbf{C}_1\mathbf{Q}\bm{\mathcal{Y}}\bm{\mathcal{Y}}^H\mathbf{Q}^H\mathbf{C}_1^H\mathbf{C}_2^H\mathbf{Q}^H)\right)
		\\&\leqslant \frac{\lambda_{\rm max}}{K^3}||\mathbf{C}_1||^2||\mathbf{C}_2||^2E\left({\rm tr}(\bm{\mathcal{Y}}^H\mathbf{Q}\mathbf{Q}\mathbf{Q}^H\mathbf{Q}^H\bm{\mathcal{Y}}) 
		\right.\\&\left.\qquad\qquad\qquad\qquad\quad+{\rm tr}(\mathbf{Q}\mathbf{Q}\bm{\mathcal{Y}}\bm{\mathcal{Y}}^H\mathbf{Q}^H\mathbf{Q}^H)\right).
	\end{aligned}
\end{equation}

By using the identities
\begin{equation}\label{QYY' = I+zQ}
	\begin{aligned}
		&\mathbf{Q}\bm{\mathcal{Y}}\bm{\mathcal{Y}}^H = \mathbf{I}_N+z\mathbf{Q},
		\\&\mathbf{Q}^H\bm{\mathcal{Y}}\bm{\mathcal{Y}}^H = \mathbf{I}_N+\bar{z}\mathbf{Q}^H,
	\end{aligned}
\end{equation}
we obtain
\begin{equation}\label{108}
\small	\begin{aligned}
		&{\rm Var}\left(\frac{1}{K}{\rm tr}(\mathbf{C}_1\mathbf{Q}\mathbf{C}_2)\right)
		\\&\leqslant\frac{\lambda_{\rm max}}{K^2}||\mathbf{C}_1||^2||\mathbf{C}_2||^2||\mathbf{Q}||^3E\left(\frac{2}{K}{\rm tr}(\mathbf{I}_N+z\mathbf{Q})\right)
		\\&\leqslant 2\frac{\lambda_{\rm max}}{K^2}||\mathbf{C}_1||^2||\mathbf{C}_2||^2\frac{1}{|{\rm Im}z|^3}(1+\frac{|z|}{|{\rm Im}z|})
		\\&\leqslant 2\frac{\lambda_{\rm max}}{K^2}||\mathbf{C}_1||^2||\mathbf{C}_2||^2 (|z|+1)\left(\frac{1}{|{\rm Im}z|^3}+\frac{1}{|{\rm Im}z|^4}\right).
	\end{aligned}
\end{equation}

Because $\lambda_{\rm max}<\infty$, $||\mathbf{C}_1||,||\mathbf{C}_2||<\infty$, we  obtain \eqref{Var0} from \eqref{108}.

Next we establish \eqref{Var1}. Note that 
\begin{equation}\label{162}
\small	\begin{aligned}
		\frac{\partial \tilde{\mathbf{Q}}_{pq}}{\partial {{\bm{\mathcal{Y}}}_{ij}}} &=-\sum_{s,t,n}\tilde{\mathbf{Q}}_{ps}\frac{\partial \bar{\bm{\mathcal{Y}}}_{ts}{\bm{\mathcal{Y}}}_{tn}}{\partial \bm{\mathcal{Y}}_{ij}}\tilde{\mathbf{Q}}_{nq}
		\\&= -\left(\tilde{\mathbf{Q}}\bm{\mathcal{Y}}^H\right)_{pi}\tilde{\mathbf{Q}}_{jq}
	\end{aligned}
\end{equation}
and
\begin{equation}\label{163}
\small	\begin{aligned}
		\frac{\partial \tilde{\mathbf{Q}}_{pq}}{\partial \bar{\bm{\mathcal{Y}}}'_{ij}} 
		&=-\sum_{s,t,n}\tilde{\mathbf{Q}}_{ps}\frac{\partial \overline{\bm{\mathcal{Y}}_{ts}}{\bm{\mathcal{Y}}}_{tn}}{\partial \overline{\bm{\mathcal{Y}}_{ij}}}\tilde{\mathbf{Q}}_{nq}
		\\&=-\tilde{\mathbf{Q}}_{pj}\left(\bm{\mathcal{Y}}\tilde{\mathbf{Q}}\right)_{iq}	.
	\end{aligned}
\end{equation}

By using the Poincare-Nash inequality, we have
\begin{equation}\label{key}
	\small \begin{aligned}
		&{\rm Var}\left(\frac{1}{K}{\rm tr}\tilde{\mathbf{Q}}\right)
		\\&\leqslant\frac{\lambda_{\rm max}}{K}\sum_{i,j}E\left(\left|\frac{1}{K}\sum_{p,q}\frac{\partial \tilde{\mathbf{Q}}_{pq}}{\partial {{\bm{\mathcal{Y}}}_{ij}}}\right|^2+\left|\frac{1}{K}\sum_{p,q}\frac{\partial \tilde{\mathbf{Q}}_{pq}}{\partial \overline{\bm{\mathcal{Y}}_{ij}}} \right|^2\right)
		\\&\leqslant\frac{\lambda_{\rm max}}{K^3}\sum_{i,j}\left(E\left|\left(\tilde{\mathbf{Q}}\tilde{\mathbf{Q}}\bm{\mathcal{Y}}^H\right)_{ji}\right|^2 + E\left|\left(\bm{\mathcal{Y}}\tilde{\mathbf{Q}}\tilde{\mathbf{Q}}\right)_{ij}\right|^2\right)
		\\& \leqslant \frac{\lambda_{\rm max}}{K^3}\sum_{i} E\left(\tilde{\mathbf{Q}}\tilde{\mathbf{Q}}^H\bm{\mathcal{Y}}^H\bm{\mathcal{Y}}\tilde{\mathbf{Q}}^H\tilde{\mathbf{Q}}\right)_{ii}\\&\quad\quad\qquad+E\left(\bm{\mathcal{Y}}\tilde{\mathbf{Q}}\tilde{\mathbf{Q}}\tilde{\mathbf{Q}}^H\tilde{\mathbf{Q}}^H\bm{\mathcal{Y}}^H\right)_{ii}
		\\& = \frac{\lambda_{\rm max}}{K^3}\left(E{\rm tr}\left(\tilde{\mathbf{Q}}\tilde{\mathbf{Q}}^H\tilde{\mathbf{Q}}^H\tilde{\mathbf{Q}}\bm{\mathcal{Y}}^H\bm{\mathcal{Y}}\right) +E{\rm tr}\left(\tilde{\mathbf{Q}}\tilde{\mathbf{Q}}\tilde{\mathbf{Q}}^H\tilde{\mathbf{Q}}^H\bm{\mathcal{Y}}^H\bm{\mathcal{Y}}\right)\right).
	\end{aligned}
\end{equation}

Using  the identities
\begin{equation}\label{164}
	\small \begin{aligned}
		&	\tilde{\mathbf{Q}}\bm{\mathcal{Y}}^H\bm{\mathcal{Y}} = \mathbf{I}_K+z\tilde{\mathbf{Q}},
		\\& \tilde{\mathbf{Q}}^H\bm{\mathcal{Y}}^H\bm{\mathcal{Y}} = \mathbf{I}_K+\bar{z}\tilde{\mathbf{Q}}^H,
	\end{aligned}
\end{equation}
gives
\begin{equation}\label{}
	\small \begin{aligned}
		{\rm Var}\left(\frac{1}{K}{\rm tr}\tilde{\mathbf{Q}}\right)&\leqslant\frac{\lambda_{\rm max}}{K^3}E{\rm tr}\left(\tilde{\mathbf{Q}}\tilde{\mathbf{Q}}^H\tilde{\mathbf{Q}}^H( \mathbf{I}_K+z\tilde{\mathbf{Q}})\right) \\&\quad+\frac{\lambda_{\rm max}}{K^3}E{\rm tr}\left(\tilde{\mathbf{Q}}\tilde{\mathbf{Q}}\tilde{\mathbf{Q}}^H( \mathbf{I}_K+\bar{z}\tilde{\mathbf{Q}})\right)
		\\&\leqslant\frac{\lambda_{\rm max}}{K^2}||\tilde{\mathbf{Q}}||^3\left(\left|E{\rm tr}\left(\frac{1}{K}(\mathbf{I}_K+z\tilde{\mathbf{Q}})\right)\right|\right.
		\\&\left.\qquad\qquad\qquad+\left|E{\rm tr}\left(\frac{1}{K}(\mathbf{I}_K+\bar{z}\tilde{\mathbf{Q}}^H)\right)\right|\right)
		\\&\leqslant\frac{\lambda_{\rm max}}{K^2}\frac{2}{|{\rm Im}z|^3}\left(1+\frac{|z|}{|{\rm Im}z|}\right)
		\\&\leqslant \frac{\lambda_{\rm max}}{K^2}(|z|+1)\left(\frac{2}{|{\rm Im}z|^3}+\frac{2}{|{\rm Im}z|^4}\right).
	\end{aligned}
\end{equation}

Since $\lambda_{\rm max}<\infty$, it is easy to obtain \eqref{Var1}.

Finally we establish \eqref{Var2}. Applying the Poincare-Nash inequality gives
\begin{equation}\label{166}
\small 	\begin{aligned}
		&{\rm Var}\left(\frac{1}{K}{\rm tr}\left(\tilde{\mathbf{Q}}\bm{\mathcal{Y}}^H\mathbf{D}_1\bm{\mathcal{Y}}\mathbf{D}_2\right)\right)
		\\&\leqslant \frac{\lambda_{\rm max}}{K}\sum_{i,j}E\left(\left|\frac{1}{K}\sum_{p,q}\frac{\partial \left(\tilde{\mathbf{Q}}\bm{\mathcal{Y}}^H\mathbf{D}_1\bm{\mathcal{Y}}\mathbf{D}_2\right)_{pq}}{\partial \bm{\mathcal{Y}}_{ij}}\right|^2\right.\\&
		\left. \qquad\qquad\qquad+\left|\frac{1}{K}\sum_{p,q}\frac{\partial \left(\tilde{\mathbf{Q}}\bm{\mathcal{Y}}^H\mathbf{D}_1\bm{\mathcal{Y}}\mathbf{D}_2\right)_{pq}}{\partial \overline{\bm{\mathcal{Y}}_{ij}}}\right|^2\right).
	\end{aligned}
\end{equation}

We are now in position to evaluate the partial differentials in \eqref{166}. We have
\begin{equation}\label{167}
	\begin{aligned}
		&\frac{\partial \left(\tilde{\mathbf{Q}}\bm{\mathcal{Y}}^H\mathbf{D}_1\bm{\mathcal{Y}}\mathbf{D}_2\right)_{pq}}{\partial \bm{\mathcal{Y}}_{ij}}
		\\&=\sum_{r} \frac{\partial \left(\tilde{\mathbf{Q}}_{pr}(\bm{\mathcal{Y}}^H\mathbf{D}_1\bm{\mathcal{Y}}\mathbf{D}_2)_{rq}\right)}{\partial{\bm{\mathcal{Y}}_{ij}}}
		\\& = \sum_{r} \left(\frac{\partial{\tilde{\mathbf{Q}}}_{pr}}{\partial \bm{\mathcal{Y}}_{ij}}(\bm{\mathcal{Y}}^H\mathbf{D}_1\bm{\mathcal{Y}}\mathbf{D}_2)_{rq} +\tilde{\mathbf{Q}}_{pr}\frac{\partial (\bm{\mathcal{Y}}^H\mathbf{D}_1\bm{\mathcal{Y}}\mathbf{D}_2)_{rq}}{\partial \bm{\mathcal{Y}}_{ij}} \right).
	\end{aligned}
\end{equation}

By using \eqref{162}, the first term of right hand side (RHS) of \eqref{167} can be simplified as
\begin{equation}\label{170}
	\sum_{r}\frac{\partial{\tilde{\mathbf{Q}}}_{pr}}{\partial \bm{\mathcal{Y}}_{ij}}(\bm{\mathcal{Y}}^H\mathbf{D}_1\bm{\mathcal{Y}}\mathbf{D}_2)_{rq} = -(\tilde{\mathbf{Q}}\bm{\mathcal{Y}}^H)_{pi}\left(\tilde{\mathbf{Q}}\bm{\mathcal{Y}}^H\mathbf{D}_1\bm{\mathcal{Y}}\mathbf{D}_2\right)_{jq}.
\end{equation}

In order to simplify the second term of RHS of \eqref{167}, we still need to evaluate the partial differential $\frac{\partial (\bm{\mathcal{Y}}^H\mathbf{D}_1\bm{\mathcal{Y}}\mathbf{D}_2)_{rq}}{\partial \bm{\mathcal{Y}}_{ij}} $. Expanding it gives 
\begin{equation}\label{key}
	\begin{aligned}
		\frac{\partial (\bm{\mathcal{Y}}^H\mathbf{D}_1\bm{\mathcal{Y}}\mathbf{D}_2)_{rq}}{\partial \bm{\mathcal{Y}}_{ij}}  &= \sum_{s,t,n}\frac{\overline{\bm{\mathcal{Y}}_{sr}}(\mathbf{D}_1)_{st}\bm{\mathcal{Y}}_{tn}(\mathbf{D}_2)_{nq}}{\bm{\mathcal{Y}}_{ij}}
		\\& = \sum_{s}\overline{\bm{\mathcal{Y}}_{sr}}(\mathbf{D}_1)_{si}(\mathbf{D}_2)_{jq}
		\\& = \left(\bm{\mathcal{Y}}^H\mathbf{D}_1\right)_{ri}(\mathbf{D}_2)_{jq}.
	\end{aligned}
\end{equation}

Thus the second term of \eqref{167} can be transform to
\begin{equation}\label{172}
	\sum_{r} \tilde{\mathbf{Q}}_{pr}\frac{\partial (\bm{\mathcal{Y}}^H\mathbf{D}_1\bm{\mathcal{Y}}\mathbf{D}_2)_{rq}}{\partial \bm{\mathcal{Y}}_{ij}} = \left(\tilde{\mathbf{Q}}\bm{\mathcal{Y}}^H\mathbf{D}_1\right)_{pi}(\mathbf{D}_2)_{jq}.
\end{equation}

Substituting \eqref{170} and \eqref{172} into \eqref{167},  we obtain 
\begin{equation}\label{173}
	\begin{aligned}
		&\frac{\partial \left(\tilde{\mathbf{Q}}\bm{\mathcal{Y}}^H\mathbf{D}_1\bm{\mathcal{Y}}\mathbf{D}_2\right)_{pq}}{\partial \bm{\mathcal{Y}}_{ij}}
		\\& = -(\tilde{\mathbf{Q}}\bm{\mathcal{Y}}^H)_{pi}\left(\tilde{\mathbf{Q}}\bm{\mathcal{Y}}^H\mathbf{D}_1\bm{\mathcal{Y}}\mathbf{D}_2\right)_{jq} +  \left(\tilde{\mathbf{Q}}\bm{\mathcal{Y}}^H\mathbf{D}_1\right)_{pi}(\mathbf{D}_2)_{jq}.
	\end{aligned}
\end{equation}

Proceeding similarly, we can obtain
\begin{equation}\label{174}
	\begin{aligned}
		&\frac{\partial \left(\tilde{\mathbf{Q}}\bm{\mathcal{Y}}^H\mathbf{D}_1\bm{\mathcal{Y}}\mathbf{D}_2\right)_{pq}}{\partial \overline{\bm{\mathcal{Y}}_{ij}}}
		\\&=-\tilde{\mathbf{Q}}_{pj}\left(\bm{\mathcal{Y}}\tilde{\mathbf{Q}}\bm{\mathcal{Y}}^H\mathbf{D}_1\bm{\mathcal{Y}}\mathbf{D}_2\right)_{iq} + \tilde{\mathbf{Q}}_{pj}(\mathbf{D}_1\bm{\mathcal{Y}}\mathbf{D}_2)_{iq}
	\end{aligned}
\end{equation}
%
%
%

Substituting \eqref{173} and \eqref{174} into \eqref{166}, we obtain
\begin{equation}\label{}
\small	\begin{aligned}
		&{\rm Var}\left(\frac{1}{K}{\rm tr}\left(\tilde{\mathbf{Q}}\bm{\mathcal{Y}}^H\mathbf{D}_1\bm{\mathcal{Y}}\mathbf{D}_2\right)\right)
		\\&\leqslant \frac{\lambda_{\rm max}}{K^3}\sum_{i,j} E\left(\left|(\tilde{\mathbf{Q}}\bm{\mathcal{Y}}^H\mathbf{D}_1\bm{\mathcal{Y}}\mathbf{D}_2\tilde{\mathbf{Q}}\bm{\mathcal{Y}}^H)_{ji}\right|^2 \right.
		\\&\left. \quad+ \left|(\mathbf{D}_2\tilde{\mathbf{Q}}\bm{\mathcal{Y}}^H\mathbf{D}_1)_{ji}\right|^2 +\left|(\bm{\mathcal{Y}}\tilde{\mathbf{Q}}\bm{\mathcal{Y}}^H\mathbf{D}_1\bm{\mathcal{Y}}\mathbf{D}_2\tilde{\mathbf{Q}})_{ij}\right|^2+\left|(\mathbf{D}_1\bm{\mathcal{Y}}\mathbf{D}_2\tilde{\mathbf{Q}})_{ij}\right|^2\right)
		\\&\leqslant \frac{\lambda_{\rm max}}{K^3} E{\rm tr}(\tilde{\mathbf{Q}}\bm{\mathcal{Y}}^H\mathbf{D}_1\bm{\mathcal{Y}}\mathbf{D}_2\tilde{\mathbf{Q}}\bm{\mathcal{Y}}^H\bm{\mathcal{Y}}\tilde{\mathbf{Q}}^H\mathbf{D}_2^H\bm{\mathcal{Y}}^H\mathbf{D}_1^H\bm{\mathcal{Y}}\tilde{\mathbf{Q}}^H) 
		\\&\quad+\frac{\lambda_{\rm max}}{K^3}E{\rm tr}(\mathbf{D}_2\tilde{\mathbf{Q}}\bm{\mathcal{Y}}^H\mathbf{D}_1\mathbf{D}_1^H\bm{\mathcal{Y}}\tilde{\mathbf{Q}}^H\mathbf{D}_2^H) 
		\\&\quad+\frac{\lambda_{\rm max}}{K^3}E{\rm tr}(\bm{\mathcal{Y}}\tilde{\mathbf{Q}}\bm{\mathcal{Y}}^H\mathbf{D}_1\bm{\mathcal{Y}}\mathbf{D}_2\tilde{\mathbf{Q}}\tilde{\mathbf{Q}}^H\mathbf{D}_2^H\bm{\mathcal{Y}}^H\mathbf{D}_1^H\bm{\mathcal{Y}}\tilde{\mathbf{Q}}^H\bm{\mathcal{Y}}^H) 
		\\&\quad+\frac{\lambda_{\rm max}}{K^3}E{\rm tr}(\mathbf{D}_1\bm{\mathcal{Y}}\mathbf{D}_2\tilde{\mathbf{Q}}\tilde{\mathbf{Q}}^H\mathbf{D}_2^H\bm{\mathcal{Y}}^H\mathbf{D}_1^H)
		\\&\leqslant \frac{\lambda_{\rm max}||\mathbf{D}_1||^2||\mathbf{D}_2||^2}{K^3} \left(E{\rm tr}(\tilde{\mathbf{Q}}\bm{\mathcal{Y}}^H\bm{\mathcal{Y}}\tilde{\mathbf{Q}}\bm{\mathcal{Y}}^H\bm{\mathcal{Y}}\tilde{\mathbf{Q}}^H\bm{\mathcal{Y}}^H\bm{\mathcal{Y}}\tilde{\mathbf{Q}}^H)
		\right.\\&\quad\left.+E{\rm tr}(\tilde{\mathbf{Q}}\bm{\mathcal{Y}}^H\bm{\mathcal{Y}}\tilde{\mathbf{Q}}^H) +  E{\rm tr}(\bm{\mathcal{Y}}\tilde{\mathbf{Q}}\bm{\mathcal{Y}}^H\mathbf{Y}'\tilde{\mathbf{Q}}\tilde{\mathbf{Q}}^H\bm{\mathcal{Y}}^H\bm{\mathcal{Y}}\tilde{\mathbf{Q}}^H\bm{\mathcal{Y}}^H) 
		\right.\\&\quad \left.+ E{\rm tr}(\bm{\mathcal{Y}}\tilde{\mathbf{Q}}\tilde{\mathbf{Q}}^H\bm{\mathcal{Y}}^H)\right).
	\end{aligned}
\end{equation}

By using the identities \eqref{164}, we obtain
\begin{equation}\label{122}
	\begin{aligned}
		&{\rm Var}\left(\frac{1}{K}{\rm tr}\left(\tilde{\mathbf{Q}}\bm{\mathcal{Y}}^H\mathbf{D}_1\bm{\mathcal{Y}}\mathbf{D}_2\right)\right)
		\\&\leqslant \frac{2\lambda_{\rm max}||\mathbf{D}_1||^2||\mathbf{D}_2||^2||\tilde{\mathbf{Q}}||}{K^2} E\left[\frac{1}{K}{\rm tr}\left((\mathbf{I}_K+z\tilde{\mathbf{Q}})^2(\mathbf{I}_K+\bar{z}\tilde{\mathbf{Q}}^H)\right.\right.
		\\&\left.\left.\qquad\qquad\quad\qquad\qquad\qquad\qquad+(\mathbf{I}_K+z\tilde{\mathbf{Q}})\right)\right]
		\\&\leqslant \frac{2\lambda_{\rm max}||\mathbf{D}_1||^2||\mathbf{D}_2||^2}{K^2}\frac{1}{|{\rm Im}z|}\left(2+\frac{4|z|}{|{\rm Im}z|}+\frac{3|z|^2}{|{\rm Im}z|^2} + \frac{|z|^3}{|{\rm Im}z|^3}\right)
		\\&	\leqslant\frac{2\lambda_{\rm max}||\mathbf{D}_1||^2||\mathbf{D}_2||^2}{K^2}\left(|z|^3+3|z|^2+4|z|+2\right)
		\\&\quad\times\left(\frac{1}{|{\rm Im}z|^4}+\frac{1}{|{\rm Im}z|^3}+\frac{1}{|{\rm Im}z|^2}+\frac{1}{|{\rm Im}z|}\right).
	\end{aligned}
\end{equation}

Since $\lambda_{\rm max}<\infty$, and $||\mathbf{D}_1||,||\mathbf{D}_2||<\infty$, we can obtain \eqref{Var2} from \eqref{122}.

\section*{Appendix E\\ Proof of Lemma \ref{Lemma Var Q'}}
Here we give the proof of Lemma \ref{Lemma Var Q'} based on the Poincare-Nash inequality. 

We first establish \eqref{Lemma Var Q' 1}. 
We introduce a differentiation formula. Let $\mathbf{A}$ be an $N\times N$ complex matrix and let $\mathbf{H}(\mathbf{A}) = (\mathbf{I}_N+\mathbf{A})^{-2}$. Let ${\bm \delta}\mathbf{ A}$ be a perturbation of $\mathbf{A}$. Then 
\begin{equation}\label{differentiation formula}
	\begin{aligned}
		\mathbf{H}(\mathbf{A}+{\bm \delta}\mathbf{ A}) &= \mathbf{H}(\mathbf{A}) - \sqrt{\mathbf{H}(\mathbf{A})}{\bm \delta}\mathbf{ A}\mathbf{H}(\mathbf{A}) 
		\\&\quad- \mathbf{H}(\mathbf{A}){\bm \delta}\mathbf{ A}\sqrt{\mathbf{H}(\mathbf{A})} +o(\Vert{\bm \delta}\mathbf{ A}\Vert)
	\end{aligned}
\end{equation}
where $o(\Vert{\bm \delta}\mathbf{ A}\Vert)$ is negligible with respect to $\Vert{\bm \delta}\mathbf{ A}\Vert$ in a neighborhood of 0. Using \eqref{differentiation formula}, we obtain
\begin{equation}\label{key}
	\begin{aligned}
		\frac{\partial{(\mathbf{Q}')_{pq}}}{\partial{\bm{\mathcal{Y}}_{ij}}} =& -\frac{\partial (\mathbf{Q}\bm{\mathcal{Y}}\bm{\mathcal{Y}}^H\mathbf{Q}')_{pq}}{\partial{\bm{\mathcal{Y}}_{ij}}} - \frac{\partial (\mathbf{Q}'\bm{\mathcal{Y}}\bm{\mathcal{Y}}^H\mathbf{Q})_{pq}}{\partial{\bm{\mathcal{Y}}_{ij}}} 
		\\& = -\sum_{r,s,t}\frac{\partial( \mathbf{Q}_{pr}\bm{\mathcal{Y}}_{rs}\overline{\bm{\mathcal{Y}}_{ts}}(\mathbf{Q}')_{tq}) }{\partial{\bm{\mathcal{Y}}_{ij}}} 
		\\&\quad- \sum_{r,s,t}\frac{\partial( (\mathbf{Q}')_{pr}\bm{\mathcal{Y}}_{rs}\overline{\bm{\mathcal{Y}}_{ts}}\mathbf{Q}_{tq}) }{\partial{\bm{\mathcal{Y}}_{ij}}}
		\\& = -\mathbf{Q}_{pi}\left(\bm{\mathcal{Y}}^H\mathbf{Q}'\right)_{jq}-(\mathbf{Q}')_{pi}(\bm{\mathcal{Y}}^H\mathbf{Q})_{jq}
	\end{aligned}
\end{equation}

Likewise, we have
\begin{equation}\label{key}
	\frac{\partial{(\mathbf{Q}')_{pq}}}{\partial{\overline{\bm{\mathcal{Y}}_{ij}}}} = - (\mathbf{Q}\bm{\mathcal{Y}})_{pj}(\mathbf{Q}')_{iq} - (\mathbf{Q}'\bm{\mathcal{Y}})_{pj}\mathbf{Q}_{iq}
\end{equation}

Hence 
\begin{equation}\label{key}
	\begin{aligned}
		&\frac{\partial (\mathbf{C}_1\mathbf{Q}'\mathbf{C}_2)_{pq}}{\partial \bm{\mathcal{Y}}_{ij}} 
		\\&\quad= -(\mathbf{C}_1\mathbf{Q})_{pi}\left(\bm{\mathcal{Y}}^H\mathbf{Q}'\mathbf{C}_2\right)_{jq}-(\mathbf{C}_1\mathbf{Q}')_{pi}(\bm{\mathcal{Y}}^H\mathbf{Q}\mathbf{C}_2)_{jq}
	\end{aligned}
\end{equation}
and
\begin{equation}\label{key}
	\frac{\partial (\mathbf{C}_1\mathbf{Q}'\mathbf{C}_2)_{pq}}{\partial \overline{\bm{\mathcal{Y}}_{ij}}} =- (\mathbf{C}_1\mathbf{Q}\bm{\mathcal{Y}})_{pj}(\mathbf{Q}'\mathbf{C}_2)_{iq} - (\mathbf{C}_1\mathbf{Q}'\bm{\mathcal{Y}})_{pj}(\mathbf{Q}\mathbf{C}_2)_{iq}
\end{equation}

Applying the Poincare-Nash inequality, we have
\begin{equation}\label{key}
	\begin{aligned}
		&{\rm Var}\left(\frac{1}{K}{\rm tr} (\mathbf{C}_1\mathbf{Q}'\mathbf{C}_2)\right)
		\\& \leqslant \frac{\lambda_{\rm max}}{K}\sum_{i,j}E\left(\left|\frac{1}{K}\sum_{p,q}\frac{\partial (\mathbf{C}_1\mathbf{Q}'\mathbf{C}_2)_{pq}}{\partial \bm{\mathcal{Y}}_{ij}}\right|^2\right. \\&\left.\qquad\qquad\qquad\quad+\left|\frac{1}{K}\sum_{p,q}\frac{\partial (\mathbf{C}_1\mathbf{Q}'\mathbf{C}_2)_{pq}}{\partial \overline{\bm{\mathcal{Y}}_{ij}}}\right|^2\right)
		\\&\leqslant \frac{\lambda_{\rm max}}{K^3}\sum_{i,j}E\left(\left|(\bm{\mathcal{Y}}^H\mathbf{Q}'\mathbf{C}_2\mathbf{C}_1\mathbf{Q})_{ji}\right|^2+\left|(\bm{\mathcal{Y}}^H\mathbf{Q}\mathbf{C}_2\mathbf{C}_1\mathbf{Q}')_{ji}\right|^2\right.
		\\&\left.\qquad\qquad\qquad\quad +\left|(\mathbf{Q}'\mathbf{C}_2\mathbf{C}_1\mathbf{Q}\bm{\mathcal{Y}})_{ij}\right|^2+\left|(\mathbf{Q}\mathbf{C}_2\mathbf{C}_1\mathbf{Q}'\bm{\mathcal{Y}})_{ij}\right|^2\right)
		\\&\leqslant 2\frac{\lambda_{\rm max}}{K^3}E{\rm tr}\left(\bm{\mathcal{Y}}^H\mathbf{Q}'\mathbf{C}_2\mathbf{C}_1\mathbf{Q}\mathbf{Q}^H\mathbf{C}_1^H\mathbf{C}_2^H\mathbf{Q}'^H\bm{\mathcal{Y}}\right.
		\\&\left.\qquad\qquad\qquad+\mathbf{Q}'\mathbf{C}_2\mathbf{C}_1\mathbf{Q}\bm{\mathcal{Y}}\bm{\mathcal{Y}}^H\mathbf{Q}^H\mathbf{C}_1^H\mathbf{C}_2^H\mathbf{Q}'^H\right)
	\end{aligned}
\end{equation}

Noticing that $\mathbf{Q}' = \mathbf{Q}\mathbf{Q}$ and using the identities \eqref{QYY' = I+zQ},  we have
\begin{equation}\label{key}
	\begin{aligned}
		&{\rm Var}\left(\frac{1}{K}{\rm tr} (\mathbf{C}_1\mathbf{Q}'\mathbf{C}_2)\right)
		\\& \leqslant 2\frac{\lambda_{\rm max}}{K^3}E{\rm tr}\left(\mathbf{Q}'\mathbf{C}_2\mathbf{C}_1\mathbf{Q}\mathbf{Q}^H\mathbf{C}_1^H\mathbf{C}_2^H\mathbf{Q}^H(\mathbf{I}_N+\bar{z}\mathbf{Q}^H)\right.
		\\&\left.\quad+\mathbf{Q}'\mathbf{C}_2\mathbf{C}_1(\mathbf{I}_N+z\mathbf{Q})\mathbf{Q}^H\mathbf{C}_1^H\mathbf{C}_2^H\mathbf{Q}'^H\right)
		\\&\leqslant \frac{2\lambda_{\rm max}}{K^2}\frac{\Vert\mathbf{C}_1\Vert^2\Vert\mathbf{C}_2\Vert^2}{|{\rm Im}z|^5}E\left[\frac{c_K}{N}{\rm tr}\left((\mathbf{I}_N+\bar{z}\mathbf{Q}^H)+(\mathbf{I}_N+z\mathbf{Q})\right)\right]
		\\& \leqslant 4c_K\frac{\lambda_{\rm max}}{K^2}\frac{\Vert\mathbf{C}_1\Vert^2\Vert\mathbf{C}_2\Vert^2}{|{\rm Im}z|^5}\left(1+\frac{|z|}{|{\rm Im}z|}\right)
		\\&\leqslant  4c_K\frac{\lambda_{\rm max}}{K^2}{\Vert\mathbf{C}_1\Vert^2\Vert\mathbf{C}_2\Vert^2}(1+|z|)\left(|{\rm Im}z|^{-6}+|{\rm Im}z|^{-5}\right)
	\end{aligned}
\end{equation}
which ends the proof of \eqref{Lemma Var Q' 1}.

Next we establish \eqref{Lemma Var Q' 3}. We note that
\begin{equation}\label{key}
	\begin{aligned}
		\frac{\partial{(\tilde{\mathbf{Q}}')_{pq}}}{\partial{\bm{\mathcal{Y}}_{ij}}} &= -\frac{\partial (\tilde{\mathbf{Q}}\bm{\mathcal{Y}}^H\bm{\mathcal{Y}}\tilde{\mathbf{Q}}')_{pq}}{\partial{\bm{\mathcal{Y}}_{ij}}} - \frac{\partial (\mathbf{Q}'\bm{\mathcal{Y}}\bm{\mathcal{Y}}^H\mathbf{Q})_{pq}}{\partial{\bm{\mathcal{Y}}_{ij}}} 
		\\& = -\sum_{r,s,t}\frac{\partial( \tilde{\mathbf{Q}}_{pr}\overline{\bm{\mathcal{Y}}_{sr}}\bm{\mathcal{Y}}_{st}(\tilde{\mathbf{Q}}')_{tq}) }{\partial{\bm{\mathcal{Y}}_{ij}}} 
		\\&\quad- \sum_{r,s,t}\frac{\partial( (\tilde{\mathbf{Q}}')_{pr}\overline{\bm{\mathcal{Y}}_{sr}}\bm{\mathcal{Y}}_{st}\tilde{\mathbf{Q}}_{tq}) }{\partial{\bm{\mathcal{Y}}_{ij}}}
		\\& = -(\tilde{\mathbf{Q}}\bm{\mathcal{Y}}^H)_{pi}(\tilde{\mathbf{Q}}')_{jq}-(\tilde{\mathbf{Q}}'\bm{\mathcal{Y}}^H)_{pi}\tilde{\mathbf{Q}}_{jq}
	\end{aligned}
\end{equation}

Likewise, we have
\begin{equation}\label{key}
	\begin{aligned}
		\frac{\partial{(\tilde{\mathbf{Q}}')_{pq}}}{\partial{\overline{\bm{\mathcal{Y}}_{ij}}}} &= -\tilde{\mathbf{Q}}_{pj}(\bm{\mathcal{Y}}\tilde{\mathbf{Q}}')_{iq}-(\tilde{\mathbf{Q}}')_{pj}(\bm{\mathcal{Y}}\tilde{\mathbf{Q}})_{iq}
	\end{aligned}
\end{equation}

Using the Poincare-Nash inequality gives
\begin{equation}\label{key}
	\begin{aligned}
		&{\rm Var}\left(\frac{1}{K}{\rm tr}\tilde{\mathbf{Q}}'\right)
		\\&\leqslant \frac{\lambda_{\rm max}}{K}\sum_{i,j}E\left(\left|\frac{1}{K}\sum_{p,q}\frac{\partial (\tilde{\mathbf{Q}}')_{pq}}{\partial \bm{\mathcal{Y}}_{ij}} \right|^2+\left|\frac{1}{K}\sum_{p,q}\frac{\partial (\tilde{\mathbf{Q}}')_{pq}}{\partial \overline{\bm{\mathcal{Y}}_{ij}}} \right|^2\right)
		\\& \leqslant \frac{\lambda_{\rm max}}{K^3}\sum_{i,j}E\left(|(\tilde{\mathbf{Q}}'\tilde{\mathbf{Q}}\bm{\mathcal{Y}}^H)_{ji}|^2+|(\tilde{\mathbf{Q}}\tilde{\mathbf{Q}}'\bm{\mathcal{Y}}^H)_{ji}|^2+|(\bm{\mathcal{Y}}\tilde{\mathbf{Q}}'\tilde{\mathbf{Q}})_{ij}|^2\right.
		\\&\qquad\qquad\qquad\quad\left.+|(\bm{\mathcal{Y}}\tilde{\mathbf{Q}}\tilde{\mathbf{Q}}')_{ij}|^2\right)
		\\&\leqslant \frac{\lambda_{\rm max}}{K^3}E{\rm tr}\left(\tilde{\mathbf{Q}}'\tilde{\mathbf{Q}}\bm{\mathcal{Y}}^H\bm{\mathcal{Y}}\tilde{\mathbf{Q}}^H\tilde{\mathbf{Q}}'^H + \tilde{\mathbf{Q}}\tilde{\mathbf{Q}}'\bm{\mathcal{Y}}^H\bm{\mathcal{Y}}\tilde{\mathbf{Q}}'^H\tilde{\mathbf{Q}}^H\right.
		\\&\qquad\qquad\quad\left.+\bm{\mathcal{Y}}\tilde{\mathbf{Q}}'\tilde{\mathbf{Q}}\tilde{\mathbf{Q}}^H\tilde{\mathbf{Q}}'^H\bm{\mathcal{Y}}^H+\bm{\mathcal{Y}}\tilde{\mathbf{Q}}\tilde{\mathbf{Q}}'\tilde{\mathbf{Q}}'^H\tilde{\mathbf{Q}}^H\bm{\mathcal{Y}}^H\right)
	\end{aligned}
\end{equation}

Using the fact that $\tilde{\mathbf{Q}}' = \tilde{\mathbf{Q}}\tilde{\mathbf{Q}}$ and the identities \eqref{164}, we have
\begin{equation}\label{key}
	\begin{aligned}
		{\rm Var}\left(\frac{1}{K}{\rm tr}\tilde{\mathbf{Q}}'\right)
		&\leqslant 4\frac{\lambda_{\rm max}}{K^2}\Vert\tilde{\mathbf{Q}}\Vert^3\Vert\tilde{\mathbf{Q}}'\Vert E\left(\frac{1}{K}{\rm tr}\left(\mathbf{I}_K+\bar{z}{\tilde{\mathbf{Q}}}\right)\right)
		\\&\leqslant4 \frac{\lambda_{\rm max}}{K^2}\frac{1}{|{\rm Im}z|^{5}}\left(1+\frac{|z|}{|{\rm Im}z|}\right)
		\\&\leqslant 4\frac{\lambda_{\rm max}}{K^2}(1+|z|)\left(|{\rm Im}z|^{-6}+|{\rm Im}z|^{-5}\right)
	\end{aligned}
\end{equation}
which completes the proof of \eqref{Lemma Var Q' 3}.

Finally, we establish \eqref{Lemma Var Q' 2}. We first note that
\begin{equation}\label{key}
	\small\begin{aligned}
		&\frac{\partial(\tilde{\mathbf{Q}}'\bm{\mathcal{Y}}^H\mathbf{D}_1\bm{\mathcal{Y}}\mathbf{D}_2)_{pq}}{\partial \bm{\mathcal{Y}}_{ij} } 
		\\&= \sum_r \frac{\partial (\tilde{\mathbf{Q}}')_{pr}}{\partial \bm{\mathcal{Y}}_{ij}}(\bm{\mathcal{Y}}^H\mathbf{D}_1\bm{\mathcal{Y}}\mathbf{D}_2)_{rq} +\sum_r (\tilde{\mathbf{Q}}')_{pr}\frac{\partial (\bm{\mathcal{Y}}^H\mathbf{D}_1\bm{\mathcal{Y}}\mathbf{D}_2)_{rq}}{\partial \bm{\mathcal{Y}}_{ij}}
		\\& =  -(\tilde{\mathbf{Q}}\bm{\mathcal{Y}}^H)_{pi}\left(\tilde{\mathbf{Q}}'\bm{\mathcal{Y}}^H\mathbf{D}_1\bm{\mathcal{Y}}\mathbf{D}_2\right)_{jq}-(\tilde{\mathbf{Q}}'\bm{\mathcal{Y}}^H)_{pi}(\tilde{\mathbf{Q}}\bm{\mathcal{Y}}^H\mathbf{D}_1\bm{\mathcal{Y}}\mathbf{D}_2)_{jq} 
		\\&\quad+(\tilde{\mathbf{Q}}'\bm{\mathcal{Y}}^H\mathbf{D}_1)_{pi}(\mathbf{D}_2)_{jq}.
	\end{aligned}
\end{equation}

Similarly, we have
\begin{equation}\label{key}
	\begin{aligned}
		&\frac{\partial(\mathbf{Q}'\bm{\mathcal{Y}}^H\mathbf{D}_1\bm{\mathcal{Y}}\mathbf{D}_2)_{pq}}{\partial \overline{\bm{\mathcal{Y}}_{ij}} } 
		\\&=- \tilde{\mathbf{Q}}_{pj}(\bm{\mathcal{Y}}\tilde{\mathbf{Q}}'\bm{\mathcal{Y}}^H\mathbf{D}_1\bm{\mathcal{Y}}\mathbf{D}_2)_{iq} - (\tilde{\mathbf{Q}}')_{pj}(\bm{\mathcal{Y}}\tilde{\mathbf{Q}}\bm{\mathcal{Y}}^H\mathbf{D}_1\bm{\mathcal{Y}}\mathbf{D}_2)_{iq}
		\\&\quad+ (\tilde{\mathbf{Q}}')_{pj}(\mathbf{D}_1\bm{\mathcal{Y}}\mathbf{D}_2)_{iq}.
	\end{aligned}
\end{equation}

The Poincare-Nash inequality gives
\begin{equation}\label{key}
	\begin{aligned}
		&{\rm Var}\left(\frac{1}{K}{\rm tr}(\tilde{\mathbf{Q}}'\bm{\mathcal{Y}}^H\mathbf{D}_1\bm{\mathcal{Y}}\mathbf{D}_2)\right)
		\\&\leqslant \frac{\lambda_{\rm max}}{K}\sum_{i,j}E\left(\left|\frac{1}{K}\frac{\partial(\tilde{\mathbf{Q}}'\bm{\mathcal{Y}}^H\mathbf{D}_1\bm{\mathcal{Y}}\mathbf{D}_2)_{pq}}{\partial \bm{\mathcal{Y}}_{ij} } \right|^2\right.
		\\&\left.\quad+\left|\frac{1}{K}\frac{\partial(\tilde{\mathbf{Q}}'\bm{\mathcal{Y}}^H\mathbf{D}_1\bm{\mathcal{Y}}\mathbf{D}_2)_{pq}}{\partial \overline{\bm{\mathcal{Y}}_{ij}} } \right|^2\right)
		\\&	\leqslant \frac{\lambda_{\rm max}}{K^3}\sum_{i,j}E\left(\left|(\tilde{\mathbf{Q}}'\bm{\mathcal{Y}}^H\mathbf{D}_1\bm{\mathcal{Y}}\mathbf{D}_2 \tilde{\mathbf{Q}}\bm{\mathcal{Y}}^H	)_{ji}\right|^2
		\right.\\&\left.\quad+\left|(\tilde{\mathbf{Q}}\bm{\mathcal{Y}}^H\mathbf{D}_1\bm{\mathcal{Y}}\mathbf{D}_2 \tilde{\mathbf{Q}}'\bm{\mathcal{Y}}^H)_{ji}\right|^2+ \left|(\mathbf{D}_2\tilde{\mathbf{Q}}'\bm{\mathcal{Y}}^H\mathbf{D}_1)_{ji}\right|^2\right.
		\\&\left.\quad+\left|(\bm{\mathcal{Y}}\tilde{\mathbf{Q}}'\bm{\mathcal{Y}}^H\mathbf{D}_1\bm{\mathcal{Y}}\mathbf{D}_2\tilde{\mathbf{Q}})_{ij}\right|^2+\left|(\bm{\mathcal{Y}}\tilde{\mathbf{Q}}\bm{\mathcal{Y}}^H\mathbf{D}_1\bm{\mathcal{Y}}\mathbf{D}_2\tilde{\mathbf{Q}}')_{ij}\right|^2\right.
		\\&\left.\quad+\left|(\mathbf{D}_1\bm{\mathcal{Y}}\mathbf{D}_2\tilde{\mathbf{Q}}')_{ij}\right|^2\right)
		\\& \leqslant \frac{\lambda_{\rm max}}{K^3}E{\rm tr}\left(\tilde{\mathbf{Q}}'\bm{\mathcal{Y}}^H\mathbf{D}_1\bm{\mathcal{Y}}\mathbf{D}_2 \tilde{\mathbf{Q}}\bm{\mathcal{Y}}^H\bm{\mathcal{Y}}\tilde{\mathbf{Q}}^H\mathbf{D}_2^H\bm{\mathcal{Y}}^H\mathbf{D}_1^H\bm{\mathcal{Y}}\tilde{\mathbf{Q}}'^H\right.
		\\&\left. \qquad\qquad\quad+ \tilde{\mathbf{Q}}\bm{\mathcal{Y}}^H\mathbf{D}_1\bm{\mathcal{Y}}\mathbf{D}_2 \tilde{\mathbf{Q}}'\bm{\mathcal{Y}}^H\bm{\mathcal{Y}}\tilde{\mathbf{Q}}'^H\mathbf{D}_2^H\bm{\mathcal{Y}}^H\mathbf{D}_1^H\bm{\mathcal{Y}}\tilde{\mathbf{Q}}^H\right.
		\\&\left. \qquad\qquad\quad+\mathbf{D}_2\tilde{\mathbf{Q}}'\bm{\mathcal{Y}}^H\mathbf{D}_1\mathbf{D}_1^H\bm{\mathcal{Y}}\tilde{\mathbf{Q}}'^H\mathbf{D}_2^H\right.
		\\&\left. \qquad\qquad\quad+ \bm{\mathcal{Y}}\tilde{\mathbf{Q}}'\bm{\mathcal{Y}}^H\mathbf{D}_1\bm{\mathcal{Y}}\mathbf{D}_2\tilde{\mathbf{Q}}\tilde{\mathbf{Q}}^H\mathbf{D}_2^H\bm{\mathcal{Y}}^H\mathbf{D}_1^H\bm{\mathcal{Y}}\tilde{\mathbf{Q}}'^H\bm{\mathcal{Y}}^H\right.
		\\&\left. \qquad\qquad\quad+\bm{\mathcal{Y}}\tilde{\mathbf{Q}}\bm{\mathcal{Y}}^H\mathbf{D}_1\bm{\mathcal{Y}}\mathbf{D}_2\tilde{\mathbf{Q}}'\tilde{\mathbf{Q}}'^H\mathbf{D}_2^H\bm{\mathcal{Y}}^H\mathbf{D}_1^H\bm{\mathcal{Y}}\tilde{\mathbf{Q}}^H\bm{\mathcal{Y}}^H\right.
		\\&\left. \qquad\qquad\quad+
		\mathbf{D}_1\bm{\mathcal{Y}}\mathbf{D}_2\tilde{\mathbf{Q}}'\tilde{\mathbf{Q}}'^H\mathbf{D}_2^H\bm{\mathcal{Y}}^H\mathbf{D}_1^H
		\right)
		\\& \leqslant 2\frac{\lambda_{\rm max}}{K^2}\Vert\mathbf{D}_1\Vert^2\Vert\mathbf{D}_2\Vert^2\Vert\tilde{\mathbf{Q}}\Vert\Vert\tilde{\mathbf{Q}}'\Vert E\left[\frac{1}{K}{\rm tr}\left((\mathbf{I}_K+z\tilde{\mathbf{Q}})\right.\right.
		\\&\left.\left.\qquad\qquad\qquad\qquad\qquad\qquad+ 2(\mathbf{I}_K+z\tilde{\mathbf{Q}})^2(\mathbf{I}_K+\bar{z}\tilde{\mathbf{Q}}^H)\right)\right]
		\\&\leqslant 2\frac{\lambda_{\rm max}}{K^2} \frac{\Vert\mathbf{D}_1\Vert^2\Vert\mathbf{D}_2\Vert^2}{|{\rm Im}z|^3}\left(3+7\frac{|z|}{|{\rm Im}z|}+6\frac{|z|^2}{|{\rm Im}z|^2}+2\frac{|z|^3}{|{\rm Im}z|^3}\right)
		\\&\leqslant 2\frac{\lambda_{\rm max}}{K^2} \Vert\mathbf{D}_1\Vert^2\Vert\mathbf{D}_2\Vert^2\left(2|z|^3+6|z|^2+7|z|+3\right)
		\\&\quad\quad\times \left(|{\rm Im}z|^{-6}+|{\rm Im}z|^{-5}+|{\rm Im}z|^{-4}+|{\rm Im}z|^{-3}\right),
	\end{aligned}
\end{equation}
which completes the proof of \eqref{Lemma Var Q' 2}.
\\
\section*{Appendix F\\Proof of Corollary 1}
We first establish (23). Write
\begin{equation}\label{E(bR-b)^2}
	\begin{aligned}
		&E\left[(b_{{\hat{\mathbf{R}}}}(z) - b_K(z))^2\right]
		\\& = E\left[(b_{{\hat{\mathbf{R}}}}(z)-Eb_{{\hat{\mathbf{R}}}}(z)+Eb_{{\hat{\mathbf{R}}}}(z)-b_K(z))^2\right]
		\\& = {\rm Var}(b_{{\hat{\mathbf{R}}}}(z)) + \left(Eb_{{\hat{\mathbf{R}}}}(z)-b_K(z)\right)^2.
	\end{aligned}
\end{equation}

Using \eqref{Var0}, we obtain
\begin{equation}\label{178}
	\begin{aligned}
		{\rm Var}(b_{\hat{\mathbf{R}}}(z)) = {\rm Var}\left(\frac{1}{c_KK}{\rm tr}\mathbf{Q}(z)\right)&\leqslant \frac{1}{K^2}P_1(|z|)P_4(|{\rm Im}z|^{-1})
		\\&\leqslant \frac{1}{K^2}P_{18}(|z|)P_{24}(|{\rm Im}z|^{-1})
	\end{aligned}
\end{equation}

Using (21), we obtain
\begin{equation}\label{179}
	\begin{aligned}
		\left(Eb_{{\hat{\mathbf{R}}}}(z)-b_K(z)\right)^2 &\leqslant \frac{1}{K^4}P_{18}(|z|)P_{24}(|{\rm Im}z|^{-1})
		\\&\leqslant\frac{1}{K^2}P_{18}(|z|)P_{24}(|{\rm Im}z|^{-1})
	\end{aligned}
\end{equation}

Hence we obtain (23) by using \eqref{179} and \eqref{178} in \eqref{E(bR-b)^2}. The proof of (24) is almost identical and hence is omitted.

\section*{Appendix G\\Proof of Proposition 1}
Here we focus on calculating the integrals $I_1$ (i.e., Eq. (83)) and $I_2$ (i.e., Eq. (84)).

We first evaluate the integral $I_1$. We need to distinct two cases, i.e., $m\ne n$ and $m=n$, to discuss. 

We first deal with the case $m\ne n$. Define 
\begin{equation}\label{key}
	r(z_1,z_2)= b_{\underline{\hat{\mathbf{R}}}}(1/z_1)b_{\underline{\hat{\mathbf{R}}}}(1/z_2)\frac{ \frac{\partial b_{\underline{\hat{\mathbf{R}}}}(1/z_1)}{\partial z_1} \frac{\partial b_{\underline{\hat{\mathbf{R}}}}(1/z_2)}{\partial z_2} }{(b_{\underline{\hat{\mathbf{R}}}}(1/z_1)-b_{\underline{\hat{\mathbf{R}}}}(1/z_2))^2}.
\end{equation}

Let $z_2$ be fixed. By observing that 
\begin{equation}\label{key}
	b_{\underline{\hat{\mathbf{R}}}}(1/z) = -\frac{c_K}{N}\sum_{i=1}^N\frac{\hat{\rho}_iz}{\hat{\rho}_i-z}-(1-c_K)z
\end{equation}
and
\begin{equation}\label{key}
	\frac{\partial{b_{\underline{\hat{\mathbf{R}}}}(1/z)}}{\partial z}= -\frac{c_K}{N}\sum_{i=1}^N\left(\frac{\hat{\rho}_i}{\hat{\rho}_i-z}\right)^2-(1-c_K),
\end{equation}
one can check that $\{\hat{\rho}_k,k\in\mathcal{N}_m\}$ are the first-order poles of $r(z_1,z_2)$ at $z_1$ within the contour $\Gamma_{m}^+$. Taking $z_1\to\hat{\rho}_k$ and applying the residue theorem give
\begin{equation}\label{84}
	\begin{aligned}
		r(\cdot,z_2)={\rm Res}(r(z_1,z_2),\hat{\rho}_k) = -b_{\underline{\hat{\mathbf{R}}}}(1/z_2) \frac{\partial b_{\underline{\hat{\mathbf{R}}}}(1/z_2)}{\partial z_2}
	\end{aligned}
\end{equation}
where $k\in \mathcal{N}_m$. 

Now we need to analyze the poles of $r(\cdot,z_2)$. A straightforward analysis shows that $r(\cdot,z_2)$ has third-order poles at $\{\hat{\rho}_l,l\in\mathcal{N}_n\}$ within contour $\Gamma_n^{'+}$.   Then by using the residue theorem again, we can calculate the residue of $r(\cdot,z_2)$ at $\hat{\rho}_l$ by
\begin{equation}\label{85}
	\begin{aligned}
		&{\rm Res}\left(r(\cdot,z_2),\hat{\rho}_l\right) \\&= \lim_{z_2\to\hat{\rho}_l}\frac{1}{2}\frac{\partial^2}{\partial z_2^2}\left(-(z_2-\hat{\rho}_l)^3b_{\underline{\hat{\mathbf{R}}}}(1/z_2) \frac{\partial b_{\underline{\hat{\mathbf{R}}}}(1/z_2)}{\partial z_2} \right)
		\\&=0.
	\end{aligned}
\end{equation}

In addition, we note that in the case $m\ne n$, the contours $\Gamma_m^+$ and $\Gamma_n^{'+}$ never intersect, which indicates that $z_1=z_2$ is not the pole of $r(z_1,z_2)$. Hence we obtain $I_1=0$ in the case $m\ne n$.

Next we calculate $I_1$ in the case $m =n$. 
Let $z_2$ be fixed, and then $r(z_1,z_2)$ has first-order poles $\left\{\hat{\rho}_k,k\in\mathcal{N}_m\right\}$ within $\Gamma_m^+$. The corresponding residue is equal to $r(\cdot,z_2)$ given in \eqref{84}. Then within $\Gamma_m^{'+}$, $r(\cdot,z_2)$ has second-order poles $\left\{\hat{\rho}_l,l\in\mathcal{N}_m\right)$ (whether $l\ne k$ or $l=k$) and the corresponding residue is 0 (equaling to \eqref{85}).

Note that in the case $m=n$, $z_1=z_2$ is the second-order pole of $r(z_1,z_2)$. Taking $z_1\to z_2$, we can calculate the residue of  $r(z_1,z_2)$ at $z_1=z_2$ by 
\begin{equation}\label{86}
	{\rm Res}(r(z_1,z_2),z_2)= b_{\underline{\hat{\mathbf{R}}}}(1/z_2) \frac{\partial b_{\underline{\hat{\mathbf{R}}}}(1/z_2)}{\partial z_2}
\end{equation}
where we use [42, Proposition 4.1.4] in calculation. Then according to \eqref{85}, we can find that the residue of \eqref{86} is 0. Hence, we obtain $I_1=0$ for the case $m=n$.

Finally, gathering the above results, we can conclude that $I_1=0$ whether in the case $m\ne n$ or $m=n$.

In what follows, we evaluate the Integral $I_2$. There are two cases to discuss depending on whether $m\ne n$ or $m=n$. 

We first deal with the case $m\ne n$. Define the integrand of $I_2$ by $\upsilon(z_1,z_2)$, i.e.
\begin{equation}\label{key}
	\upsilon(z_1,z_2) = - \frac{b_{\underline{\hat{\mathbf{R}}}}(1/z_1)b_{\underline{\hat{\mathbf{R}}}}(1/z_2)}{(z_1-z_2)^2}.
\end{equation}

Let $z_2$ be fixed. We find that $\upsilon(z_1,z_2)$ has the first-order poles $\{\hat{\rho}_k,k\in\mathcal{N}_m\}$ at $z_1$ within $\Gamma_m^+$.  The residue is calculated by
\begin{equation}\label{89}
	\begin{aligned}
		\upsilon(\cdot,z_2)=	{\rm Res}(	\upsilon(z_1,z_2),\hat{\rho}_k) 
		= -\frac{\frac{c_K}{N}\hat{\rho}_k^2}{(\hat{\rho}_k-z_2)^2}b_{\underline{\hat{\mathbf{R}}}}(1/z_2)
	\end{aligned}
\end{equation}

Through a simple analysis, it is easy to find that $\{\hat{\rho}_l,l\in\mathcal{N}_n\}$ are the first-order poles of $\upsilon(\cdot,z_2)$ within $\Gamma_n^{'+}$. Hence by using the residue theorem,  we have
\begin{equation}\label{90}
	{\rm Res}\left(\upsilon(\cdot,z_2) ,\hat{\rho}_l\right) = -\left(\frac{c_K}{N}\right)^2\frac{\hat{\rho}_k^2\hat{\rho}_l^2}{(\hat{\rho}_k-\hat{\rho}_l)^2}.
\end{equation}

In addition, in the case $m\ne n$, $z_1=z_2$ is not the residue of $\upsilon(z_1,z_2)$. Therefore, for the case $m\ne n$, we have
\begin{equation}\label{91}
	\begin{aligned}
		I_2 &= \frac{K^2}{N_mN_n}\sum_{k\in\mathcal{N}_m}\sum_{l\in\mathcal{N}_n}{\rm Res}\left(\upsilon(\cdot,z_2) ,\hat{\rho}_l\right)
		\\&=-\frac{1}{N_mN_n}\sum_{k\in\mathcal{N}_m}\sum_{l\in\mathcal{N}_n}\frac{\hat{\rho}_k^2\hat{\rho}_l^2}{(\hat{\rho}_k-\hat{\rho}_l)^2}.
	\end{aligned}
\end{equation}


Next we calculate $I_2$ in the case $m=n$.  Let $z_2$ be fixed and then we find that within contour $\Gamma_m^+$, $\upsilon(z_1,z_2)$ has the first-order poles at $\{\hat{\rho}_k,k\in\mathcal{N}_m\}$ and the corresponding residue is given by $\upsilon(\cdot,z_2)$ (see \eqref{89}). We now need to carefully analyze the poles of $\upsilon(\cdot,z_2)$. A simple analysis shows that $\upsilon(\cdot,z_2)$ has poles at $\{\hat{\rho}_l,l\in\mathcal{N}_m\}$. If $l\ne k$, $\hat{\rho}_l$ are the first-order poles of $\upsilon(\cdot,z_2)$ and the corresponding residue has been computed in \eqref{90}. If $l=k$, we may find that the pole $\hat{\rho}_k$ is of order 3. The residue is calculated by
\begin{equation}\label{92}
	{\rm Res}(\upsilon(\cdot,z_2),{\hat{\rho}}_k) = \frac{c_K}{N}\hat{\rho}_k^2\left(\frac{c_K}{N}\sum_{i=1\ne k}^N\left(\frac{\hat{\rho}_i}{\hat{\rho}_i-\hat{\rho}_k}\right)^2+(1-c_K)\right).
\end{equation}

In addition, we note that $z_1=z_2$ is the second-order pole of $\upsilon(z_1,z_2)$. Taking $z_1\to z_2$, we can calculate the residue by
\begin{equation}\label{93}
	{\rm Res}(\upsilon(z_1,z_2),z_2)= -b_{\underline{\hat{\mathbf{R}}}}(1/z_2) \frac{\partial b_{\underline{\hat{\mathbf{R}}}}(1/z_2)}{\partial z_2}
\end{equation}

Then by using the same analysis of \eqref{86}, we find that the residue of \eqref{93} is 0. 

Finally, gathering the above results, we can calculate $I_2$ in the case $m=n$ by
\begin{equation}\label{94}
	\small\begin{aligned}
		I_2 = \frac{K^2}{N_m^2}\left(\sum\limits_{(k,l)\in\mathcal{N}_m^2,k\ne l}{\rm Res}\left(\upsilon(\cdot,z_2) ,\hat{\rho}_l\right) + \sum_{k\in \mathcal{N}_m}{\rm Res}(\upsilon(\cdot,z_2),{\hat{\rho}}_k)\right).
	\end{aligned}
\end{equation}

Substituting \eqref{90} and \eqref{92} into \eqref{94}, we get 
\begin{equation}\label{95}
	\begin{aligned}
		I_2  
		&=  \frac{1}{N_m^2}\sum_{(k,l)\in\mathcal{N}_m^2,k\ne l}-\frac{\hat{\rho}_k^2\hat{\rho}_l^2}{(\hat{\rho}_k-\hat{\rho}_l)^2} 
		\\& + \frac{K^2}{N_m^2}\sum_{k\in \mathcal{N}_m} \frac{c_K}{N}\hat{\rho}_k^2\left(\frac{c_K}{N}\sum_{i=1\ne k}^N\left(\frac{\hat{\rho}_i}{\hat{\rho}_i-\hat{\rho}_k}\right)^2+(1-c_K)\right).
	\end{aligned}
\end{equation}

Note that for $k\in\mathcal{N}_m$, the following relation holds true
\begin{equation}\label{96}
	\small \sum_{i=1\ne k}^N\left(\frac{\hat{\rho}_i}{\hat{\rho}_i-\hat{\rho}_k}\right)^2 = \sum_{i\in\mathcal{N}_m,i\ne k}\left(\frac{\hat{\rho}_i}{\hat{\rho}_i-\hat{\rho}_k}\right)^2 + \sum_{i\in{\mathcal{N}}^c_m}\left(\frac{\hat{\rho}_i}{\hat{\rho}_i-\hat{\rho}_k}\right)^2 
\end{equation}
where ${\mathcal{N}}^c_m$ represents the set $\{1,2,...,N\}\setminus\mathcal{N}_m$.

Substituting \eqref{96} into \eqref{95}, we get $I_2$ for the case $m=n$:
\begin{equation}\label{97}
	\begin{aligned}
		I_2  = \frac{1}{N_m^2}\sum_{k\in\mathcal{N}_m}\sum_{l\in{\mathcal{N}}^c_m}\frac{\hat{\rho}_k^2\hat{\rho}_l^2}{(\hat{\rho}_k-\hat{\rho}_l)^2}+\frac{K-N}{N_m^2}\sum_{k\in\mathcal{N}_m}\hat{\rho}_k^2.
	\end{aligned}
\end{equation}

Gathering the results in \eqref{91} and \eqref{97}, we obtain (85).
\end{document}